\documentclass[10pt, reqno]{amsart}
\usepackage[utf8x]{inputenc}
\usepackage{amssymb,amscd,amsmath}
\usepackage{amsthm,mathtools,xypic}
\usepackage{amsfonts}
\usepackage{mathrsfs}
\usepackage{enumerate}
\usepackage{multicol}
\usepackage[bookmarks]{hyperref}
\usepackage{bbm}
\usepackage{xspace,color}
\usepackage{verbatim}
\usepackage{todonotes}
\usepackage{xy}
\usepackage[T1]{fontenc}
\usepackage{lmodern}
\usepackage{amsxtra}
\usepackage{nicefrac,mathtools}

\newcommand{\al}{\alpha}

\newcommand{\fac}[1]{{\leavevmode\color{red!70!blue}#1}}

\newcommand{\be}{\beta}

\newcommand{\dsh}{\mbox{-}}
\newcommand{\FH}{Hilbert }
\newcommand{\ilt}{\text{ilt}}
\newcommand{\kernel}{\operatorname{ker}}

\newcommand{\la}{\langle}
\newcommand{\laa}{{}_\Acal\langle}
\newcommand{\lab}{\langle}
\newcommand{\lac}{\langle}

\newcommand{\id}{\operatorname{id}}
\newcommand{\plaua}{{{}_\Acal^\Ucal\langle}}
\newcommand{\plauc}{{\langle}}
\newcommand{\praua}{{\rangle}}
\newcommand{\prauc}{{\rangle_\Ccal^\Ucal}}
\newcommand{\pmu}{{p^{-1}}}
\newcommand{\qmu}{{q^{-1}}}
\newcommand{\ra}{\rangle}
\newcommand{\raa}{{\rangle}}
\newcommand{\rab}{{\rangle_\Bcal}}
\newcommand{\rac}{{\rangle_\Ccal}}
\newcommand{\rmu}{{r^{-1}}}
\newcommand{\smu}{{s^{-1}}}
\newcommand{\supp}{\operatorname{supp}}
\newcommand{\tmu}{{t^{-1}}}

\newcommand{\spn}{\operatorname{span}}
\newcommand{\spncl}{\overline{\spn}\, }

\newcommand{\Acal}{\mathcal{A}}
\newcommand{\Bcal}{\mathcal{B}}
\newcommand{\Ccal}{\mathcal{C}}
\newcommand{\Dcal}{\mathcal{D}}
\newcommand{\Ecal}{\mathcal{E}}
\newcommand{\Fcal}{\mathcal{F}}
\newcommand{\Hcal}{\mathcal{H}}
\newcommand{\Kcal}{\mathcal{K}}

\newcommand{\Ncal}{\mathcal{N}}

\newcommand{\Ucal}{\mathcal{U}}
\newcommand{\Vcal}{\mathcal{V}}
\newcommand{\Wcal}{\mathcal{W}}
\newcommand{\Xcal}{\mathcal{X}}
\newcommand{\Ycal}{\mathcal{Y}}
\newcommand{\Zcal}{\mathcal{Z}}

\newcommand{\Bb}{\mathbb{B}}
\newcommand{\C}{\mathbb{C}}
\newcommand{\fb}{\mathbbm{f}}
\newcommand{\gb}{\mathbbm{g}}

\newcommand{\Kb}{\mathbb{K}}
\newcommand{\Lb}{\mathbb{L}}
\newcommand{\Mb}{\mathbb{M}}
\newcommand{\N}{\mathbb{N}}
\newcommand{\pb}{\mathbbm{p}}
\newcommand{\R}{\mathbb{R}}
\newcommand{\rb}{\mathbbm{r}}
\newcommand{\tb}{\mathbbm{t}}
\newcommand{\Xb}{\mathbb{X}}
\newcommand{\Yb}{\mathbb{Y}}

\newcommand{\pr}[2]{\ensuremath{\langle {#1},{#2}\rangle}}
\newcommand{\norm}[1]{\ensuremath{\|#1\|}}

\theoremstyle{plain}
\newtheorem{theorem}{Theorem}[section]
\newtheorem{lemma}[theorem]{Lemma}
\newtheorem{corollary}[theorem]{Corollary}
\newtheorem{proposition}[theorem]{Proposition}
\theoremstyle{definition}
\newtheorem{definition}[theorem]{Definition}
\theoremstyle{remark}
\newtheorem{remark}[theorem]{Remark}
\newtheorem{example}[theorem]{Example}

\numberwithin{equation}{section}

\title[Equivalence of Fell bundles over groups]{Equivalence of\\
  Fell bundles over groups}  
\author[F. Abadie]{Fernando Abadie}
\address{Centro de Matemática, Facultad de Ciencias, Universidad de la
  República\\ Igu\'a 4225\\ Montevideo, Uruguay 11400} 
\email{fabadie@cmat.edu.uy}

\author[D. Ferraro]{Dami\'an Ferraro}
\address{Departamento de Matemática y Estadística del Litoral, CENUR
  Litoral Norte, Universidad de la República\\Gral. Rivera 1350\\
  Salto, Uruguay 50000} 
\email{dferraro@unorte.edu.uy}

\date{August 18, 2016}


\makeindex
\begin{document}
\begin{abstract}
  We give a notion of equivalence for Fell bundles over
  groups, not necessarily saturated nor separable, and show that
  equivalent Fell bundles have Morita-Rieffel equivalent
  cross-sectional $C^*$-algebras. Our notion is originated in the
  context of partial actions and their enveloping actions. The
  equivalence between two Fell bundles   
  is implemented by a bundle of Hilbert bimodules with some extra
  structure. Suitable cross-sectional spaces of such a bundle turn 
  out to be imprimitivity bimodules for the cross-sectional
  $C^*$-algebras of the involved Fell bundles. 
  We show that amenability is preserved under this equivalence
  and, by means of a convenient notion of internal tensor product
  between Fell bundles, we show that equivalence of Fell bundles is
  an equivalence relation.
\end{abstract}

\maketitle

\tableofcontents

\section{Introduction}
\par Towards the end of the sixth decade of the past century,
  J.~M.~G.~Fell introduced the notion of Banach *-algebraic bundle
  over a group, to study and extend the Mackey normal subgroup
  analysis (see \cite{FlDr88} and the references therein). In
  particular he defined $C^*$-algebraic bundles, today known as Fell
  bundles. A Fell bundle can be thought of as the abstraction of a
  grading of a $C^*$-algebra over a locally compact group, and its
  cross-sectional $C^*$-algebras generalize crossed products by
  actions, even by twisted actions. The Fell bundles that have
  received more attention are those called \textit{saturated}, in
  which every fiber $B_t$ (over a group element $t$) is naturally a 
  $B_e\dsh B_e\dsh$imprimitivity bimodule ($e$ being the unit of the
  group).
  The main reason for this is probably that
  such bundles were enough to study most of the examples. That
  situation changed with the introduction of partial crossed products
  in the nineties (see \cite{Ex94Circle},\cite{MC95} and \cite{Ex97}), since
  partial crossed products give rise to non-saturated Fell bundles.
\par In the present work we introduce and study a notion of 
equivalence between arbitrary Fell bundles over groups, these
understood as $C^*$-algebraic bundles as in~\cite{FlDr88}. This notion is
already present, implicitly, in \cite{Ab03} and \cite{AbMr09} (see the 
examples \ref{subsection enveloping actions} and
\ref{subsection Morita equivalence of partial actions} below), where
it was used to prove Morita-Rieffel equivalences between 
several crossed products by partial actions or their enveloping
actions (in this paper we use the expression Morita-Rieffel equivalence to mean strong Morita equivalence, as defined by Rieffel). In those works, in general at least one of the involved Fell 
bundles is 
not saturated, because it is the Fell bundle associated to
a partial action, which is saturated if and only if the action is
global. Still, 
both the reduced and the universal cross-sectional algebras of these Fell
bundles are shown to be Morita-Rieffel equivalent, due to a kind of Morita
equivalence between 
the bundles themselves. It is precisely this notion of 
equivalence between Fell bundles that we study in this article.   
\smallskip 
\par Many authors have studied equivalence of Fell bundles
over \textit{groupoids} (see \cite{muhly2001bundles, muhly2008equivalence} and
the references therein). In this setting the bundles considered 
are such that each fiber $B_\gamma$ over any element $\gamma$ of the
groupoid is an 
imprimitivity bimodule that establishes a Morita-Rieffel equivalence
between 
the fibers over $s(\gamma)$ and $r(\gamma)$, the source and range of
$\gamma$ respectively. In case the groupoid is actually a group, this
means that the bundle is saturated, which is equivalent to the fact
that, for all elements $r$, $s$ in the 
group, the linear span of $B_rB_s$ is dense in $B_{rs}$. Therefore
this theory does not apply to non-saturated Fell bundles over groups, for instance those associated to partial actions.   
\par An equivalence between groupoid based Fell bundles
is possible even when the base groupoids, say $G$ and $H$, are not
isomorphic groupoids, but there exists a $(G\dsh H)\dsh$equivalence
\cite[Definition 2.1]{muhly1987equivalence} between them. 
In case $G$ and $H$ are groups, the 
existence of a $(G\dsh H)\dsh$equivalence implies that $G$ and $H$ are
isomorphic as topological groups. Thus, when dealing with 
equivalence of Fell bundles over groups, we may suppose that the bundles
have the same group as a base space. 
\par Another feature of Fell bundles over groupoids (in the sense of
\cite{muhly2008equivalence}) is that the norm is supposed to be only
upper semicontinuous instead of continuous, as must be the case for
$C^*$-algebraic bundles \cite{FlDr88}.  Once again, this difference is
only apparent when the base space is a group, because then the norm is
automatically continuous (see \cite[Lemma~3.30]{bmz}).
%In fact, assume
%$\Bcal=(B_t)_{t\in G}$ is a Fell bundle over $G$ in the sense of
%\cite[Definition 1.1]{muhly2008equivalence} (except that we do not
%require separability). For every converging net of $\Bcal,$ say
%$b_i\to b,$ we have $b_i^*b_i-b^*b\to 0$ in the fiber over the unit.
%Then $\|b_i^*b_i-b^*b\|\to 0$ (because $\Bcal$ is upper
%semicontinuous) and this implies $\|b_i\|\to \|b\|$ because $ |
%\|b_i\|^2- \|b\|^2|\leq  \|b_i^*b_i-b^*b\|.$ 
\par On the other hand the notion of Fell
bundle over a groupoid includes some separability conditions. The
result is that, when the base space 
is actually a group, the usual notion of Fell bundle over a groupoid
ammounts to a separable and saturated $C^*$-algebraic bundle over a second
countable group. Our aim is to remove all these restrictions,
specially that of saturation, and even so to develop a useful notion
of equivalence between Fell bundles.     
\medskip
\par The organization of our exposition is the following. 
\par The next section is devoted
to the introduction of equivalence 
bundles between Fell bundles, 
which can be thought of as the abstraction of a grading of 
an imprimitivity bimodule over a
group. At the beginning we fix notations and 
review some aspects of the theory of Fell bundles, as well as some
functors associated to them, specially the functors $C^*$ and $C^*_r$,
which to every Fell bundle associate its full (also called universal)
and reduced cross-sectional $C^*$-algebras respectively. Then we
present a couple of examples from 
\cite{Ab03} which have guided us, not only to the definition of Hilbert
$\Bcal$-bundles and equivalence bundles, but also to the proof that
equivalent Fell bundles have Morita-Rieffel equivalent cross-sectional
$C^*$-algebras. After the proof of an  
important technical result, Lemma~\ref{lemma approximate units}, we
define morphisms between equivalence bundles, thus obtaining a
category, and we prove that there exist two special functors from this
category into the category of Fell bundles.            
\par The aim of the third section is to pave the way to the proof
  that equivalent Fell bundles give rise to Morita-Rieffel equivalent
  cross-sectional $C^*$-algebras, to be accomplished in
  Section~4. % antes: present the object we will 
  % use to prove that the equivalence between two Fell bundles implies
  % the Morita-Rieffel equivalence of their associated cross-sectional
  % $C^*$-algebras. 
Namely, we proceed to %antes: The main step is 
the construction of the ``linking Fell bundle''
associated to an equivalence bundle. In fact, as it is shown later in
the same section, 
it is possible to start just with a right 
Hilbert bundle $\Xcal$ over the Fell bundle $\Bcal$, because it is
automatically a $\Kb(\Xcal)-\Bcal\dsh$equivalence bundle, where
$\Kb(\Xcal)$ is a graded version of the usual $C^*$-algebra of
generalized compact operators of a
Hilbert module. The assignment of a linking Fell bundle to each 
equivalence bundle is itself a functor, which  
later will be shown, in Section 4, to commute with the functor $C^*$, 
suitably extended from the category of Fell bundles to the category
of equivalence bundles.  
Essentially, what happens is the following. If $\Xcal$ is an
equivalence bundle between the Fell bundles $\Acal$ and $\Bcal$, 
then there exists a $C^*(\Acal)\dsh C^*(\Bcal)$ imprimitivity bimodule
$C^*(\Xcal)$, which is a certain completion of $C_c(\Xcal)$, such that
if $\mathbb{L}(\Xcal)$ is the linking Fell bundle of $\Xcal$, then
$C^*(\mathbb{L}(\Xcal))=\mathbb{L}(C^*(\Xcal))$ (see
Theorem~\ref{theorem construction of the equivalence bimodule} and  
Corollary~\ref{corollary linking algebra of cross-sectional
  equivalence module}).    
\par A more general situation can be considered -and this is done at
the end of Section~4- in which the functor $C^*$ is replaced by other
type of functors from the category of Fell bundles to the category of 
$C^*$-algebras, for instance the functor $C^*_r$. The functors
we consider are a generalization of the crossed product functors
considered in \cite{buss2015exotic}. 
\par The last section of the article is quite technical. Its main
objective is to prove that
the equivalence of Fell bundles is transitive, thus an equivalence
relation. To this end internal tensor products of Hilbert bundles are
defined, and it is shown that their cross-sectional algebras are
isomorphic to the internal tensor product of the corresponding 
cross-sectional algebras of the Hilbert bundles.

%\section{Examples, definitions and fundamental properties} 
\section{Equivalence bundles.}  

\subsection{Some preliminaries and notations.}\label{sub section
  notation} 
\par In this paper we will be dealing with Fell bundles over
groups. Along the work $G$ will always denote a fixed locally compact
and Hausdorff group with unit element $e.$
We also fix a left invariant Haar measure, $dt,$ of $G$ and denote
$\Delta$ the modular function of $G.$
We understand by a Fell bundle a $C^*$-algebraic
bundle in the sense of Fell \cite{FlDr88}. %\todo{recall the concept}
If $\Acal$ and $\Bcal$ are
Fell bundles over $G$, a \textit{morphism} $\pi\colon \Acal\to \Bcal$
is a continuous, multiplicative and $*$-preserving map, such that for
each $t\in G,$ $\pi(A_t)\subset B_t$ and $\pi|_{A_t}$ is linear. Fell
bundles with these morphisms form a category $\mathscr{F}$. 
%\todo{every fiber is a $C^*$-tring; introduce this category} 
Note that every fiber of a Fell bundle is a $C^*$-ternary ring (in the
sense of \cite{Zl83}) with
the product $(a,b,c)\mapsto ab^*c$, and the restriction of a morphism
of Fell bundles to a fiber is a homomorphism of $C^*$-ternary rings  
(see \cite{abadie2016applications,Zl83}). See below for some
  additional information about $C^*$-ternary rings and their
  homomorphisms.
\par There are several functors of interest to us defined on the
category $\mathscr{F}$. It is the purpose of this section to review
some of them, and to introduce some notation along the
way. The reader is referred to \cite{Ab03} for more details.     
\par Recall that if $\Bcal$ is a Fell bundle, then $C_c(\Bcal)$ is a
$*$-algebra. Besides, given a morphism $\phi:\Acal\to\Bcal$ of Fell
bundles, we have a homomorphism of $*$-algebras: 
$\phi_c:C_c(\Acal)\to C_c(\Bcal)$, $\phi_c(f)=\phi\circ f.$
This functor ($\phi\mapsto \phi_c$) extends to a functor from Fell bundles into Banach $*$-algebras:
$(\Acal\stackrel{\phi}{\to}\Bcal)\longmapsto
\big(L^1(\Acal)\stackrel{\phi_1}{\to}L^1(\Bcal)\big)$, where 
$L^1(\Acal)$ and $L^1(\Bcal)$ are the Banach $*$-algebras obtained by
completing respectively the $*$-algebras $C_c(\Acal)$ and $C_c(\Bcal)$
with respect to the norm $\|f\|_1=\int_G\|f(t)\|dt$, and $\phi_1$ is
the continuous extension of $\phi_c$. Composing the latter functor
with the functor from the category of Banach $*$-algebras into the
category of $C^*$-algebras which consists of taking the enveloping
$C^*$-algebra, we obtain another functor
$C^*:\mathscr{F}\to\mathscr{C}$: $(\Acal\stackrel{\phi}{\to}\Bcal)\longmapsto 
\big(C^*(\Acal)\stackrel{\phi^*}{\to}C^*(\Bcal)\big)$,
from the category of Fell bundles into the category $\mathscr{C}$ of
$C^*$-algebras (with their usual homomorphisms). We call 
$C^*(\Acal)$ the \textit{universal or full cross-sectional algebra} of $\Acal$.  

\par The universal $C^*$-algebra of a Fell bundle has the property
that its non\-degenerate representations are in a bijective 
correspondence with the nonde\-gene\-rate representations of the bundle
\cite[VIII 12.8]{FlDr88}.

\par Given a Fell bundle $\Acal$, let $L^2(\Acal)$ be the right
Hilbert $A_e$-module obtained by completing $C_c(\Acal)$ with respect
to the inner product $\pr{f}{g}:=\int_Gf(t)^*g(t)dt.$
If $\Acal\stackrel{\phi}{\to}\Bcal$ is a morphism of Fell
bundles, we have a map $\phi_2:L^2(\Acal)\to L^2(\Bcal)$, which is the
continuous extension of $\phi_c$ (note
$\pr{\phi_c(f)}{\phi_c(g)}_{L^2(\Bcal)}=\phi(\pr{f}{g}_{L^2(\Acal)})$,
$\forall f,g\in C_c(\Acal)$). The map $\phi_2$ is a morphism of
$C^*$-ternary rings (see \cite{Ab03},
\cite{abadie2016applications}). Thus
$(\Acal\stackrel{\phi}{\to}\Bcal)\longmapsto   
\big(L^2(\Acal)\stackrel{\phi_2}{\longrightarrow}L^2(\Bcal)\big)$ is a
functor from the category of Fell bundles into the category of
$C^*$-ternary rings.  
\par Among the representations of a Fell bundle $\Acal=(A_t)_{t\in G}$
there is the  
so called (left) regular representation, which we briefly recall
now. If $a_t\in A_t$ and $\xi\in C_c(\Acal)$, we define
$\Lambda_{a_t}\xi\in C_c(\Bcal)$ by $\Lambda_{a_t}\xi(s):=a_t\xi(t^{-1}s)$. Then
$\Lambda_{a_t}$ extends 
to an adjointable map $\Lambda_{a_t}\in\Bb(L^2(\Acal))$, where
$\Bb(L^2(\Acal))$ denotes the $C^*$-algebra of adjointable maps of
$L^2(\Acal)$. Besides, the map
$\Lambda:\Acal\to\Bb(L^2(\Acal))$
is a representation of the bundle $\Acal$, called the regular
representation of $\Acal$. The integrated form of $\Lambda$, that is,
its associated representation
$\Lambda^{\Acal}:C^*(\Acal)\to\Bb(L^2(\Acal))$ is 
also called regular representation, and its image $C^*_r(\Acal)$ is
called the \textit{reduced cross-sectional algebra} of
$\Acal$. We say that 
$\Acal$ is amenable when $\Lambda^{\Acal}:C^*(\Acal) \to C^*_r(\Acal)$ is an
isomorphism. Since $\Lambda^{\Acal}$ is injective on $C_c(\Acal)$, we
will consider $C^*_r(\Acal)$ as a completion of $C_c(\Acal)$. 
\par If $\phi:\Acal\to\Bcal$ is a morphism, it is easily
checked that, for all $x\in C^*(\Acal)$, $\phi_2$ intertwines
$\Lambda^{\Bcal}_{\phi^*(x)}$ and $\Lambda_x^{\Acal}$: 
$\Lambda^{\Bcal}_{\phi^*(x)}\phi_2=\phi_2\Lambda_x^{\Acal}$.  
%Therefore $\phi_2(L^2(\Acal))$ is invariant under
%$\Lambda^{\Bcal}\phi^*(C^*(\Acal))$. 
In case $\phi$ is surjective, both $\phi_2$ and $\phi^*$ are also
surjective, because they have closed and dense ranges, since
$\phi_c(C_c(\Acal))$ is uniformly dense in $C_c(\Bcal)$ by
\cite[14.1]{FlDr88}. So in this case 
$\phi^*(\ker\Lambda^{\Acal})\subseteq\ker\Lambda^{\Bcal}$, and          
$\phi^*$ induces a unique homomorphism $\phi_r:C^*_r(\Acal)\to
C^*_r(\Bcal)$ such that the following diagram commutes:  
\begin{align}
\xymatrix{C^*(\Acal)\ar[d]_{\phi^*}\ar[r]^{\Lambda_{\Acal}}&C^*_r(\Acal)\ar[d]^{\phi_r}\\ 
C^*(\Bcal)\ar[r]_{\Lambda_{\Bcal}}&C^*_r(\Bcal)
}\end{align}
\par In case $\phi$ is not surjective, let $\Ccal$ be its image. Then
$\Ccal$ is a Fell subbundle of $\Bcal$, and $\psi:\Acal\to\Ccal$
such that $\psi(a):=\phi(a)$, $\forall a\in \Acal$, is a surjective
morphism, and $\phi=\iota\psi$, where $\iota:\Ccal\to\Bcal$ is the
natural inclusion. By \cite[Proposition~3.2]{Ab03}, $C^*_r(\Ccal)$ is
isomorphic to the closure of $C_c(\Ccal)$ in $C^*_r(\Bcal)$, so we
have an injective homomorphism $\iota_r:C^*_r(\Ccal)\to C^*_r(\Bcal)$
such that $\iota_r\Lambda_{\Ccal}=\Lambda_{\Bcal}\iota^*$. Then the
map $\phi_r:=\iota_r\psi_r: 
C^*_r(\Acal)\to C^*_r(\Bcal)$ also makes commutative the diagram
above, so again
$\phi^*(\ker\Lambda^{\Acal})\subseteq\ker\Lambda^{\Bcal}$. Therefore
we have another functor 
$C^*_r:\mathscr{F}\to\mathscr{C}$, such that 
$(\Acal\stackrel{\phi}{\to}\Bcal)\longmapsto  
\big(C^*_r(\Acal)\stackrel{\phi_r}{\longrightarrow}C^*_r(\Bcal)\big)$,
and $\Lambda:C^*\to C^*_r$ is a natural transformation. 
\medskip
\par A particular type of Fell bundle, which is a guiding example for
us, is that associated to a partial 
action $\alpha=(\{D_t\}_{t\in G},\{\alpha_t\}_{t\in G})$ of a group
$G$ on a $C^*$-algebra $A$. This bundle, denoted $\Bcal_\alpha$, has
total space $\{(t,a):\,a\in D_t\}\subseteq G\times A$ with the
topology inherited from the product topology on $G\times A$. We set
$\norm{(t,a)}:=\norm{a}$. The operations on 
$\Bcal_\alpha$ are defined as follows:
$(t,a)^*:=(t^{-1},\alpha_{t^{-1}}(a^*))$, and
$(r,a)(s,b):=(rs,\alpha_{r}(\alpha_r^{-1}(a)b))$.  
\medskip
\par If $X$ is an $A-B$ imprimitivity bimodule (see
  \cite{Raeburn1998morita}) between the $C^*$algebras $A$ and $B$, we
  will say that $A$ and $B$ are Morita-Rieffel equivalent (rather than
  strong Morita equivalent, the original terminology used by Rieffel),
  and $X$ will be referred just as an equivalence bimodule between $A$
  and $B$.
\medskip
\par When convenient, Hilbert modules will be regarded as ternary
  $C^*$-rings 
($C^*$-trings for short) and often we will work with homomorphisms of
$C^*$-trings as in \cite{Ab03,abadie2016applications}, where the reader
is referred to for more information. Here we just recall the basic
definitions and properties. A $C^*$-tring is a Banach space $X$ with a
ternary product $X\times X\times X\to X$ that is linear in the odd
variables and conjugate linear in the second variable, satisfies a 
certain associativity property, and moroever
$\|(x,y,z)\|\leq\|x\|\,\|y\|\,\|z\|$ and $\|(x,x,x)\|=\|x\|^3$,
$\forall x,y,z\in X$. For instance a right Hilbert $B$-module $X$ can
be seen as a $C^*$-tring with the product $(x,y,z):=x\langle
y,z\rangle$. As shown in \cite{Zl83}, this is almost the general
case: for every $C^*$-tring $X$ there exist a $C^*$-algebra $X^r$ and
a map $\langle\,,\,\!\rangle_r: X\times X\to X^r$ such that $X$ is a
full right Hilbert $X^r$-module (except that $\langle x,x\rangle_r$ does not
need to be a positive element of $X^r$) that satisfies:
$(x,y,z)=x\langle y,z\rangle_r$, $\forall x,y,z\in X$. Besides the pair
$(\langle\,,\,\!\rangle_r,X^r)$ is unique up to canonical
isomorphisms. There also exists an essentially unique pair
$(\langle\,,\,\!\rangle_l,X^l)$ with similar properties, but now
$X$ is a left $X^l$-module, and $(x,y,z)=\langle x,y\rangle_lz$,
$\forall x,y,z\in X$. A homomorphism $\pi:X\to Y$ of $C^*$-trings is a
linear map that preserves the ternary products. Such a homomorphism
$\pi$ induces homomorphisms of $C^*$-algebras $\pi^l:X^l\to Y^l$ and
$\pi^r:X^r\to Y^r$ determined by the properties $\pi^l(\langle
x,y\rangle_l)\pi(z)=(\pi(x),\pi(y),\pi(z))=\pi(y)\pi^r(\langle
y,z\rangle_l)$, $\forall x,y,z\in X$. In this way we have two functors
from the category of $C^*$-trings $\mathscr{T}$ into the category of
$C^*$-algebras $\mathscr{C}$: the left functor
$(X\stackrel{\pi}{\to}Y)\longmapsto (X^l\stackrel{\pi^l}{\to}Y^l)$, 
and the right functor: $(X\stackrel{\pi}{\to}Y)\longmapsto (X^r\stackrel{\pi^r}{\to}Y^r)$. We specially want to call the
attention of the reader to the following facts: 
every homomorphism of $C^*$-trings is contractive; for such maps, being
injective is equivalent to being an isometry and, finally, the image
of a $C^*$-tring homomorphism is closed, and so a $C^*$-tring itself.  

% ANTES de DFC: 
% \par For the rest of this work $G$ will be a locally compact and
% Hausdorff group and $e$ will denote its unit. 
% We also fix a left invariant Haar measure, $dt,$ of $G$ and denote
% $\Delta$ the modular function of $G.$ 

\par Finally we fix some more notation. Letters $\Hcal$ and $\Kcal$
will be used to denote Hilbert spaces,  
while (right) Hilbert modules \cite{Raeburn1998morita} will be denoted
$X,$ or $X_A$ if it is necessary to indicate the involved algebra $A$.   
On the set of adjointable operators from $X_A$ to $Y_A,$
$\Bb(X_A,Y_A),$ we consider the operator norm and we regard this set
as a ternary $C^*$-ring \cite{Zl83} with the operation $(R,S,T)\mapsto
RS^*T.$ We denote by $\Kb(X_A,Y_A)$ the set of generalized compact
operators, which is a Banach subspace of $\Bb(X_A,Y_A)$ (in fact an
ideal of $\Bb(X_A,Y_A)$ in the sense of
\cite{Ab03,abadie2016applications}).  

\par The letters $\Acal,\Bcal,\Ccal$ and $\Dcal$ will be used to
represent Fell bundles over $G$ and $\Xcal,\Ycal,\Zcal$ to denote
Banach bundles over $G$ (in the sense of \cite{FlDr88}). 
%that will implement the Morita equivalence between Fell bundles. 
The fiber over $t$ of $\Acal,\Bcal,\ldots,\Xcal,\ldots$ will be
denoted $A_t,B_t,\ldots,X_t,\ldots$ respectively. The space of
compactly supported continuous sections of the Banach bundle $\Xcal$
will be denoted $C_c(\Xcal)$.

\subsection{Motivating examples}
\par In \cite{Ab03} (see also \cite{AbMr09}) it was shown that Morita
equivalent partial actions give rise to Morita-Rieffel equivalent
crossed products, and also that the crossed product of a partial
action is Morita-Rieffel equivalent to the crossed product by its
enveloping action (and even to the crossed product by its Morita
enveloping action). These results were obtained as particular cases of
more general results about cross-sectional algebras of Fell
bundles. The involved Fell bundles are related by a kind of
equivalence bundles, whose properties inspired our
Definition~\ref{defi equivalence bundle}, and which we briefly review
below. 

\subsubsection{Enveloping actions}\label{subsection enveloping actions}
  Let $\beta$ be a continuous action by automorphisms of $G$ on the
  $C^*$-algebra $B$ and $A$ an ideal of $B$, with $B=\spncl
  \{\be_t(A)\colon t\in G\}.$ 
  Now let $\al$ be the partial action obtained as the restriction of
  $\be$ to $A$ \cite{Ab03}, so that $\be$ is an enveloping action for
  $\al.$ 
  
  We think of $\Bcal_\al$ and $\Xcal:=A\times G$ as Banach
  subbundles 
  of $\Bcal_\be.$ 
  Since $\Bcal_\al \Xcal\subset \Xcal,$ $\Xcal\Bcal_\be\subset \Xcal$ and $\Xcal\Xcal^*\subset \Bcal_\al,$ we have the operations
  \begin{align}
     \Bcal_\al \times \Xcal\to \Xcal\ (a,x)\mapsto ax,\ {}_{\Bcal_\al}\la \ ,\ \ra\colon \Xcal\times \Xcal\to \Bcal_\al \ (x,y)\mapsto xy^*,\label{equ left bimodule operations enveloping algebras}\\
     \Xcal\times \Bcal_\be\to \Xcal\ (x,b)\mapsto xb,\mbox{ and } \la \ ,\ \ra_{\Bcal_\be}\colon \Xcal\times \Xcal\to \Bcal_\be \ (x,y)\mapsto x^*y.\label{equ right bimodule operations enveloping algebras}
  \end{align}
  \indent Note that, if $X_r:=\{(a,r):\,a\in A\}$, and $A_r$, $B_r$ are respectively the fibers
 of $\Bcal_\alpha$ and $\Bcal_\beta$ at $r$, then we have
 that ${}_{\Bcal_\al}\la X_r,X_s \ra\subseteq A_{rs^{-1}}$, $\la
 X_r,X_s \ra_{\Bcal_\beta}\subseteq B_{r^{-1}s}$,  
 $A_rX_s\subseteq X_{rs}$, $X_rB_s\subseteq X_{rs}$,  
 $A_t=\overline{\textrm{span}}\{ {}_{\Bcal_\al}\la X_{ts},X_s \ra:\,
 s\in G\}$, and $B_t=\overline{\textrm{span}}\{\la X_{r},X_{rt} \ra_{\Bcal_\beta}:\,
 s\in G\}$, $\forall t\in G $. These are precisely the properties that
inspire the definition of equivalence between Fell bundles in
\ref{defi equivalence bundle}.      
 \par It is shown in \cite{Ab03} that $C^*_r(\Bcal_\al)$ is a full hereditary $C^*$-subalgebra of $C^*_r(\Bcal_\be).$
  Moreover, the canonical equivalence bimodule is the closure of
  $C_c(\Xcal)$ within $C^*_r(\Bcal_\be).$ 
  In terms of the operations described in \eqref{equ left bimodule operations enveloping algebras} and \eqref{equ right bimodule operations enveloping algebras}, the bimodule structure is given by
  \begin{align}
    uf(t) & = \int_G u(tr)f(\rmu)\, dr\label{equ left action}\\
    fv(t) & =\int_G f(r)v(\rmu t)\, dr\label{equ right action}\\
    {}_{C_c(\Bcal_\alpha)}\la f,g\ra(t) &=\int_G {}_{\Bcal_\alpha}\la f(tr),g(r)\ra\Delta(r)\, dr\label{equ left inner product}\\
    \la f,g\ra_{C_c(\Bcal_\beta)}(t) &=\int_G \la f(r),g(rt)\ra_{\Bcal_\beta}\, dr\label{equ rigt inner product}
  \end{align}
  for $u\in C_c(\Bcal_\al),$ $f,g\in C_c(\Xcal)$ and $v\in C_c(\Bcal_\be).$
Note that $\Bcal_\beta$ is a saturated Fell bundle (i.e.:
$\textrm{span}B_rB_s$ is dense in $B_{rs}$, $\forall r,s\in G$), while 
$\Bcal_\alpha$ is saturated only if $\alpha$ is a global action, that is: $A=B$ and $\alpha=\beta$.   
\subsubsection{Morita equivalence of partial
  actions}\label{subsection 
  Morita equivalence of partial actions} 
\par Assume now that $\alpha$ and $\beta$ are Morita equivalent partial actions of $G$
on the $C^*$-algebras $A$ and $B,$ respectively. 
  Thus (using the notation of \cite{Ab03}) there exists a
  partial action $\gamma=\left(\{X_t\}_{t\in G},\{\gamma_t\}_{t\in
      G}\right)$ of $G$ on the $A\dsh B\dsh$equivalence bimodule $X$
  with $\alpha=\gamma^l$ and $\beta=\gamma^r.$ If $\Lb(\gamma)$ is the
  linking partial action of $\gamma,$ then 
  $\Bcal_\alpha$ and $\Bcal_\beta$ are Fell subbundles of
  $\Bcal_{\Lb(\gamma)}$ and we may think of $\Bcal_{\Lb(\gamma)}$ as a
  Banach subbundle of $G\times \Lb(X).$ 
  Recall that in Exel's notation the element $(t,S)\in G\times \Lb(X)$
  is represented by $S\delta_t.$ 
  Consider
  \begin{equation}\label{equation bundle of a pa on a Hilber module}
  \Xcal_\gamma :=\left\{ \left(\begin{array}{cc} 0 & x\\ 0 & 0\end{array}
  \right)\delta_t\colon x\in X_t,\ t\in G \right\}
  \end{equation}
  Then $\Xcal_\gamma$ is a Banach subbundle of $\Bcal_{\Lb(\gamma)}.$ 

  In this situation it is possible to define the following operations: 
  %on $\Bcal_{\Lb(\gamma)}:$ 
  \begin{align}
    \Bcal_\al \times \Xcal_\gamma\to \Xcal_\gamma\colon\ (a,x)\mapsto ax,\
    {}_{\Bcal_\al}\la \ ,\ \ra\colon \Xcal_\gamma\times
    \Xcal_\gamma\to \Bcal_\al\colon \ (x,y)\mapsto xy^*,\label{equation
      inspiration left \FH bundle}\\ 
    \Xcal_\gamma\times \Bcal_\be\to \Xcal_\gamma\colon\ (x,b)\mapsto
    xb,\mbox{ and } \la \ ,\ \ra_{\Bcal_\be}\colon \Xcal_\gamma\times
    \Xcal_\gamma\to \Bcal_\be\colon \ (x,y)\mapsto x^*y.\label{equation
      inspiration right \FH bundle} 
  \end{align}

  Using the product and involution of $C_c(\Bcal_{\Lb(\gamma)})$ we
  can make $C_c(\Xcal_\gamma)$ into a pre $C_c(\Bcal_\alpha)\dsh
  C_c(\Bcal_\beta)$ Hilbert bimodule. 
  More precisely, the operations are given by 
  $$ uf:=u*f,\ fv:=f*v,\ {}_{C_c(\Bcal_\al)}\la f,g\ra : = f*g^*\mbox{
    and } \la f,g\ra_{C_c(\Bcal_\al)} : =f^**g,$$ 
  for $u\in C_c(\Bcal_\al),$ $v\in C_c(\Bcal_\beta)$ and $f,g\in
  C_c(\Xcal_\gamma).$ 
  It is easy to check that the operations above correspond to the 
  expressions (\ref{equ left action} - \ref{equ rigt inner product}). 
%we know these inner products are positive in the (full) cross
%sectional $C^*$-algebras. 
 \par  The techniques of \cite[Section 3]{Ab03} can be used to show that
  the completion of $C_c(\Xcal_\gamma)$ in $C^*(\Bcal_{\Lb(\gamma)}),$
  $C^*(\Xcal_\gamma),$ is a $C^*(\Bcal_\alpha)\dsh
  C^*(\Bcal_\beta)$ Hilbert module. 
  Moreover, $C^*(\Bcal_{\Lb(\gamma)})$ is isomorphic to the linking
  algebra of $C^*(\Xcal_\gamma)$ (see \cite[Theorem 1.1]{AbMr09} and
  \cite[Corollary 5.3]{abadie2016applications} for details)  
  This will turn out to be a particular case of a general situation (Theorem \ref{theorem construction of the equivalence bimodule}).
 
\subsection{Equivalence bundles}

A Fell bundle is an abstraction of a grading of a $C^*$-algebra over a
group. In the same way, a Hilbert bundle over a Fell bundle, as we
define below, can be thought of as an abstraction of a graded Hilbert
module.   
\par Let $\Acal$ and $\Bcal$ be Fell bundles over $G.$
Considering our motivating examples and the discussion in the
Introduction we state the following. 

\begin{definition}\label{definition right \FH}
  A \textit{right} \textit{\FH  $\Bcal\dsh$bundle} is a complex Banach bundle over $G,$ $\Xcal:=\{X_t\}_{t\in G},$ with continuous maps 
  $$\Xcal\times \Bcal \to \Xcal,\ (x,b)\mapsto xb,\mbox{ and }\lab \ ,\ \rab\colon \Xcal\times \Xcal\to \Bcal,\ (x,y)\mapsto \lab x,y\rab$$
  such that:
  \begin{enumerate}[(1R)]
   \item $X_r B_s\subset X_{rs}$ and $\lab X_r,X_s\rab \subset B_{\rmu
       s},$ for all $r,s\in G.$ 
   \item $X_r\times B_s\to X_{rs},\ (x,b)\mapsto xb,$ is bilinear for
     all $r,s\in G.$ 
   \item $X_s\to B_{\rmu s},\ y\mapsto \lab x,y\rab,$ is linear for
     all $x\in X_r$ and $s\in G.$ 
   \item $\lab x,yb\rab =\lab x,y\rab b$ and $\lab x,y\rab^*=\lab
     y,x\rab$ for all $x,y\in \Xcal$ and $b\in \Bcal.$ 
   \item $\lab x,x\rab\geq 0,$ for all $x\in \Xcal,$ and $\lab
     x,x\rab=0$ implies $x=0.$ 
   Besides, each fiber $X_t$ is complete with respect to the
   norm %\footnote{Note $X_t$ has a natural right $B_e\dsh$pre Hilbert
       %module with the inner product $(x,y)\mapsto \lab x,y\rab.$} 
   $x\mapsto \|\lab x,x\rab\|^{1/2}.$
   \item For all $x\in \Xcal,$ $\|x\|^2=\|\lab x,x\rab\|.$
   \item\label{right density} $\spncl\{ \lab X_s,X_s\rab \colon s\in G\}=B_e.$
  \end{enumerate}
\end{definition}

Condition (6R) expresses the compatibility of the norm of the fibers
of $\Xcal$ with the norm given by considering each one of the fibers
as a right pre Hilbert $B_e$-module (with the action and inner
product specified in the definition). 
%We have separated (6R) from (5R) because, when constructing a \FH
%bundle, it is not a condition to be verified but a condition to be
%used to define the norm of the bundle.
\par A word must be said about condition (\ref{right density}R). As
mentioned before, we may think of a Hilbert bundle as a grading of
a Hilbert module. In this sense, conditions (1R)--(6R) would be
enough to define a Hilbert $\Bcal$-bundle. However, along
this work we are mainly interested in obtaining equivalence bimodules
as completions of cross-sectional spaces of Hilbert bundles, so we
consider just Hilbert bundles analogous to $\textit{full}$ Hilbert
modules, that is, Hilbert bundles satisfying (7R). On the other hand,
it is this crucial property that will allow to have, for instance,  
an equivalence between a saturated Fell bundle with a non-saturated
one. Note that we are not requiring, unlike the usual notion in the
Fell bundles over groupoids context, that 
$\spncl \lab X_s,X_s\rab=B_e$, $\forall s\in G$, but only that the sum
of \textit{all} these spaces is dense in $B_e$. Note that this also
implies $B_t=\spncl\{\lab X_r,X_{rt}\rab\colon r\in G\}$, $\forall
t\in G$ (a proof of this fact will be provided in (6) of
Lemma~\ref{lemma first tools}).          

\textit{Left \FH bundles} are defined similarly: properties
(1R)-(7R) are changed to:  
  \begin{enumerate}[(1L)]
   \item $A_r X_s\subset X_{rs}$ and $\laa X_r,X_s\raa \subset
     A_{r\smu},$ for all $r,s\in G.$ 
   \item $A_r\times X_s\to X_{rs},\ (a,x)\mapsto ax,$ is bilinear for
     all $r,s\in G.$ 
   \item $X_s\to A_{s\rmu },\ y\mapsto \laa y,x\raa,$ is linear for
     all $x\in X_r$ and $s\in G.$ 
   \item $\laa ax,y\raa =a\laa x,y\raa$ and $\laa x,y\raa^*=\laa
     y,x\raa$ for all $x,y\in \Xcal$ and $a\in \Acal.$ 
   \item $\laa x,x\raa\geq 0,$ for all $x\in \Xcal,$ and $\laa
     x,x\raa=0$ implies $x=0.$ 
   Besides, each fiber $X_t$ is complete with respect to the norm
   $x\mapsto \|\laa x,x\raa\|^{1/2}.$ 
   \item For all $x\in \Xcal,$ $\|x\|^2=\|\laa x,x\raa\|.$
   \item $\spncl\{ \laa X_s,X_s\raa \colon s\in G\}=A_e.$
  \end{enumerate}

\begin{definition}\label{defi equivalence bundle}
  We say that $\Xcal$ is an $\Acal\dsh \Bcal$ \textit{equivalence
    bundle} if it is both a left Hilbert $\Acal\dsh$bundle, a right
  Hilbert $\Bcal\dsh$bundle and $\laa x,y\raa z = x\lab y,z\rab,$ 
  for all $x,y,z\in \Xcal.$ 
  Besides, we say that $\Acal$ is equivalent to $\Bcal$ if there
  exists an $\Acal\dsh \Bcal\dsh$equivalence bundle. 
\end{definition}

\begin{example}\label{subsection enveloping actions 2}
  In the context and with the notation of \ref{subsection enveloping
    actions}, $\Xcal$ is a $\Bcal_\al\dsh \Bcal_\be\dsh$equivalence
  bundle. 
\end{example}

\begin{example}\label{example morita equivalence of pa 2}
  Suppose that $\gamma$ is a partial action on the $A\dsh B$ Hilbert
  bimodule $X$, as in \ref{subsection Morita equivalence of partial
    actions}. 
  Then the bundle $\Xcal_\gamma$ is a $\Bcal_{\gamma^l}\dsh
  \Bcal_{\gamma^r}\dsh$equivalence bundle. 
\end{example}

\begin{example}\label{example reflexive}
  If $\pi\colon \Acal\to \Bcal$ is an isomorphism of Fell bundles, 
  then the bundle $\Xcal=\Bcal$ is an $\Acal\dsh \Bcal\dsh$equivalence
  bundle with the operations
  $$ \Acal\times \Xcal\to \Xcal:\ (a,x)\mapsto \pi(a)x,\ \laa \ ,\ \raa\colon \Xcal\times \Xcal\to \Acal:\ (x,y)\mapsto \pi^{-1}(xy^*),  $$
  $$ \Xcal\times \Bcal\to \Xcal:\ (x,b)\mapsto xb,\mbox{ and } \lab \ ,\ \rab\colon \Xcal\times \Xcal\to \Bcal: \ (x,y)\mapsto x^*y. $$
  Then isomorphic Fell bundles are equivalent and every Fell bundle is
  equivalent to itself. 
\end{example}

There are two questions that can be immediately answered using our
motivating examples: Do equivalent Fell bundles have Morita-Rieffel 
equivalent unit fibers? Can a non-saturated Fell bundle be 
equivalent to a saturated Fell bundle? 

To answer both questions at once consider the action $\be$ of $G=\R$
on $B=C_0(\R)$ given by $\be_t(f(\cdot))=f(\cdot +t)$ and let $A$ be
the C*-ideal of $B$ corresponding to the open set $(0,1)\cup (1,2).$ Since
$\beta$ is the enveloping action of $\alpha$, $\Bcal_\alpha$ and
$\Bcal_\beta$ are equivalent by \ref{subsection enveloping actions}.   
Now the unit fibers of $\Bcal_\al$ and $\Bcal_\be$ are $C_0((0,1)\cup
(1,2))$ and $C_0(\R),$ respectively, which are not Morita-Rieffel 
equivalent because they are commutative and not isomorphic. 
Also note that $\Bcal_\beta$ is saturated but $\Bcal_\alpha$ is not.

An interesting fact to be proved later in
Corollary~\ref{cor:partial-saturated}, which will follow from the main
result of \cite{Ab03} and the transitivity of 
equivalence of Fell bundles, is that every Fell
bundle associated to a partial action is equivalent to the Fell
bundle associated to an action.  

As a preparation for future sections we prove some basic facts about
right \FH bundles. Of course left \FH bundles will have similar 
properties. A way of translating results from left to right and vice 
versa is to consider adjoint bundles, which we introduce next.  

Assume that $\Xcal$ is a left \FH $\Acal\dsh$bundle over $G$, and let
$\widetilde{\Xcal}$ be the retraction of $\Xcal$ by the inversion map of
$G$ (according to \cite[II 13.3]{FlDr88}), except that the product by
scalars is given by $\lambda \tilde{x}=\widetilde{\bar\lambda x}$,
where $\lambda\in\C$ and $\tilde{x}$ is the element $x\in\Xcal$ seen
as an element of $\widetilde{X}$. Thus the fiber of $\widetilde{\Xcal}$ over
$t\in G$ is $\tilde{X}_{t^{-1}}$, where $\tilde{X}_t$ is the
complex-conjugate Banach space of $X_t$. Then the \textit{adjoint} of
$\Xcal$ is the right \FH $\Acal$-bundle $\widetilde{\Xcal}$ where the 
action $\widetilde{\Xcal}\times  \Acal\to  \widetilde{\Xcal}$ is given
by $(\widetilde{x},a)\mapsto \widetilde{a^*x}$, and the inner product 
$\widetilde{\Xcal} \times \widetilde{\Xcal}\to \Acal$ by 
$(\widetilde{x},\widetilde{y})\mapsto \laa x,y\raa.$ 

\begin{remark}\label{remark symmetric}
  A similar construction can be performed on a right \FH bundle to
  obtain a left \FH bundle. In case $\Xcal$ is an $\Acal\dsh
  \Bcal\dsh$equivalence bundle, $\widetilde{\Xcal}$ is a $\Bcal\dsh
  \Acal\dsh$equivalence bundle. Also note that\  \ 
  $\widetilde{\!\!\widetilde{\Xcal}}=\Xcal$.  
\end{remark}

\begin{lemma}\label{lemma first tools}
  Given a right \FH $\Bcal\dsh$bundle $\Xcal,$ for all $x,y\in \Xcal$
  and $b,c\in \Bcal$, the following  relations hold: 
  \begin{enumerate} 
 \item $\lab xb,y\rab= b^*\lab x,y \rab$ 
 \item $\|xb\|\leq \|x\|\|b\|$
 \item $\|\lab x,y\rab\|\leq \|x\|\|y\|$
 \item $(xb)c=x(bc)$ 
 \item $\|b\|=\sup\{ \|zb\|\colon z\in \Xcal, \|z\|\leq 1\}$
 \item $B_t=\spncl\{\lab X_r,X_{rt}\rab\colon r\in G\}$
  \end{enumerate}
  In case $\Xcal$ is an $\Acal\dsh \Bcal\dsh$equivalence bundle, the
  following equalities hold for all $a\in \Acal,$ $x,y\in \Xcal$, and
  $b\in \Bcal:$ 
  \begin{enumerate}  
  \setcounter{enumi}{6}
  \item $(ax)b=a(xb)$
  \item $\laa xb,y\raa = \laa x,yb^*\raa$
  \item $\lab ax,y\rab=\lab x,a^*y\rab. $
  \end{enumerate} 
\end{lemma}
\begin{proof}
  The proof of (1) is left to the reader.
  Meanwhile (2) holds because
  \begin{equation}\label{equation norm bound of action}
   \|xb\|^2=\|\lab xb,xb\rab \| = \|b^*\lab x,x\rab b\|\leq \|x\|^2\|b\|^2.
  \end{equation}
  
  Note that $x\lab x,y\rab$ and $y$ belong to the same fiber of
  $\Xcal,$ say $X_t.$ 
  Regarding $X_t$ as a right Hilbert $B_e\dsh$module we have:   
  $$ \|\lab x,y\rab\|^2=\| \lab x\lab x,y\rab,y\rab \|\leq \|x\lab
  x,y\rab \|\|y\|\leq \|\lab x,y\rab\|\|x\|\|y\|, $$ 
  so (3) follows. %Then $\|\lab x,y\rab\|\leq \|x\|\|y\|.$
  To prove (4) note that, since $(xb)c$ and $x(bc)$ belong to the same
  fiber, say $X_t,$ the 
  identity $(xb)c=x(bc)$ holds if and only if $\lab z,(xb)c\rab= \lab
  z,x(bc)\rab,$ for all $z\in X_t$, and the latter equality holds,
  because 
  \[\lab 
  z,(xb)c\rab = \lab z,xb\rab c= \lab z,x \rab bc = \lab z,x(bc)\rab,
  \forall z\in X_t.\] 
\fac 
 \par To prove (5) set $\tau(b):=\sup\{
  \|xb\|\colon x\in \Xcal,\|x\|\leq 1\}.$ 
  From (2) follows that $\tau(b)\leq
  \|b\|.$ 
  Consider each fiber $X_t$ as a left $\Kb(X_t)\dsh$ Hilbert
  module, and let $B_e^{\textrm{op}}$ be the opposite C*-algebra of
  $B_e$. Then we have a representation $\phi_t\colon
  B_e^{\textrm{op}}\to \Bb_{\Kb(X_t)}(X_t),$ $\phi_t(c)u=uc.$ Note that 
\begin{equation}\label{eqn:truco}
\|\phi_t(c)\|^2
=\sup\{\|{}_{\Kb(X_t)}\langle xc,xc\rangle\|:\| x\|\leq
1\}
=\sup\{\|\langle xc,xc\rangle_{\Bcal}\|:\| x\|\leq
1\}
\end{equation} 
  Since condition (\ref{right density}R) together with $\phi_t(c)=0$ for
  all $t\in G$ 
  implies $B_ec=0$, thus $c=0$,  
  the direct sum $\phi=\oplus_{t\in G}\phi_t$ is injective, and
  therefore isometric.
  This implies $\|b\|=\tau(b)$ because by \eqref{eqn:truco} we have 
  $$\|b\| = \sup_{t\in G}\|\phi_t(b)\| = \sup_{t\in
    G}\sup\{\| xb \|\colon x\in X_t,\ \|x\|\leq 1\}=\tau (b).$$
Finally, if $b$ is any element of $\Bcal$, we have 
$\|b\|^2=\tau(bb^*)\leq\tau(b)\| b\|$, which shows that
$\|b\|\leq\tau(b)$.     
\begin{comment}
  Consider each fiber $X_t$ as a $\Kb(X_t)\dsh B_e$ Hilbert
  bimodule and let $Y_t$ be its adjoint module. 
  Then we have a representation $\phi_t\colon B_e\to \Bb(Y_t),$
  $\phi_t(c)\widetilde{u}=\widetilde{uc^*}.$ 
  Since the condition $\phi_t(c)=0$ for all $t\in G$ implies $cB_e=0,$
  the direct sum $\oplus_{t\in G}\phi_t$ is injective. 
  This implies $\|b\|=\tau$ because
  $$\|b\|^2 = \|bb^*\| = \sup_{t\in G}\|\phi_t(bb^*)\| = \sup_{t\in
    G}\sup\{\| xbb^* \|\colon x\in X_t,\ \|x\|\leq 1\}\leq \tau \|b\|.$$  
\end{comment}
  \par As for (6), from \cite[VIII 16.3]{FlDr88} and (7R) it follows
  that  $$B_t = \spncl \{ 
  \lab X_r, X_r\rab B_t\colon r\in G \}= \spncl \{ \lab X_{r},
  X_{rt}\rab\colon r\in G \}\subset B_t.$$ 
  
  For the rest of the proof we assume that $\Xcal$ is an $\Acal\dsh
  \Bcal\dsh$equivalence bundle. 
  Using the conclusion of the last paragraph together with continuity
  arguments, we see that it suffices to show (7)--(9)
  for elements $a$ and $b$ of the form $a=\laa u,v\raa$ and $b=\lab
  z,w\rab.$  
  The proof finishes after we observe that
  $$a(xb) = \laa u,v\raa (x\lab z,w\rab) =\laa \laa u,v\raa x,z\raa w
  = (\laa u,v\raa x)\lab z,w\rab = (ax)b. $$ 
  $$ \laa xb,y\raa = \laa x\lab z,w\rab,y\raa =  \laa x, z\raa \laa
  w,y\raa = \laa x, y\lab w,\rab \raa = \laa x,yb^*\raa.$$ 
  $$\lab ax,y\rab = \lab u\lab v,x\rab,y\rab =\lab x,v\rab\lab u,y\rab
  %=\lab x,v\lab u,y\rab\rab 
  = \lab x,\laa v,u\raa
  y\rab =   \lab x,a^*y\rab.$$ 
\end{proof}

Approximate units of Fell bundles are a powerful tool. In what follows
we construct a special kind of approximate units, which will prove to
be extremely useful.  
\begin{comment}
The next lemma is the cornerstone of this
article because with it we 
will construct the linking Fell bundle of a equivalence bundle, prove
Morita equivalence of Fell bundles is an equivalence relation and
prove Morita equivalent Fell bundles have Morita equivalent
cross-sectional $C^*$-algebras. \todo{feote esto} 
\end{comment}
\par The expression $\Mb_n(X)$ stands for the $n\times n$ matrices with
entries in the set $X.$ 
The $(i,j)$ entry of $M\in \Mb_n(X)$ will be denoted $M_{i,j}.$

\begin{lemma}[cf. \textit{\cite[Lemma 5.1]{abadie2016applications}}]\label{lemma approximate units}
  Let $\Xcal$ be an $\Acal\dsh \Bcal\dsh$equivalence bundle.
  Then
  \begin{enumerate}
    \item\label{item approximate units} $\Acal$ and $\Bcal$ have
      approximate units, $\{a_i\}_{i\in I}$ and $\{b_j\}_{j\in J},$
      such that for all $i\in I$ and $j\in J$ there exist
      $x_1,\ldots,x_{n_i},y_1,\ldots,y_{n_j}\in \Xcal$ such that   
    $$a_i = \sum_{k=1}^{n_i}\laa x_k,x_k\raa\mbox{ and } b_j
    =\sum_{k=1}^{n_{j}} \lab y_k,y_k\rab.$$ 
    \item For all $\tb=(t_1,\ldots,t_n)\in G^n,$ the set 
    $$\Mb_\tb(\Bcal):=\{ M\in \Mb_n(\Bcal)\colon M_{i,j}\in
    B_{{t_i}^{-1}t_j}\ \forall\ i,j=1,\ldots,n\}$$ 
    is a $C^*$-algebra with entrywise vector space operations, matrix
    multiplication as product and *-transpose
    (${M^*}_{i,j}={M_{j,i}}^*$) as involution. 
    Moreover, its $C^*$-norm is equivalent to the supremum norm 
    $\|M\|_\infty:=\max_{i,j}\|M_{i,j}\|.$ 
    \item For all $t,r_i\in G,$ $x_i\in X_{r_i}$ and $y_i\in X_{r_i
        t}$ ($i=1,\ldots,n$): 
    \begin{equation}\label{equation the hard bound}
      \| \sum_{i=1}^n \lab x_i,y_i\rab \|^2\leq  \|\sum_{i=1}^n \lab x_i,x_i\rab\|\|\sum_{i=1}^n \lab y_i,y_i\rab\|.
    \end{equation}
  \end{enumerate}
\end{lemma}
\begin{proof}
  Let $\Lambda$ be the set $$\{ b\in B_e\colon \|b\|< 1,\ \exists \
  y_1,\ldots,y_n\in \Xcal \mbox{ such that } b=\sum_{j=1}^n \lab
  y_k,y_k\rab \}.$$ 
  
  To show that $\Lambda$ is a directed set take $b_1,b_2\in \Lambda,$ with
  $b_j = \sum_{k=1}^n\lab x^j_k,x^j_k\rab.$ 
  Each fiber $X_t$ is a Hilbert $B_e\dsh$module, so we may think of
  $X_t$ as a Hilbert module over the unitization of $B_e.$ 
  Set
  $$ c':= b_1(1-b_1)^{-1} + b_2(1-b_2)^{-1} = \sum_{j=1}^2\sum_{k=1}^n
  \lab x^j_k(1-b_j)^{-1/2}, x^j_k(1-b_j)^{-1/2}\rab.$$ 
  Then $c'\geq 0$ and, using functional calculus, we see that $c:=
  c'(1+c')^{-1}$ equals 
  $$ \sum_{j=1}^2\sum_{k=1}^n  \lab  x^j_k(1-b_j)^{-1/2}(1+c')^{-1/2},
  x^j_k(1-b_j)^{-1/2}(1+c')^{-1/2}\rab,$$ 
  and belongs to $\Lambda.$
  Moreover it can be shown that $b_1,b_2\leq c$ (see \cite[page
  78]{Mpy90}).  
  
  To show that $\{\lambda\}_{\lambda\in \Lambda}$ is an approximate unit of
  $B_e$ (and so also of $\Bcal$) it suffices to show that $\|
  b-b\lambda \|\to 0,$ for all $b\in B_e^+$ with $\|b\|<1.$ 
  From the proof of Theorem 3.1.1 of \cite{Mpy90} we know
  $\lambda\mapsto \|b-b\lambda\|$ is decreasing, so it suffices to
  show that, given $\varepsilon>0,$ there is $\lambda\in\Lambda$ such
  that $\|b-b\lambda\|<\varepsilon.$ To this end, fix $\varepsilon>0$
  and consider the right Hilbert $B_e\dsh$module obtained as the
  direct sum of all the right $B_e\dsh$Hilbert modules $X_t$,  
  $M:=\oplus_{t\in G}X_t$, which is full by (7R).  
  From \cite[Lemma 7.2]{lance1995hilbert} we know there exist 
  $\xi_1,\ldots,\xi_n\in M$ such that $\|b-b\sum_{k=1}^n\la
  \xi_k,\xi_k\ra\|<\varepsilon$ and $\|\sum_{k=1}^n\la
  \xi_k,\xi_k\ra\|<1.$ 
  Since $\sum_{k=1}^n\la \xi_k,\xi_k\ra$ lies in the closure of
  $\Lambda,$ there exists $\lambda\in \Lambda$ such that $\|b-b\lambda
  \|<\varepsilon.$  
  
  As for the *-algebra structure of $\Mb_\tb(\Bcal),$ note
  that the product and involution are defined, because given $M,N\in
  \Mb_\tb(\Bcal)$ we have $M_{i,k}N_{k,j}\in  B_{ t_i {t_j}^{-1} }$ 
  and ${M_{i,j}}^*\in {B_{ t_i {t_j}^{-1} }}^*=B_{ t_j {t_i}^{-1} }. $  
  The routine algebraic verifications needed to see that $\Mb_\tb(\Bcal)$
  is a *-algebra are left to the reader. 
  
  In order to define a $C^*$-norm on $\Mb_\tb(\Bcal),$ take a
  representation $T\colon \Bcal\to \Bb(\Hcal)$ such that $T|_{B_e}$ is
  faithful (and so an isometry). 
  Then the restriction of $T$ to each fiber is an isometry and we have
  a *-representation $T^\tb \colon \Mb_\tb(\Bcal)\to M_n(\Bb(\Hcal))\cong\Bb(\Hcal^n),$
  given by $T^\tb (M)_{i,j}=T_{M_{i,j}}.$ 
  Observe that $\Mb_\tb(\Bcal)$ is $\|\ \|_\infty\dsh$complete and 
that $T^\tb$ is an isometry when we consider on its domain and range
the supremum norm. 
  Thus $T^\tb(\Mb_\tb(\Bcal))$ is a C*-subalgebra of $M_n(\Bb(\Hcal))$ and its
  C*-norm is equivalent to $\|\ \|_\infty.$
  Hence $\Mb_\tb(\Bcal)$ is a C*-algebra and its C*-norm is equivalent
  to the supremum norm.
  
  To prove claim (3) we start by noticing that
  \begin{equation*}
    \| \sum_{i=1}^n \lab x_i,y_i\rab \|^2  = \|  \sum_{i,j=1}^n \lab
    x_i\lab x_j,y_j\rab,y_i\rab  \| = \|  \sum_{i,j=1}^n \lab \laa
    x_i, x_j\raa y_j,y_i\rab  \| 
  \end{equation*}
  and that the last term looks like the norm of a matrix multiplication.
  Given $\rb:=({r_1},\ldots,{r_n})\in G^n$, consider the sum
  of right Hilbert $B_e\dsh$modules
  $\Xb_{\rb}:=X_{r_1}\oplus\cdots\oplus X_{r_n}$.  
  Writing the elements of $\Xb_{\rb}$ as column matrices, matrix
  multiplication gives an action of $\Mb_{\rb^{-1}}(\Acal)$ on
  $\Xb_{\rb}$ by adjointable operators.  
  Moreover, the formula ${}_{\Mb_{\rb^{-1}}(\Acal)}\la \eta,\zeta\ra :=
  (\laa \eta_i,\zeta_j\raa)_{i,j=1}^n$ defines a
  $\Mb_{\rb^{-1}}(\Acal)\dsh$valued inner product making
  $\Xb_{\rb}$ a $\Mb_{\rb^{-1}}(\Acal)\dsh B_e$ Hilbert
  bimodule (not full in general).
  We should warn the reader that we gave no justification of the
  positivity of the left inner product of $\Xb_{\rb}.$ 
  This actually follows from the positivity of the right inner product
  because, according to  \cite{abadie2016applications}, this implies
  $\Xb_{\rb}$ is a positive $C^*$-tring, so the left inner
  product must also be positive. Despite the previous comment, we give
  next a direct proof of this fact, for the convenience of the reader.  
  Note that $I:=\spn \{{}_{\Mb_{\rb^{-1}}(\Acal)}\la
  \eta,\zeta\ra\colon \eta,\zeta\in \Xb_{\rb^{-1}}\}$ is a *-ideal of
  $\Mb_{\rb^{-1}}(\Acal)$ and that \cite[VI 19.11]{FlDr88} implies it
  has a unique $C^*$-norm (see also \cite[Corollary
  3.8]{abadie2016applications} for a complete argumentation). 
    If $\phi:\Mb_{\rb^{-1}}(\Acal)\to \Bb(\Xb_{\rb})$ is the
  homomorphism corresponding to the action of $\Mb_{\rb^{-1}}(\Acal)$
  on $\Xb_{\rb}$ by adjointable operators, its   
 restriction $\phi|_I\colon I\to \Kb(\Xb_{\rb})$ is
  injective, because for every $a\in I$ the condition $\phi(a)\xi=0,$
  $\forall\ \xi\in \Xb_{\rb},$ implies $aa^*=0.$ 
  Then the closure of $I$ in $\Mb_{\rb^{-1}}(\Acal)$ is isomorphic, as
  a $C^*$-algebra, to the closure of $\phi(I)$ in
  $\Kb(\Xb_{\rb}).$ 
  Finally note that $\phi({}_{\Mb_{\rb^{-1}}(\Acal)}\la \eta,\eta\ra)$
  is the generalized compact operator $\zeta\mapsto \eta\la
  \eta,\zeta\ra_{B_e},$ which is positive. So $
  {}_{\Mb_{\rb^{-1}}(\Acal)}\la \eta,\eta\ra $ is positive in
  $\Mb_{\rb^{-1}}(\Acal).$ 
  
  The discussion above implies that
  \begin{equation}\label{equation norm identity for columnn vector}
    \|(\laa x_i,x_j\raa)_{i,j=1}^n\| = \|(x_1,\ldots,x_n)^t\|^2 =
    \|\sum_{i=1}^n\lab x_i,x_i\rab\|. 
  \end{equation}
  
  On the other hand, if $\rb t:=(r_1t,\ldots,r_nt)$, so $\Xb_{\rb
    t}=X_{r_1t}\oplus \cdots\oplus X_{r_nt}$, matrix multiplication
  gives a representation $\varphi\colon 
  \Mb_{\rb^{-1}}(\Acal)\to \Bb(\Xb_{\rb t})$
  and, if
  $x=(x_1,\ldots,x_n)^t\in \Xb_{\rb}$ and 
  $y=(y_1,\ldots,y_n)^t\in \Xb_{\rb t}$:  
  \begin{equation}\label{equation norm inequality for inner products}
    \| \sum_{i=1}^n \lab x_i,y_i\rab \|^2 =\| \la \varphi({}_{\Mb_{\rb^{-1}}(\Acal)}\la x,x\ra)y,y\ra_{B_e} \|\leq \|y\|^2\|(\laa x_i,x_j\raa)_{i,j=1}^n\|.
  \end{equation}
  Finally, since $\|y\|^2 = \|\sum_{i=1}\lab y_i,y_i\rab\|,$
  \eqref{equation the hard bound} 
  follows from \eqref{equation norm identity for columnn vector} and
  \eqref{equation norm inequality for inner products}.   
\end{proof}

\begin{remark}
Given a Fell bundle $\Bcal$ over $G$, consider the complex vector
space  
\[\Bbbk_c(\Bcal):=\{k:G\times G\to\Bcal: k(r,s)\in B_{rs^{-1}}, \forall r,s\in
G,\textrm{ and
}k\textrm{ has compact support}\}\]
with the operations $k_1*k_2(r,s):=\int_G k_1(r,t)k_2(t,s)dt$ and
$k^*(r,s):=k(s,r)^*$. As shown in \cite{Ab03}, $\Bbbk_c(\Bcal)$ is a
$*$-algebra, and has a C*-completion $\Bbbk(\Bcal)$ (which in fact is
equal to $C^*_r(\Bcal)\rtimes_{\delta,r}G$, where $\delta$ is the 
dual coaction of $G$ on $C^*_r(\Bcal)$). Note that, given
$\tb=(t_1,\ldots,t_n)\in G^n$, we can think of every element of
$M_{\tb}$ as a function $M:G\times G\to \Bcal$ supported in   
$\{t_1,\ldots,t_n\}^2$ and such that $M(r^{-1}s)\in
B_{r^{-1}s}$, $\forall r,s\in G$. Then 
we have a natural inclusion of *-algebras $M_{\tb}\hookrightarrow 
\Bbbk_c(\Bcal)$, given by $M\mapsto k_{M}$, where
$k_{M}(r,s):=M(r^{-1},s^{-1})$, $\forall r,s\in G$. This is an
alternative way of proving that $M_{\tb}$ has a C*-algebra
structure. Besides it follows that, when $G$ is discrete, we have 
$C^*_r(\Bcal)\rtimes_{\delta,r}G=\varinjlim\limits_{\tb}M_{\tb}$. 
\end{remark}
\subsection{Morphisms of equivalence bundles}
In order to define a map between the equivalence bundles $\Xcal$ and
$\Ycal$ that takes into account the equivalence bundle structure, it
is convenient to think of $\Xcal$ and $\Ycal$ as bundles of
$C^*$-trings. Suppose $\Xcal$ and $\Ycal$ are an $\Acal\dsh \Bcal$ and
a $\Ccal\dsh \Dcal\dsh$equivalence bundle, respectively (all of them
bundles over $G$). 

\begin{definition}
  We say that $\rho\colon \Xcal\to \Ycal$ is a \textit{morphism of
    equivalence bundles} if it is a continuous map such that 

  \begin{enumerate}
  \item $\rho(X_t)\subset Y_t$, and the restriction
    of $\rho$ to $X_t$ is linear, for all $t\in G$. 
  \item $\rho(x\lab
  y,z\rab)=\rho(x)\la \rho(y),\rho(z)\ra_\Dcal,$ for all $x,y,z\in
  \Xcal.$
  \end{enumerate}

\end{definition}
\par It is easy to check that equivalence bundles and their morphisms
form a category. We denote this category by $\mathscr{E}$.   
\par We have a ``bundle version'' of \cite[Proposition 3.1]{Ab03}:

\begin{theorem}\label{theorem isomorphism of equivalence bundles}
  Let $\Xcal$ and $\Ycal$ be $\Acal\dsh \Bcal$ and $\Ccal\dsh
  \Dcal\dsh$equivalence bundles, respectively. 
  Assume $\rho\colon \Xcal\to \Ycal$ is a morphism of equivalence bundles.
  Then
  \begin{enumerate}
   \item $\|\rho(x)\|\leq \|x\|,$ for all $x\in \Xcal.$
   \item There are unique morphisms of Fell bundles over $G$,  
     %(Example \ref{example reflexive}) 
     denoted $\rho^l\colon \Acal\to \Ccal$
     and $\rho^r\colon \Bcal\to \Dcal,$ such that for all $x,y\in
     \Xcal$ we have: 
   $$\rho^l(\laa x,y\raa)={}_\Ccal\la \rho(x),\rho(y)\ra \mbox{ and }
   \rho^r( \lab x,y\rab)=\la \rho(x),\rho(y)\ra_\Dcal.$$ 
   \item $\rho^l(a)\rho(x)=\rho(ax)$ and $\rho(xb)=\rho(x)\rho^r(b),$
     for all $a\in \Acal,$ $x\in \Xcal$, and $b\in \Bcal.$ 
   \item\label{lateral isos of morphism} In case $\rho$ is bijective,
     it is an isomorphism of equivalence bundles over $G,$ $\rho^r$ and
     $\rho^l$ are isomorphisms, $(\rho^r)^{-1}=(\rho^{-1})^r$ and
     $(\rho^l)^{-1}=(\rho^{-1})^l.$ 
  \end{enumerate}
\end{theorem}
\begin{proof}
  Take $t\in G$ and consider $X_t$ as a $C^*$-tring with the ternary
  operation $(x,y,z)=x\lab y,z\rab.$ 
  Then $\mu_t\colon X_t\to Y_t$ such that $x\mapsto \rho(x), $ is a homomorphism
  of $C^*$-trings and \cite[Proposition 3.1]{Ab03} implies
  $\|\rho(x)\|\leq \|x\|.$ 
  Moreover, if $I_t:=\spncl \lab X_t,X_t\rab$ and $J_t:=\spncl \la
  X_t,X_t\ra_\Dcal,$ the above cited proposition implies there exists
  a unique *-homomorphism $\mu_t^r\colon I_t\to J_t$ sending $\lab 
  x,y\rab$ to $\la \rho(x),\rho(y)\ra_\Dcal. $ 
  
  Set $B^0_t:=\spn\{ \lab X_{r}, X_{rt}\rab\colon r\in G \}.$
  We claim that there exists a unique linear contraction $\nu_t\colon
  B^0_t\to D_t$ such that $\nu_t(\lab x,y\rab)=\la \rho(x),
  \rho(y)\ra_{\Dcal}.$ 
  Take $b=\sum_{j=1}^n\lab x_j,y_j\rab\in B^0_t$ and set $d:=
  \sum_{j=1}^n\la \rho(x_j),\rho(y_j)\ra_\Dcal.$ 
  It suffices to show that $\|d\|\leq \|b\|,$ which follows from
  \begin{align*}
   \|d\|^2
      & = \|\la d,d\ra_\Dcal \|
          = \|\sum_{j,k=1}^n \la \rho(x_k\lab x_j,y_j\rab  ),\rho(y_k)\ra_\Dcal \|
%       & = \|  \sum_{j,k=1}^n \mu_t^r\left(\lab x_k\lab x_j,y_j\rab
%       ,y_k\rab  \right) \| 
          = \|\mu_t^r(b^*b)\|
          \leq \|b\|^2.
  \end{align*}

  Then $\nu_t$ has a unique extension to a linear contraction
  $\rho^r_t\colon B_t\to D_t.$ 
  Let $\rho\colon \Bcal\to \Dcal$ be such that $\rho|_{B_t}=\rho^r_t.$
  It is readily checked that, given $b\in B^0_s$ and $c\in B^0_t,$
  we have
  $\rho^r(bc)=\nu_{st}(bc)=\nu_s(b)\nu_t(c)=\rho^r(b)\rho^r(c)$ and 
  $\rho^r(b^*)=\nu_\smu(b^*)=\nu_s(b)^*=\rho^r(b)^*$, from where it
  follows that $\rho^r$ is multiplicative and *-preserving.  
    
  We use \cite[II 13.16]{FlDr88} to show that $\rho^r$ is continuous. 
  Given $f,g\in C_c(\Xcal)$ and $t\in G$, let $[f,g,t]\colon G\to
  \Bcal$ be defined as $[f,g,t](r)=\lab f(t),g(tr)\rab.$ 
  Then $\Gamma^\Xcal_\Bcal:=\{[f,g,t]\colon f,g\in C_c(\Xcal),\ t\in
  G\}$ satisfies:  
\begin{enumerate} 
 \item $B_s=\spncl \{ u(s)\colon u\in \Gamma_\Bcal \}$
  for all $s\in G$ (by Lemma \ref{lemma first tools}).
 \item 
  $\rho^r\circ f\in \Gamma^\Ycal_\Dcal,$ for all $f\in
  \Gamma^\Xcal_\Bcal.$
\end{enumerate} 
  Hence $\rho^r$ is continuous.
  
  From the last two paragraphs it follows that $\rho^r$ is a morphism of Fell bundles over $G$.
  In order to prove the existence of $\rho^l$ define
  $\widetilde{\rho}\colon \widetilde{\Xcal}\to \widetilde{\Ycal}$ as
  $\widetilde{\rho}=\rho$ (recall that, as topological spaces,
  $\widetilde{\Xcal}=\Xcal$). 
  Then $\rho^l$ is nothing but $\widetilde{\rho}^r.$
  We leave to the reader the routine verification of the claims in
  (3).
  
  Regarding (4), if $\rho$ is injective then each $\rho|_{X_t}$ is an injective homomorphism of C*-trings and thus an isometry.
  Hence, by \cite[II 13.17]{FlDr88}, $\rho^{-1}$ is continuous.
  The rest of the proof, which follows from the remark below, is left to the reader. 
\end{proof}
\begin{remark}
As a result of the second part of \ref{theorem isomorphism of
  equivalence bundles} we get two functors
$\mathscr{E}\to\mathscr{F}$: the left functor,  
given by $(\Xcal\stackrel{\rho}{\to}\Ycal)\longmapsto
(\Acal\stackrel{\rho^l}{\to}\Ccal)$, and the right functor, given by 
$(\Xcal\stackrel{\rho}{\to}\Ycal)\longmapsto
(\Bcal\stackrel{\rho^r}{\to}\Dcal)$.  
\end{remark} 
\section{Fell bundles associated to \FH bundles}

\subsection{The linking bundle}

Each of the equivalence bundles presented in our motivating
examples was constructed inside an ambient Fell bundle. 
Although this will turn out to be the general situation, a priori we
do not dispose of a Fell bundle that contains a given equivalence
bundle. The first purpose of this section is precisely to show that
any equivalence bundle can be included in a certain Fell bundle, the
so called linking Fell bundle, provided by 
Theorem~\ref{theorem construction of linking bundle} below.
To this
end we follow the idea used in \cite[Proposition~4.5]{Ab03} to define
the linking partial action of two Morita equivalent partial actions
(see Example~\ref{subsection Morita equivalence of partial actions}).
\par The next
result will help us to prove the continuity of the operations to be
defined along its construction.    
%tool to construct a Fell bundle from an untopologized bundle (as in
%\cite[II 18]{FlDr88}). 
%The main problem is to show the operations are continuous with respect
%to a topology specified by a set of continuous sections. 
%We solve this problem adapting \cite[II 13.16]{FlDr88}.

\begin{proposition}\label{prop continuity}
  Let $\Ucal,$ $\Vcal$ and $\Wcal$ be (real or complex) Banach bundles
  over the LCH spaces $X,$ $Y$ and $Z,$ respectively. 
  Assume $\Phi\colon \Ucal\times \Vcal \to \Wcal$ is a function for
  which there exist a continuous map $f\colon X\times Y \to Z,$ a
  constant $k\geq 0$ and sets of sections $\Gamma_Q \subset C(Q)$ for
  $Q \in \{\Ucal, \Vcal,\Wcal\}$ such that: 
  \begin{enumerate}[(1)]
    \item $\Phi(U_x\times V_y)\subset W_{f(x,y)}$ and $U_x\times V_y\to W_{f(x,y)},\ (u,v)\mapsto \Phi(u,v),$ is $\R\dsh$bilinear, for all $(x,y)\in X\times Y.$
    \item  $\|\Phi(u,v)\|\leq k\|u\|\|v\|,$ for all $u\in \Ucal $ and $v\in \Vcal.$  
    \item For all $x\in X,$ $\{\xi(x)\colon \xi \in \Gamma_\Ucal\}$ is dense in $U_x$ and analogous conditions hold for $\Vcal$ and $\Wcal.$
    \item For all $\xi \in \Gamma_\Ucal,$ $\eta \in \Gamma_\Vcal$ and $\zeta \in \Gamma_\Wcal$  the function 
    $$X\times Y \to \R,\ (x,y)\mapsto \|\Phi(\xi(x),\eta(y)) - \zeta(f(x,y))\|,$$
    is continuous.
  \end{enumerate}
  Then $\Phi$ is continuous.
\end{proposition}
\begin{proof}
  We start by observing that the inequality 
  \begin{equation}\label{equ bound}
   \|\Phi(u_1,v_1)-\Phi(u_2,v_2)\|\leq k(\|u_1\|+\|v_2\|)(\|u_1-u_2\|+\|v_1-v_2\|)
  \end{equation}
  holds for all $u_1,u_2\in U_x,\ v_1,v_2\in V_y,\ x\in X$ and $y\in Y.$
  
   Let us show that given a converging net in $\Ucal\times \Vcal,$
  $(u_\lambda,v_\lambda)\to (u,v),$ it follows that
  $\Phi(u_\lambda,v_\lambda)\to \Phi(u,v).$ 
  Suppose $u_\lambda\in U_{x_\lambda},$ $v_\lambda\in U_{y_\lambda},$
  $u\in U_x$ and $v\in V_y.$ 
  Obviously, $\Phi(u_\lambda,v_\lambda)\in V_{f(x_\lambda,y_\lambda)}$
  and $f(x_\lambda,y_\lambda)\to f(x,y).$  
  Given $\varepsilon >0,$ choose $\delta\in (0,\varepsilon)$ such that
  $k(\|u\|+\|v\|+2\delta)2\delta<\varepsilon.$ 
  Then use condition (3) to find $\xi\in \Gamma_\Ucal$ and $\eta\in
  \Gamma_\Vcal$ such that $\| u - \xi(x) \|<\delta$ and
  $\|v-\eta(y)\|<\delta.$ 
  Then Equation \eqref{equ bound} implies $\| \Phi(\xi(x),\eta(y)) -
  \Phi(u,v) \|<\varepsilon.$ 
  A direct continuity argument implies that there exists $\lambda_0\in
  \Lambda$ such that $\|u_\lambda-\xi(x_\lambda)\|<\delta$ and
  $\|v_\lambda-\xi(y_\lambda)\|<\delta,$ for all $\lambda\geq
  \lambda_0.$ 
  Using inequality \eqref{equ bound} again we get $\|
  \Phi(\xi(x_\lambda),\eta(y_\lambda)) - \Phi(u_\lambda,v_\lambda)
  \|<\varepsilon,$ for all $\lambda\geq \lambda_0.$ 
  In case $\Phi(\xi(x_\lambda),\eta(y_\lambda))\to
  \Phi(\xi(x),\eta(y)),$ \cite[II 13.12]{FlDr88} implies that
  $\Phi(u_\lambda,v_\lambda)\to \Phi(u,v).$ 
  
  Now we show $\Phi(\xi(x_\lambda),\eta(y_\lambda))\to
  \Phi(\xi(x),\eta(y)).$ 
  Fix $\varepsilon>0,$ condition (3) implies the existence of $\zeta \in \Gamma_\Wcal$ such that
  $\|\Phi(\xi(x),\eta(y))-\zeta(f(x,y))\|<\varepsilon.$ 
  Then there exists $\lambda_0$ such that $\|
  \Phi(\xi(x_\lambda),\eta(y_\lambda)) - \zeta(f(x_\lambda,y_\lambda))
  \|<\varepsilon$ for all $\lambda\geq \lambda_0.$ 
  It is clear that $\zeta(f(x_\lambda,y_\lambda))\to \zeta(f(x,y)),$
  thus \cite[II 13.12]{FlDr88} implies
  $\Phi(\xi(x_\lambda),\eta(y_\lambda))\to \Phi(\xi(x),\eta(y)).$ 
\end{proof}

\begin{comment}  
Take such an action, $\gamma,$ and construct the equivalence bundle
$\Xcal_\gamma$ as in \ref{subsection Morita equivalence of
  partial actions}. 
We used the Fell bundle associated to the linking partial action
$\Lb(\gamma)$ to define operations and, more importantly, to construct
the Fell bundle $\Bcal_{\Lb(\gamma)}$, the ``ambient space'' where
the argumentation regarding crossed products was carried out
in \cite{Ab03,AbMr09}. 
In our next theorem we construct such an ambient space for every
equivalence bundle. 
\end{comment}  

\begin{theorem}\label{theorem construction of linking bundle}
  Given an $\Acal\dsh \Bcal\dsh$equivalence bundle, $\Xcal,$ there is
  a unique Fell bundle $\Lb(\Xcal)=\{L_t\}_{t\in G}$ such that: 
  \begin{enumerate}
   \item For all $t\in G,$ $L_t 
   =\begin{pmatrix}
    A_t & X_t\\
    \widetilde{X}_\tmu & B_t
   \end{pmatrix}$ with entrywise vector space operations.% (here $\widetilde{X_\tmu}$ stands for the Banach space $X_\tmu$ with complex conjugate scalar multiplication).
  \item Product and involution are given by
  $$ \begin{pmatrix}
    a & x\\
    \widetilde{y} & b
   \end{pmatrix}
    \begin{pmatrix}
    c & u\\
    \widetilde{v} & d
      \end{pmatrix}
    =\begin{pmatrix}
    ac + \laa x,v\raa & au+xd\\
    \widetilde{c^*y} + \widetilde{vb^*} & \lab y,u\rab + bd
   \end{pmatrix} \quad \mbox{and}
   $$
   $$     
   \begin{pmatrix}
    a & x\\
    \widetilde{y} & b
   \end{pmatrix}^*
    =\begin{pmatrix}
    a^* & y\\
    \widetilde{x} & b^*
   \end{pmatrix}.$$
   \item Given $\xi\in C_c(\Acal),$ $\eta\in C_c(\Bcal)$ and $f,g\in C_c(\Xcal)$ the function
   $$\begin{pmatrix}
    \xi & f\\
    \widetilde{g} & \eta
   \end{pmatrix}\colon G\to \Lb(\Xcal), \ t\mapsto 
   \begin{pmatrix}
    \xi(t) & f(t)\\
    \widetilde{g(\tmu)} & \eta(t)
   \end{pmatrix} $$
   is a continuous section.
  \end{enumerate}
\end{theorem}
\begin{proof}
  All the necessary algebraic verifications follow from Lemma
  \ref{lemma first tools}, and will be ommited.
  To define a $C^*$-norm on $\mathbb{L}(\Xcal)$ and to (automatically) show $u^*u\geq 0$ in $L_e,$ for all $u\in \mathbb{L}(\Xcal),$ we will provide a representation of $\mathbb{L}(\Xcal)$.
  We proceed as follows. First consider the Hilbert $B_e$-module direct sums:
\[ \ell^2(\mathcal{B}):=\oplus_{t\in G}B_t\qquad\textrm{ and }\qquad
\ell^2(\mathcal{X}):=\oplus_{t\in G}X_t.\] 
Given $x_r\in X_r$ and
a section $\xi\in \ell^2(\Bcal),$ define the section 
$\Omega_{x_r}\xi \colon G\to \Xcal$ by 
$\Omega_{x_r}\xi(s):=x_r\xi(r^{-1}s)$. Note that $x_r\xi(r^{-1}s)\in
X_rB_{r^{-1}s}\subseteq X_s$, as needed.
Besides, 
\begin{align*}
\pr{\Omega_{x_r}\xi}{\Omega_{x_r}\xi}_{\ell^2(\Xcal)}
  &  =\sum_s\pr{\Omega_{x_r}\xi(s)}{\Omega_{x_r}\xi(s)}_{\Bcal}
     =\sum_s\pr{x_r\xi(r^{-1}s)}{x_r\xi(r^{-1}s)}_{{\Bcal}}\\
  & =\sum_s\xi(r^{-1}s)^*\pr{x_r}{x_r}_{{\Bcal}}\xi(r^{-1}s)
  \leq\norm{x_r}^2\,\pr{\xi}{\xi}_{\ell^2(\Bcal)}.
\end{align*}
Therefore there exists a unique bounded operator $\Omega_{x_r}\colon
\ell^2(\Bcal)\to \ell^2(\Xcal)$, $\xi \mapsto \Omega_{x_r}\xi$, which
is easily seen to be adjointable with adjoint given by:
\[\Omega_{x_r}^*\eta(t)=\pr{x_r}{\eta(rt)}_{\Bcal}, \qquad\forall
\eta\in\ell^2(\Xcal), \,t\in G.\]   
In the particular case $\Xcal=\Bcal$, for each
$b_r\in\Bcal$ we have an adjointable map
$\Lambda_{b_r}\in\Bb(\ell^2(\Bcal))$, such that
$\Lambda_{b_r}\xi(s)=b_r\xi(r^{-1}s)$, $\forall \xi\in\ell^2(\Bcal)$,
$s\in G$ (note that $\Lambda$ so defined is nothing but the left 
regular representation of $\Bcal$ as a Fell bundle over the group $G$
with the discrete topology). Similarly, we also have a map 
$\Lambda':\mathcal{A}\to\Bb(\ell^2(\mathcal{X}))$, such that
$\Lambda'_{a_r}\eta(s)=a_r\eta(r^{-1}s)$. Now it is easy to check that
the following relations hold:       
\begin{gather*}
\Omega_{ax}=\Lambda'_a\Omega_x,\qquad \textrm{and}\qquad
   \Omega_{xb}=\Omega_x\Lambda_b.\\ 
 \Omega_x^*\Omega_y=\Lambda_{\pr{x}{y}_{\mathcal{B}}}\in\Bb (\ell^2(\mathcal{B})),\  
   \forall x,y\in\mathcal{X}\\   
 \Omega_x\Omega_y^*=\Lambda'_{\laa x ,y
   \raa}\in\Bb(\ell^2(\mathcal{X})), \   
   \forall x,y\in\mathcal{X}
\end{gather*}
\par Define 
$\phi:\mathbb{L}(\Xcal) \to \mathbf{C}:=
\begin{pmatrix}\Bb(\ell^2(\mathcal{X}))&\Bb(\ell^2(\mathcal{B}),\ell^2(\mathcal{X}))\\
  \Bb(\ell^2(\mathcal{X}),\ell^2(\mathcal{B}))&\Bb(\ell^2(\mathcal{B}))\end{pmatrix}$
by\[\phi\begin{pmatrix}a&x\\ \tilde{y}&b\end{pmatrix}
=\begin{pmatrix}\Lambda'_{a}&\Omega_{x}\\
  \Omega_{y}^*&\Lambda_{b}\end{pmatrix}\] 

Then $\phi$ is multiplicative, preserves adjoints and is linear in each
fiber of $\mathbb{L}(\Xcal)$. Moreover, $\phi|_{L_e}$ is a faithful
homomorphism into the $C^*$-algebra $\mathbf{C}$. Thus if we define 
 $\norm{\ell}:=\norm{\phi(\ell)}_{\mathbf{C}}$, we get a C*-norm on
 $\mathbb{L}(\Xcal).$ Of course this $C^*$-norm is equivalent to
the norm $\norm{\ }_\infty$, defined as
$\norm{\ell}_\infty:=\max\{\norm{a},\norm{x},\norm{\tilde{y}},\norm{b}\}$
for $\ell=\begin{pmatrix}a&x\\ \tilde{y}&b\end{pmatrix}$.  
\par  To endow $\Lb(\Xcal)$ with a bundle topology, note that 
   $$\Gamma:=\Big\{\left( \begin{array}{cc}
    \xi & f\\
    \widetilde{g} & \eta
   \end{array} \right)\colon f,g\in C_c(\Xcal),\xi\in
 C_c(\Acal),\eta\in C_c(\Bcal)\Big\}$$ 
   is a subspace of sections such that $\{u(t)\colon u\in \Gamma\}=L_t,$
   for all $t\in G$ (recall the definition of $\left( \begin{array}{cc}
    \xi & f\\
    \widetilde{g} & \eta
   \end{array} \right)$ made in part (3) of the statement of the
 theorem).    
   Moreover, for every $u\in \Gamma$ the entries of $t\mapsto
   u(t)^*u(t)\in L_e$ are continuous functions. 
   Then $t\mapsto \|u(t)\|=\|u(t)^*u(t)\|^{1/2}$ is continuous, so 
   there is a unique Banach bundle structure on $\Lb(\Xcal)$ such that
   $\Gamma\subset C_c(\Lb(\Xcal)).$ 
   \par Using similar arguments it can be shown that, for $u,v,w\in
   \Gamma,$ the functions $ G\times G\to \R:\ (r,s)\mapsto \|
   u(r)v(s)-w(rs)\| $ and $G\times G\to \R: \ (r,s)\mapsto
   \|u(s)^*-v(\smu)\| $ are continuous. 
   For example, note that 
   $$(r,s)\mapsto (u(r)v(s)-w(rs))^*(u(r)v(s)-w(rs))$$
   has continuous entries.
   Then Proposition \ref{prop continuity} implies the involution and
   multiplication of $\Lb(\Xcal)$ are continuous. 
\end{proof}

It can be shown that the continuous sections described in condition
(3) of the previous theorem are all of the continuous sections of compact
support of $\Lb(\Xcal).$ 

\begin{remark}\label{remark direct sums}
We can construct the direct sum $\Acal\oplus \Xcal$ as the Banach
subbundle 
$$\left\{ \left( \begin{array}{cc} a & x\\ 0 & 0 \end{array}
  \right)\colon a\in A_t,\ x\in X_t,\ t\in G \right\}\subset
\Lb(\Xcal).$$ 

Then $\Acal\oplus \Xcal$ becomes an $\Acal\dsh
\Lb(\Xcal)\dsh$equivalence bundle. Moreover, in a similar way we can
define $\Xcal\oplus \Bcal$ and make it into a $\Lb(\Xcal)\dsh
\Bcal\dsh$equivalence bundle.  
\end{remark}

As we expected, linking partial actions give rise to linking bundles: 

\begin{proposition}
  Let $X$ be an equivalence $A\dsh B$ bimodule and $\gamma$ a partial
  action on $X$ (see section \ref{subsection Morita equivalence of
    partial actions}). 
  If $\Xcal_\gamma$ is the Fell bundle associated to $\gamma,$ then
  $\Lb(\Xcal_\gamma)$ is isomorphic to $\Bcal_{\Lb(\gamma)}.$ 
\end{proposition}
\begin{proof}
  Define $\beta:=\gamma^r$ as in Section \ref{subsection Morita
    equivalence of partial actions}, and let 
  $ \Ycal:= \Xcal_\gamma\oplus \Bcal_{\be}$ \mbox{ and }
  $$\Zcal:=\left\{ \left( \begin{array}{cc} 0 & x\\0 & b \end{array}
 \right)\delta_t\colon x\in X_t,\ b\in B_t,\ t\in G \right\} \subset
\Bcal_{\Lb(\gamma)}.$$ 
 Then $\Ycal$ is a $\Lb(\Xcal_\gamma)\dsh \Bcal_\beta\dsh$equivalence
 bundle and $\Zcal$ a $\Bcal_{\Lb(\gamma)}\dsh
 \Bcal_\beta\dsh$equivalence bundle. 
 Since the map 
 $$\rho\colon \Ycal\to \Zcal \mbox{ given by }\rho
 \left( \begin{array}{cc} 0 & x\delta_t\\0 & b\delta_t \end{array} 
 \right)=\left( \begin{array}{cc} 0 & x\\0 & b \end{array}
 \right)\delta_t$$
 satisfies the hypotheses of Theorem \ref{theorem isomorphism of
   equivalence bundles} (\ref{lateral isos of morphism}), 
 $\rho^l\colon\Lb(\Xcal)\to \Bcal_{\Lb(\gamma)} $ is an isomorphism of
 Fell bundles. 
\end{proof}
\begin{theorem}\label{theorem the linking of a morphism}
Let $\rho:\Xcal\to\Ycal$ be a morphism from the $\Acal-\Bcal$
equivalence bundle $\Xcal$ to the $\Ccal-\Dcal\dsh$equivalence bundle
$\Ycal$. Let $\Lb(\rho):\Lb(\Xcal)\to\Lb(\Ycal)$ be the map given by (recall
 $\rho^l$ and $\rho^r$ were defined in \ref{theorem isomorphism of
   equivalence bundles}): 
\[\Lb(\rho)\begin{pmatrix}
    a & x\\
    \widetilde{y} & b
   \end{pmatrix}
  :=\begin{pmatrix}
    \rho^l(a) & \rho(x)\\
    \widetilde{\rho(y)} & \rho^r(b)
   \end{pmatrix},\quad\forall \begin{pmatrix}
    a & x\\
    \widetilde{y} & b
   \end{pmatrix}\in \Lb(\Xcal).
\]
Then $\Lb(\rho)$ is a morphism of Fell bundles, and
$(\Xcal\stackrel{\rho}{\to}\Ycal)\longmapsto
\Big(\Lb(\Xcal)\stackrel{\Lb(\rho)}{\longrightarrow}\Lb(\Ycal) \Big)$ is a functor
from the category $\mathscr{E}$ of equivalence bundles to the category
$\mathscr{F}$ of Fell bundles.    
\end{theorem}
\begin{proof}
  We just need to verify that the map $\Lb(\rho)$ defined in the
  statement is a morphism of Fell bundles over $G$.   
  The routine algebraic verifications are left to the reader. 
  Now note that the map $\Lb(\Xcal)_e\to \Lb(\Ycal)_e,\ S\mapsto
  \Lb(\rho)S,$ is contractive because it is a homomorphism 
  of $C^*$-algebras. 
  Then, for all $S\in \Lb(\Xcal),$ we have
  $$\|\Lb(\rho)S\|=\|\Lb(\rho)[S^*S]\|^{1/2}\leq \|S^*S\|^{1/2}=\|S\|.$$
  Now observe that given $\left( \begin{array}{cc}
    \xi & f\\
    \widetilde{g} & \eta
   \end{array} \right)\in C_c(\Lb(\Xcal))$ as in \ref{theorem construction of linking bundle},  we have
   \begin{equation}\label{equation linking rho composed with a special section}
    \Lb(\rho)\circ \left( \begin{array}{cc}
    \xi & f\\
    \widetilde{g} & \eta
   \end{array} \right) = \left( \begin{array}{cc}
    \rho^l\circ \xi & \rho\circ f\\
    \widetilde{\rho\circ g} & \rho^r\circ \eta
   \end{array} \right)\in C_c(\Lb(\Ycal)).
   \end{equation}
   Using \cite[II 13.16]{FlDr88} we conclude $\Lb(\rho)$ is continuous, so it is a morphism of Fell bundles over $G$.
\end{proof}

\subsection{The bundle of generalized compact operators}

Assume $\Xcal$ is a right \FH $\Bcal\dsh$bundle. 
We will construct a Fell bundle $\Kb(\Xcal)$ in such a way that
$\Xcal$ is a $\Kb(\Xcal)\dsh \Bcal\dsh$equivalence bundle. 
In view of Theorem \ref{theorem isomorphism of equivalence bundles},
this Fell bundle is uniquely determined by the right \FH bundle
structure of $\Xcal.$ 

\begin{definition}\label{defi:adjointable operator of equivalence bundle}
  An adjointable operator of order $t\in G$ of $\Xcal$ is a continuous
  map $S:\Xcal\to\Xcal$ with the following properties:  
  \begin{itemize}
    \item there exists $c\in \R$ such that $\norm{Sx}\leq
      c\norm{x},$ for all $x\in \Xcal.$ 
    \item $S(X_r)\subset X_{tr}$, for all $r\in G.$
    \item There exists $S^*:\Xcal\to\Xcal$ such that
      $\pr{Sx}{y}_{\Bcal}=\pr{x}{S^*y}_{\Bcal}$, $\forall x,y\in
      \Xcal$.  
  \end{itemize}
  The set of adjointable operators of order $t$ will be denoted
  $\Bb_t(\Xcal).$  
\end{definition}
\par If $S_1,S_2\in \Bb_t(\Xcal)$ and
$\alpha\in\mathbb{C}$, is clear that
$\alpha S_1+S_2$ is also an adjointable operator of order
$t$. We define a norm on $\Bb_t(\Xcal)$ by 
$\norm{S}:=\sup\{\norm{Sx}:\,\norm{x}\leq 1\}$. Then we have an
isometric map $\Bb_t(\Xcal)\to\oplus_{r\in G}\Bb(X_r,X_{tr})$, given
by $S\mapsto(S_r)_{r\in G}$, where $S_r:X_r\to X_{tr}$ is such that
$S_r(x)=S(x)$, $\forall x\in X_r$. Since for all $x\in X_r$ and $y\in 
X_{tr}$ we have
$\pr{Sx}{y}_{\Bcal}=\pr{S_rx}{y}_{\Bcal}=\pr{x}{S_r^*y}_{\Bcal}$, we
see that the map $S^*$ is determined by $S$, and $(S^*)_r=(S_r)^*$,
$\forall r\in G$. We call $S^*$ the adjoint of $S$. 
 Note that, using Proposition \ref{prop continuity} and
 Definition \ref{defi:adjointable operator of equivalence bundle}, it
 can be shown that $S^*$ is continuous.
Then $S^*$ is an adjointable map of order $t^{-1}$. 
Conversely, for
every $(T_r)\in \oplus_{r\in G}\Bb(X_r,X_{tr})$ there exists a unique
$S:\Xcal\to\Xcal$ such that $S_r=T_r$, $\forall r\in G$. However, this
$S$ does not need to be a continuous map. Despite this fact we have:    
\begin{lemma}
 $\Bb_t(\Xcal)$ is a Banach space.
%, in particular $\Bb_e(\Xcal)$ is a $C^*$-algebra. 
\end{lemma}
\begin{proof}
Let $(S^{(n)})$ be a Cauchy sequence in $\Bb_t(\Xcal)$. Then
$(S^{(n)}_r)_{r\in G}$ is a Cauchy sequence in the complete space
$\oplus_{r\in G}\Bb(X_r,X_{tr})$, so it has a limit
$(S_r)\in\Bb(X_r,X_{tr})$. Let $S:\Xcal\to\Xcal$ be given by
$Sx:=S_rx$, $\forall x\in X_r$, $r\in G$. It is enough to show that
$S$ is continuous, what we do next. Take a net
$\{x_\lambda\}_{\lambda\in \Lambda}\subset 
\Xcal$ converging to $x_0\in \Xcal.$ Suppose $x_\lambda\in
X_{r_\lambda}$ and $x_0\in X_{r_0}$ and fix $\varepsilon >0.$  
 We can find $n\in \N$ such that $\|Sx-S_nx\|<\varepsilon,$ for
 all $x\in \Xcal$ with $\norm{x}\leq 1$.  
 Since $S^{(n)}$ is continuous, $S^{(n)}x_\lambda\to_\lambda S^{(n)}x_0.$
 Hence \cite[II 13.12]{FlDr88} implies $Sx_\lambda\to Sx_0.$   
\end{proof}
\begin{corollary}
Let $G_d$ be the group $G$ endowed with the discrete topology, and let
$\Bb(\Xcal):=(\Bb_t(\Xcal))_{t\in G_d}$. Then $\Bb(\Xcal)$ is a Fell
bundle over $G_d$, with the product given by the composition of maps. 
\end{corollary}
\begin{proof}
Observe that if $S_1$ is of order $t_1$ and $S_2$ is of order
$t_2$, then $S_1S_2$ is of order $t_1t_2$. On the other hand, if $S\in
\Bb_t(\Xcal)$ corresponds to $(S_r)\in \oplus_{r\in
  G}\Bb(X_r,X_{tr})$, then $S^*S$ corresponds to $(S_r^*S_r)\in
\oplus_{r\in G}\Bb(X_r)$. Since $S_r^*S_r$ is a positive element of
$\Bb(X_r)$, we have $S_r^*S_r=T_r^*T_r$, for some $T_r\in\Bb(X_r)$,
$\forall r\in G$. Thus $S^*S=T^*T$, where $T\in \Bb_e(\Xcal)$
corresponds to the element $(T_r)\in \oplus_{r\in
  G}\Bb(X_r)$. Therefore $S^*S$ is a positive element of the C*-algebra
$\Bb_e(\Xcal)$. The remaining verifications are routine and we ommit
them.        
\end{proof}

\begin{theorem}\label{thm:compacts bundle}
  Let $\Bcal$ be a Fell bundle over the group $G$. Given a right
  \FH $\Bcal\dsh$bundle $\Xcal$ there exists 
  a unique Fell bundle over $G$, which we denote by $\Kb(\Xcal)$, such
  that:  
  \begin{enumerate}[(1)]
   \item For all $t\in G$ the fiber $\Kb(\Xcal)_t$ is, as a Banach
     space, the closure in $\Bb_t(\Xcal)$ of 
   $$\spn \{ [x,y]\colon x\in X_{ts},\  y\in X_s,\ s\in G  \}$$
where $[x,y]:\Xcal\to\Xcal$ is defined to be $[x,y]z:=x\lab y,z\rab$. 
%   \item The product and involution are the one defined for double
%   centralizers. 
   \item Given $f,g\in C_c(\Xcal)$ and $s\in G,$ the function
     $[f,g,s]\colon G\to \Kb(\Xcal)$ given by $[f,g,s](t)=
     [f(ts),g(s)],$  is a continuous section of $\Kb(\Xcal)$. 
  \end{enumerate}
\end{theorem}
\begin{proof}
  Note first that, if $x\in X_{ts}$ and $y\in X_s$, then
  $[x,y]X_r=x\pr{y}{X_r}_\Bcal\subseteq X_{ts}\Bcal_{s^{-1}r}\subseteq
  X_{tr}$, so $[x,y]\in\Bb_t(\Xcal)$. Since $[x,y][z,w]=[x\lab
  y,z\rab,w]$ and $[x,y]^*=[y,x]$, 
  $\Kb(\Xcal)$ is closed under multiplication and involution. 
  We want to define a topology on $\Kb(\Xcal)$ such that $\Kb(\Xcal)$
  is a Banach bundle with it, and the space $\Gamma:= \spn
  \{[f,g,s]\colon f,g\in C_c(\Xcal),\ s\in G\}$ is contained in the
  subspace of continuous sections of that bundle. By \cite[II
  13.18]{FlDr88} there exists at most one such 
  topology and, to prove its existence we must show that given 
  $n\in \N,$ $f_1,g_1,\ldots,f_n,g_n\in C_c(\Xcal)$, and
  $s_1,\ldots,s_n\in G,$  the function $h\colon G\to \R$, given by
  $h(t)=\|\sum_{j=1}^n[f_j,g_j,s_j](t)\|$, is continuous. Now if
  $k(t):=\sum_{j=1}^n[f_j,g_j,s_j](t)$, we have   
  $h(t)=\|k(t)k(t)^*\|^{1/2}$, so it suffices to show that the map $G\to 
  \Bb_e(\Xcal)$ $t\mapsto k(t)k(t)^*,$ is continuous. 
  In fact we just need to show that $t\mapsto
  [k(t)(g_j(s_j)),f_j(ts_j)]$ is continuous (for all $j=1,\ldots,n$)
  because $$ k(t)k(t)^* = \sum_{j=1}^n [k(t)(g_j(s_j)),f_j(ts_j)].$$ 
  Fix $j=1,\ldots, n$ and let $u,v\colon G\to \Xcal$ be defined as 
  $u(t):=k(t)(g_j(s_j))$ and $v(t):=f_j(ts_j).$ 
  Then $u$ and $v$ are continuous and, for all $z\in \Xcal$ with
  $\|z\|\leq 1,$ we have 
  \begin{align*}
    \|u(t) & \lab v(t),z\rab - u(r)\lab v(r),z\rab\|^2 \\
      & = \Big\|  \Big\lab u(t)\lab v(t),z\rab - u(r)\lab
      v(r),z\rab,u(t)\lab v(t),z\rab - u(r)\lab v(r),z\rab
      \Big\rangle_{\Bcal}  \Big\|\\ 
      & \leq \Big\| \lab z,v(t)\rab  \Big\lab v(t)\lab
      u(t),u(t)\rab - v(r)\lab u(r),u(t)\rab,z\Big\rangle_{\Bcal}
      \Big\| + 
       \\  
      & \qquad \qquad \Big\|\lab z,v(r)\rab\Big\lab v(t)\lab
      u(t),u(r)\rab - v(r)\lab u(r),u(r)\rab\Big\rangle_{\Bcal} \Big\|
       \\
      & \leq \|v(t)\|\|v(t)\lab u(t),u(t)\rab - v(r)\lab u(r),u(t)\rab \| +\\
      & \qquad \qquad \|v(r)\|\| v(t)\lab u(t),u(r)\rab - v(r)\lab u(r),u(r)\rab \|
  \end{align*}
  The right member of the above inequality is the sum of two terms
  that do not depend on $z$ and have limit $0$ when $r\to
  t$. Hence $t\mapsto [k(t)(g(s_j)),f_j(ts_j)]$ is continuous.  
\par We still have to show that multiplication and involution are 
continuous, for which we use Proposition \ref{prop continuity}. 
  As for the multiplication we need to show that, given $u,v,w\in
  \Gamma,$ the function $G\times G\to \R,\ (r,s)\mapsto
  \|u(r)v(s)-w(rs)\|,$ is continuous. 
  It is enough to prove that $(r,s)\mapsto
  (u(r)v(s)-w(rs))^*(u(r)v(s)-w(rs))$ is a continuous function from
  $G\times G$ to $\Kb(\Xcal)_e$, and this can be done by using the same 
  arguments we have used in the previous paragraphs.  
  \par To prove the involution is continuous,  
  let $\Vcal$ be the Banach bundle over $\{e\}$ with fiber 
  $\C$, and define $\Phi\colon \Kb(\Xcal)\times \Vcal\to \Kb(\Xcal)$  
  such that $\Phi(b,\lambda)=\lambda b^*.$ 
  The map $\Phi$ is continuous because of Proposition \ref{prop
    continuity}. Then the involution $\Kb(\Xcal)\to \Kb(\Xcal),\
  b\mapsto \Phi(b,1),$ also is continuous.   
\end{proof}

\begin{corollary}\label{corollary right \FH bundles are equivalence bundles}
  Every right \FH $\Bcal\dsh$bundle, $\Xcal,$ is a
  $\Kb(\Xcal)\dsh\Bcal\dsh$equivalence bundle with the action  
  $\Kb(\Xcal)\times \Xcal\to \Xcal$ given by $(b,x)\mapsto b(x),$ and the left
  inner product $\Xcal\times \Xcal\to \Kb(\Xcal)$ given by
  $(x,y)\mapsto[x,y].$  
\end{corollary}
\begin{proof}
 It is a straightforward consequence of Theorem~\ref{thm:compacts
   bundle}. 
\end{proof}

\begin{corollary}
  If $\Xcal$ is an $\Acal\dsh \Bcal\dsh$equivalence bundle, then there
  exists a unique isomorphism of Fell bundles $\pi\colon \Acal\to
  \Kb(\Xcal)$ such that $\pi(\laa x,y\raa)=[x,y].$ 
\end{corollary}
\begin{proof}
  Let $\Ycal$ be the bundle $\Xcal$ considered as a $\Kb(\Xcal)\dsh  
  \Bcal\dsh$equivalence bimodule and let $\id\colon \Xcal\to \Ycal$ be
  the identity. 
  Then, by Theorem \ref{theorem isomorphism of equivalence bundles},
  $\pi:=\id^l\colon \Kb(\Xcal)\to \Acal$ is the isomorphism we are
  looking for. 
\end{proof}

\begin{remark}\label{remark linking bundle}
  With the notation of the previous Corollary, $\Kb(\Xcal\oplus \Bcal)$ is isomorphic to $\Lb(\Xcal)$ because $\Xcal\oplus \Bcal$ is a $\Lb(\Xcal)\dsh \Bcal\dsh$equivalence bundle.
\end{remark}

\section{Morita-Rieffel equivalence of cross-sectional
  \texorpdfstring{$C^*$}{C*}-algebras}\label{section Morita equiva
  of cross-sectional Cast algebras} 

It is well known (see \cite{combes1984crossed}) that equivalent
actions on $C^*$-algebras have Morita-Rieffel equivalent crossed
products (full and reduced), and the same can be said about 
equivalent partial actions (\cite{Ab03} and \cite{AbMr09}). 
We will show in this section that, more generally, any $\Acal\dsh
\Bcal\dsh$equivalence bundle $\Xcal$ gives rise to 
a $C^*(\Acal)\dsh C^*(\Bcal)$ and also a $C^*_r(\Acal)\dsh
C^*_r(\Bcal)\dsh$equivalence bimodule. We consider first the case
of the full $C^*$-algebras, for which we will construct 
an equivalence bimodule contained in $C^*(\Lb(\Xcal))$. 
We will make use of the fact that $\Acal$ and $\Bcal$ are hereditary 
in $\Lb(\Xcal)$ in the following sense:  

\begin{definition}
  Given a Fell bundle $\Ccal$ and a Fell subbundle $\Acal\subset
  \Ccal,$ we say $\Acal$ is hereditary (in $\Ccal$) if
  $\Acal\Ccal\Acal\subset \Acal.$  
\end{definition}

The condition $A_e\Ccal A_e\subset \Acal$ may look weaker than
$\Acal\Ccal\Acal\subset \Acal,$ but in fact they are equivalent. 
Indeed, suppose the former condition holds and take $a,c\in  \Acal$
and $b\in \Ccal.$ 
Let $\{d_\lambda\}_{\lambda \in \Lambda}\subset A_e$ be an approximate
unit of $A_e$ and assume $abc\in C_t.$ 
Then the net $\{ ad_\lambda b d_\lambda c \}_{\lambda\in \Lambda}$
converges to $abc$ and is contained in $A_t.$ 
This implies $abc\in \Acal$ because $A_t$ is closed in $C_t.$

\begin{proposition}[Example of hereditary subbundles]
  Let $\beta=\left(\{B_t\}_{t\in G},\{\beta_t\}_{t\in G}\right)$ be a
  partial action of the group $G$ on the $C^*$-algebra $B.$ 
  Then for every ideal $A$ of $B$ there exists a unique partial action
  of $G$ on $A,$ $\beta|_A:=\left(\{A_t\}_{t\in G},\{\alpha_t\}_{t\in
      G}\right),$ such that: 
   $A_\tmu=A\cap \beta_\tmu(B_t\cap A)$ and $\alpha_t(b)=\beta_t(b),$
   for all $t\in G$ and $b\in A_\tmu.$ 
   Moreover, $\Bcal_{\be|_A}$ is hereditary in $\Bcal_\be.$
\end{proposition}
\begin{proof}
  Note that $\beta_t(A_\tmu) = \beta_t(B_\tmu \cap A\cap
  \beta_\tmu(B_t\cap A)) = \beta_t(B_\tmu\cap A)\cap A=A_t,$ so there
  exists a unique isomorphism of $C^*$-algebras $\alpha_t\colon
  A_\tmu\to A_t$ such that $\alpha_t(a)=\beta_t(a).$ 
  Clearly, $\al_e$ is the identity on $A.$
  If $a\in A_\tmu$ and $\alpha_t(a)\in A_\smu,$ then
  $$  a\in \beta_\tmu ( B_t \cap \beta_\smu(B_s\cap A)  )\subset
  \beta_\tmu(B_t\cap B_\smu)\subset B_\tmu\cap B_{\tmu\smu}. $$ 
  This implies
  $\beta_{st}(a)=\beta_s(\beta_t(a))=\beta_s(\alpha_t(a))\in \beta_s(
  \beta_\smu(B_s\cap A)  )\subset A.$ 
  Putting all this together we conclude that $a\in A\cap 
  \beta_{\tmu\smu}(B_{st}\cap A)=A_{\tmu\smu}$ and
  $\alpha_{st}(a)=\beta_{st}(a)=\beta_s(\beta_t(a))=\alpha_s(\alpha_t(a)).$ 
  Thus $\beta|_A$ is a set theoretic partial action.
\par To show that $\beta|_A$ is a continuous partial action on $A,$ it
suffices to prove that $\{A_t\}_{t\in G}$ is a continuous family.   
  To see this fix $t\in G$ and $b\in A_t.$ 
  It suffices to find $g\in C(G,A)$ such that $g(t)=b$ and $g(r)\in
  A_r,$ for all $r\in G$. The Cohen-Hewitt Theorem provides $x,y\in A$
  and $z\in B_\tmu$ such that $b=x\beta_t(yz).$ 
  Now pick $f\in C(G,B)$ such that $f(r)\in B_r$ (for all $r\in G$)
  and $f(\tmu)=z,$ what we can do because $\{B_r\}_{r\in G}$ is a
  continuous family. 
  Then the function $g\colon G\to A$ defined as
  $g(r)=x\beta_r(yf(\rmu))$ is continuous, $g(t)=b$ and $g(r)\in A_r$
  for all $r\in G.$ 
  
  To show $\Bcal_{\be|_A}$ is hereditary in $\Bcal_\be$ fix $a,c\in
  A_e$ and $b\in B_r$ and observe that 
  \begin{align*}
   a\delta_eb\delta_rc\delta_e
    & = a\delta_e(b\delta_rc\delta_e)= a\beta_r(\beta_\rmu(b)c) )\delta_r.
  \end{align*}
  Clearly, $a\beta_r(\beta_\rmu(b)c))\in A\cap \be_r( B_\rmu \cap A
  )=A_r,$ so $a\delta_eb\delta_rc\delta_e\in \Bcal_{\be|_A}.$ 
\end{proof}

\begin{comment}
The combination of the Proposition-Example above with the Theorem
below is a kind of combination\todo{What does it mean?} of the ideal
property and the 
hereditary subalgebra property of \cite{buss2015exotic} for partial
crossed products. 
\end{comment}

\begin{theorem}\label{theorem hereditary sub bundle}
  If $\Acal$ is an hereditary Fell subbundle of $\Bcal,$ then
  $C^*(\Acal)$ is the closure of $L^1(\Acal)$ in $C^*(\Bcal)$, and it
  is an hereditary subalgebra of $C^*(\Bcal)$.  
\end{theorem}
\begin{proof}
  Let $\Xcal$ be the Banach subbundle of $\Bcal$ such that, for each
  $t\in G,$ $X_t=\spncl\{ ab\colon a\in A_r,\ b\in B_{\rmu t},\ r\in G
  \}.$ 
  Now let $\Ccal$ be the Banach subbundle of $\Bcal$ such that, for
  each $t\in G,$ $C_t = \spncl \{ x^*y\colon x\in X_r,\ y\in X_{rt},\
  r\in G \}.$ 
  In fact $\Ccal$ is a Fell subbundle of $\Bcal$ and $\Ccal\Bcal\cup
  \Bcal\Ccal\subset \Ccal$; in other words $\Ccal$ is an ideal of
  $\Bcal.$ 
  Then we can think of $L^1(\Acal)$ as a *-Banach subalgebra of
  $L^1(\Ccal)$ and of $L^1(\Ccal)$ as a closed *-ideal of
  $L^1(\Bcal).$ 
  
  Using \cite[Theorem 1.1]{AbMr09} and \cite[Corollary
  5.3]{abadie2016applications} with $\Ecal=\Xcal,$ we conclude that
  $C^*(\Acal)$ is the closure of $L^1(\Acal)$ in $C^*(\Ccal).$ 
  Let $\pi\colon C^*(\Ccal)\to C^*(\Bcal)$ be the unique
  *-homomorphism extending the natural inclusion of $L^1(\Ccal)$ in
  $L^1(\Bcal).$ 
  To show that $\pi$ is injective take a non-degenerate faithful
  representation $\rho\colon C^*(\Ccal)\to \Bb(\Hcal).$ 
  In this situation we know form \cite[VI 19.11]{FlDr88} that
  $\rho|_{L^1(\Ccal)}$ can be extended in a unique way to a
  representation defined on all of $L^1(\Bcal)$. 
  Then there exists a unique representation $\overline{\rho}\colon
  C^*(\Bcal)\to \Bb(\Hcal)$ such that $\overline{\rho}\circ\pi
  (f)=\rho(f),$ for all $f\in L^1(\Ccal).$ 
  This implies that $\pi$ is injective because $\rho= \overline{\rho}\circ
  \pi.$ 
  Putting all this together we conclude that the maximal $C^*$-norm of 
  $L^1(\Acal)$ is the restriction of the maximal $C^*$-norm of
  $L^1(\Bcal)$. The last assertion of the statement is clear.   
\end{proof}

\begin{corollary}
  If $\Xcal$ is an $\Acal\dsh \Bcal\dsh$equivalence bundle, then
  $C^*(\Acal)$ and $C^*(\Bcal)$ are the closure of $L^1(\Acal)$ and of
  $L^1(\Bcal)$ in $C^*(\Lb(\Xcal)),$ respectively. 
\end{corollary}
\begin{proof}
  Just note that $A_e \Lb(\Xcal)A_e \subset \Acal$ and $B_e \Lb(\Xcal) B_e \subset \Bcal.$
\end{proof}

From now on we will think of $C^*(\Acal)$ and $C^*(\Bcal)$ as $C^*$-subalgebras of $C^*(\Lb(\Xcal)).$

\begin{theorem}\label{theorem construction of the equivalence bimodule}
  For every $\Acal\dsh \Bcal\dsh$equivalence bundle, $\Xcal,$ the
  closure of $C_c(\Xcal)$ in $C^*(\Lb(\Xcal)),$ $C^*(\Xcal),$ is a
  $C^*(\Acal)\dsh C^*(\Bcal)\dsh$equivalence bimodule with the
  bimodule structure inherited from $C^*(\Lb(\Xcal)).$ 
\end{theorem}
\begin{proof}
  We must show that $C^*(\Xcal)C^*(\Bcal)\subset C^*(\Xcal).$
  Given $f\in C_c(\Xcal)$ and $g\in C_c(\Bcal)$ we have $f*u\in C_c(\Xcal)$ because $f*u\in C_c(\Lb(\Xcal)),$  $f*u(t) = \int_G f(r)u(\rmu t)\, dr$ and $u(r)f(\rmu t)\in X_t,$ for all $r,t\in G.$
  Using the continuity of the product we see that
  $C^*(\Xcal)C^*(\Bcal)\subset C^*(\Xcal).$ 
  In a similar way we can show that $C^*(\Xcal)^*C^*(\Xcal)\subset C^*(\Bcal),$ so the right inner product $C^*(\Xcal)\times C^*(\Xcal)\to C^*(\Bcal),$ $(f,g)\mapsto f^**g,$ is defined.
  Moreover, this inner product is positive because $f^**f$ is positive in $C^*(\Lb(\Xcal)).$

  To prove that $C^*(\Xcal)$ is a full Hilbert $C^*(\Bcal)$-module it
  suffices to prove that every element of the form $f^**g$ ($f,g\in
  C_c(\Bcal)$) can be approximated, in the inductive limit topology,
  by a sum of (right) inner products. 
  Given $b\in B_e$ define $bg\in C_c(\Bcal)$ as $[bg](r):=bg(r).$
  Let $\{b_\lambda\}_{\lambda\in \Lambda}$ be an approximate unit of
  $B_e$ as the one given in Lemma \ref{lemma approximate units}. 
  Then $ b_\lambda g \to g$ and $f^**(b_\lambda g) \to f^**g$ in the
  inductive limit topology. 
  For every $\lambda \in \Lambda,$ the function $f^**(b_\lambda g)$ is
  a sum of elements of the form $f^**(\lab x,x\rab g),$ which we will
  prove are inner products. 
  Given $x\in X_s,$ consider $xf\in C_c(\Xcal)$ given by
  $(xf)(r):=xf(\smu r)$, and note that 
  $f^**(\lab x,x\rab g) = \la xf,xg\ra_{C^*(\Bcal)}$ because, for all $t\in G,$
  \begin{align*}
   f^**(\lab x,x\rab g)(t)
    & = \int_G \Delta(r)^{-1} \lab x f(\rmu),xg(\rmu t)\rab \, dr\\
    & = \int_G \lab x f(\smu r),xg(\smu rt)\rab \, dr
      = \la xf,xg\ra_{C^*(\Bcal)}(t).
  \end{align*}
  
  Similar arguments can be used to prove the claims concerning the $C^*(\Acal)\dsh$valued inner product.
  Finally, the compatibility of the operations is immediate because all the computations are performed within $C^*(\Lb(\Xcal)).$
\end{proof}

The construction of a $C^*$-algebra from a Fell bundle \cite{FlDr88}
motivates the following definition. 

\begin{definition}
  The \textit{(full) cross-sectional Hilbert bimodule} of the $\Acal\dsh
  \Bcal\dsh$equivalence bundle $\Xcal$ is the $C^*(\Acal)\dsh
  C^*(\Bcal)\dsh$equivalence bimodule $C^*(\Xcal)$ of Theorem
  \ref{theorem construction of the equivalence bimodule}. 
\end{definition}

It is important to recall that $C^*(\Xcal)$ is the closure of
$C_c(\Xcal)$ in $C^*(\Lb(\Xcal))$ and that we regard $C^*(\Acal)$ and
$C^*(\Bcal)$ as $C^*$-subalgebras of $C^*(\Lb(\Xcal)).$ 
These representations will be used without explicit mention in the
rest of the text. 
\begin{remark}\label{remark:topologies on C_c(X)}
 The inductive limit topology of $C_c(\Xcal), $ $\tau^\Xcal_{ilt},$ contains the topology 
 relative to $\tau^{\Lb(\Xcal)}_{ilt}$ (see the universal property described in \cite[II 14.3]{FlDr88}).
 Besides, the topology of $C_c(\Xcal)$ relative to the norm topology
 of $C^*(\Xcal),$ $\tau^\Xcal_*,$ 
 is the one relative to $\tau^{\Lb(\Xcal)}_*.$
 Since $\tau^{\Lb(\Xcal)}_*\subset \tau^{\Lb(\Xcal)}_{ilt},$ we have $\tau^\Xcal_* \subset \tau^{\Xcal}_{ilt}.$
\end{remark}
\begin{corollary}
  If $\Xcal$ is a right \FH $\Bcal\dsh$bundle and we construct
  $C^*(\Xcal)$ considering $\Xcal$ as a
  $\Kb(\Xcal)\dsh\Bcal\dsh$equivalence bundle, then $\Kb(C^*(\Xcal))$
  is isomorphic to $C^*(\Kb(\Xcal)).$ 
\end{corollary}
\begin{proof}
  By the general theory of Hilbert modules \cite{Raeburn1998morita},
  if ${}_A X_B$ is an $A\dsh B\dsh$equivalence bimodule then $A$ is
  isomorphic to $\Kb(X_B).$ 
  To get the desired result consider the bimodule
  ${}_{C^*(\Kb(\Xcal))}C^*(\Xcal)_{C^*(\Bcal)}.$ 
\end{proof}

The notation adopted for the cross-sectional equivalence bimodule is
justified by the following Corollary of Theorem \ref{theorem
  construction of the equivalence bimodule}. 

\begin{corollary}\label{corollary cross sectional equiv bundle and
    cross-sectional cast algebra of a Fell bundle} 
  If the Fell bundle $\Bcal$ is regarded as a
  $\Bcal\dsh\Bcal\dsh$equivalence bundle (Example \ref{example
    reflexive}) then the cross-sectional equivalence bimodule of
  $\Bcal$ is the cross-sectional $C^*$-algebra of $\Bcal$ (regarded as
  an equivalence bimodule). 
\end{corollary}
\begin{proof}
  Suppose $\Acal: = \Bcal$, and denote $\Xcal$ the bundle
  $\Bcal$ when regarding it as an $\Acal\dsh \Bcal\dsh$equivalence
  bimodule. 
  Then $C^*(\Bcal)$ is the cross-sectional $C^*$-algebra of $\Bcal$
  and $C^*(\Xcal)$ is the cross-sectional equivalence bimodule of
  $\Xcal=\Bcal.$ 
  
  We claim that the identity $\id\colon C_c(\Xcal)\to C_c(\Bcal)$ has
  a unique extension to a unitary $U\colon C^*(\Xcal)\to C^*(\Bcal).$ 
  It suffices to show that, for all $f,g\in C_c(\Bcal),$ the element
  $f^**g$ computed in $C^*(\Lb(\Xcal))$ agrees with $f^**g$ computed
  in $C^*(\Bcal).$ 
  Recall that we may think of $\Bcal$ and $\Xcal$ as Banach subbundles
  of $\Lb(\Xcal).$ 
  To avoid complicated notation we make the following identifications,
  for all $b\in \Bcal$ and $x\in \Xcal,$  
  $$ b =  \left( \begin{array}{cc} 0 & 0\\ 0 & b \end{array} 
 \right)\in \Lb(\Xcal)\qquad \mbox{and}\qquad x =
 \left( \begin{array}{cc} 0 & x\\ 0 & 0 \end{array} 
 \right)\in \Lb(\Xcal).$$
  If we compute $f^**g$ in $C^*(\Lb(\Xcal))$ we obtain, for all $t\in
  G,$ 
  \begin{align*}
   f^**g(t)
    & = \int_G \Delta(s)^{-1}\left( \begin{array}{cc} 0 & 0\\
        \widetilde{f(\smu)} & 0 \end{array} 
 \right)
 \left( \begin{array}{cc} 0 & g(\smu t)\\ 0 & 0 \end{array}
 \right)\, ds\\
    & = \int_G \left( \begin{array}{cc} 0 & 0\\ 0 & f(s)^*g(st) \end{array}
 \right)\, ds
      =  \int_G f(s)^*g(st)\, ds
  \end{align*}
  On the other hand, computing $f^**g$ in $C^*(\Bcal)$ we obtain
  \begin{align*}
   f^**g(t)
    & = \int_G \Delta(s)^{-1}f(\smu)^*  g(\smu t)\, ds
    = \int_G f(s)^*  g(st)\, ds.
  \end{align*}
  Hence the claim follows.
\end{proof}

\begin{corollary}\label{corollary linking algebra of cross-sectional
    equivalence module} 
  If $\Xcal$ is an $\Acal\dsh \Bcal\dsh$equivalence bundle, then
  $\Lb(C^*(\Xcal))$ is isomorphic to $C^*(\Lb(\Xcal)).$ 
\end{corollary}
\begin{proof}  
  Let $C^*(\Xcal)\oplus C^*(\Bcal)$ be considered as a
  Hilbert $C^*(\Bcal)\dsh$module with the inner product $\la x\oplus 
  y,u\oplus v\ra = \la x,u\ra_{C^*(\Bcal)} + y^**v.$ 
  We identify $C_c(\Xcal)\oplus C_c(\Bcal)$ with $C_c(\Xcal\oplus
  \Bcal)$ in the natural way, and represent by $M$ the closure of
  $C_c(\Xcal\oplus \Bcal)$ in $C^*(\Lb(\Xcal)).$  
  Let $U\colon C^*(\Xcal)\oplus C^*(\Bcal)\to M$ be the unitary
  extending the identification $C_c(\Xcal)\oplus
  C_c(\Bcal)=C_c(\Xcal\oplus \Bcal).$  
  
  We claim that, as a Hilbert module, $M$ is $C^*(\Xcal\oplus \Bcal).$
  In fact, if we take $f\in C_c(\Xcal\oplus \Bcal)$ and compute
  $f^**f$ using the product and involution of $C^*(\Lb(\Xcal))$ and of
  $C^*(\Lb(\Xcal\oplus \Bcal)),$ we obtain the same element of
  $C^*(\Bcal).$ 
  Moreover, at the level of $C_c(\Xcal\oplus \Bcal)$ and $C_c(\Bcal),$
  it does not matter whether we use $C^*(\Lb(\Xcal))$ or
  $C^*(\Lb(\Xcal\oplus \Bcal))$ to compute the right inner products
  and the action. 
  Then $M$ is unitary equivalent, as a right Hilbert module, to
  $C^*(\Xcal\oplus \Bcal).$ 
  
  Finally, recall that $\Xcal\oplus \Bcal$ is an $\Lb(\Xcal)\dsh
  \Bcal\dsh$equivalence bimodule, thus we may think of
  $C^*(\Lb(\Xcal))$ as the algebra of generalized compact operators of 
  $C^*(\Xcal)\oplus C^*(\Bcal).$ Thus (up to canonical isomorphisms) 
  $$\Lb(C^*(\Xcal))=\Kb(C^*(\Xcal)\oplus C^*(\Bcal)) = \Kb(C^*(\Xcal
  \oplus \Bcal)) = C^*(\Lb(\Xcal)). $$ 
\end{proof}

\begin{remark}\label{remark view direct sums inside the linking for hilber modules}
  In the proof above we showed that $C^*(\Xcal\oplus \Bcal)$ can be
  regarded as the completion of $C_c(\Xcal\oplus \Bcal)$ in
  $C^*(\Lb(\Xcal)).$ 
  This representation of $C^*(\Xcal\oplus \Bcal)$ will be used instead
  of the representation in $C^*(\Lb(\Xcal\oplus \Bcal)).$ 
  Similar considerations apply for $C^*(\Acal\oplus \Bcal).$
\end{remark}

\subsection{Induction of ideals through cross-sectional Hilbert
  bimodules}\label{section cross-sectional Hilbert modules} 
All the constructions we have carried out can be performed using reduced cross-sectional $C^*$-algebras.
In fact we can use other quotients of the full cross-sectional $C^*$-algebra, as the ones defined in \cite{BssEffMaximality}.
In fact we will give an alternative (and equivalent) way of extendig exotic crossed products to the realm of Fell bundles.

\par Suppose $\mu:\mathscr{F}\to\mathscr{C}$ is a functor, from the
category of Fell bundles to the category of $C^*$-algebras, 
that associates to each Fell bundle $\Bcal$ a quotient 
$C^*_\mu(\Bcal)$ of $C^*(\Bcal)$, such that $C^*_r(\Bcal)$ is in turn
a quotient of $C^*_\mu(\Bcal)$, in such a way that 
the collections of the corresponding quotient maps
$q_\mu^{\Bcal}:C^*(\Bcal)\to C^*_\mu(\Bcal)$ and
$p_\mu^{\Bcal}:C^*_\mu(\Bcal)\to C^*_r(\Bcal)$ are natural
transformations
$C^*\stackrel{q_\mu}{\to}\mu\stackrel{p_\mu}{\to}C^*_r$ satisfying
$pq=\Lambda$, where $\Lambda$ is the regular representation. In other
words, for every morphism $\phi:\Acal\to\Bcal$ 
the following diagram is commutative: 
\begin{equation}\label{eq:crossed product functor diagram}
 \xymatrix{
C^*(\Acal)\ar@/^{1.8pc}/[urrd]^{\Lambda^{\Acal}}
\ar[r]^{q^{\Acal}_\mu}\ar[d]_{\phi^*}&C^*_\mu(\Acal)\ar[r]^{p^{\Acal}_\mu}\ar[d]_{\phi_\mu}&C^*_r(\Acal)\ar[d]^{\phi_r}\\
C^*(\Bcal)\ar@/_{1.8pc}/[urrd]_{\Lambda^{\Bcal}}
\ar[r]_{q^{\Bcal}_\mu}&C^*_\mu(\Bcal)\ar[r]_{p^{\Bcal}_\mu}&C^*_r(\Bcal)
}
\end{equation}
For instance, the functors $C^*$ and $C^*_r$ given by taking the
universal and the reduced cross-sectional algebras respectively,
satisfy the property above. Following 
\cite{buss2015exotic}, we call any
such functor $\mu$ a \textit{crossed product functor}, and we
refer to $C^*_\mu(\Bcal)$ as the $\mu$-crossed product of
$\Bcal$. When only the functors $C^*$ and $\mu$ are involved, as well
as the natural transformation $q$, and the left square is commutative
in the above diagram, we say that $\mu$ a \textit{pseudo crossed product 
functor}, and we refer to $C^*_\mu(\Bcal)$ as the
$\mu$-pseudo crossed product of 
$\Bcal$. Also following \cite{buss2015exotic}, $\mu$ is said to be  
an exotic crossed product functor when it is 
neither the full crossed product functor $C^*$ nor the reduced crossed 
product functor $C^*_r$.  
\par If we define $I^{\Bcal}_\mu$
to be the kernel of $q_\mu^{\Bcal}$, the assignment $\Bcal\mapsto
I^{\Bcal}_\mu$ is another functor $\mathscr{F}\to\mathscr{C}$,  
because the diagram above implies $\phi_*(I^{\Acal}_\mu)\subseteq
I^{\Bcal}_\mu$. Observe that if $\Acal$ is a Fell subbundle of
$\Bcal$ such that $C^*(\Acal)\subseteq C^*(\Bcal)$ (e.g. $\Acal$ is an
hereditary subbundle of $\Bcal$, according to \ref{theorem hereditary
  sub bundle}),  then $C^*_\mu(\Acal)$ is a $C^*$-subalgebra of
$C^*_\mu(\Bcal)$ if and only if $I^\Acal_{\mu}=C^*(\Acal)\cap
I^\Bcal_{\mu}$. We are interested in those functors which satisfy the
above properties for any hereditary subbundle $\Acal$ of $\Bcal$:  
\begin{definition}\label{definition hereditary crossed product}
  A pseudo crossed product functor $\mu$ is said to have the
  hereditary subbundle property if for every Fell
  bundle $\Bcal$ and every hereditary Fell subbundle $\Acal$ of
  $\Bcal,$ it follows that $I^\Acal_\mu = C^*(\Acal)\cap I^\Bcal_\mu$.
%  where we are considering $C^*(\Acal)$ as a $C^*$-subalgebra of
%  $C^*(\Bcal)$ (Theorem \ref{theorem hereditary sub bundle}). 
\end{definition}
%Thus in this case we can regard $C^*_\mu(\Acal)$ as a $C^*$-subalgebra of
%$C^*_\mu(\Bcal).$ 
%Moreover, $\|q^\Acal_\mu(f)\|=\|q^\Bcal_\mu(f)\|$ for all $f\in C^*(\Acal).$
\par Of course the functor $C^*$ has the hereditary subbundle
property. Since, according to \cite[Proposition 3.2]{Ab03},
$C^*_r(\Acal)\subseteq C^*_r(\Bcal)$ for every Fell subbundle 
$\Acal$ of $\Bcal$, also the reduced crossed product functor
$C^*_r$ has the hereditary subbundle property.

%It was shown in \cite[Proposition 3.2]{Ab03} that the reduced crossed product is hereditary and, almost by definition, the full crossed product is hereditary.

\begin{proposition}\label{prop exotic cross sectional Hilbert module}
Let $\mu$ be a pseudo crossed product functor with the hereditary
subbundle property, and $\Xcal$ an $\Acal\dsh
\Bcal\dsh$equivalence bundle. Then
$C^*_\mu(\Xcal):=q^{\Lb(\Xcal)}_\mu(C^*(\Xcal))$ is a
$C^*_\mu(\Acal)\dsh C^*_\mu(\Bcal)\dsh$equivalence bimodule. 
\end{proposition}
\begin{proof}
  Recall from Section \ref{sub section notation} that the image of a
  homomorphism of $C^*$-trings is a $C^*$-tring. 
  Then $C^*_\mu(\Xcal)$ is a $C^*$-subtring of $C^*_\mu(\Lb(\Xcal)).$
  Besides, Theorem \ref{theorem construction of the equivalence
    bimodule} implies $C^*_\mu(\Xcal)$ is a
  $q^{\Lb(\Xcal)}_\mu(C^*(\Acal))\dsh
  q^{\Lb(\Xcal)}_\mu(C^*(\Bcal))\dsh$equivalence bimodule. 
  Finally, since $\Acal$ and $\Bcal$ are hereditary Fell subbundles of
  $\Lb(\Xcal),$ 
  $$q^{\Lb(\Xcal)}_\mu(C^*(\Acal))=C^*_\mu(\Acal) \quad
  \mbox{and}\quad q^{\Lb(\Xcal)}_\mu(C^*(\Bcal))=C^*_\mu(\Bcal),$$ 
  which ends the proof.  
\end{proof}

Crossed product functors give rise to new cross-sectional Hilbert
modules:    

\begin{definition}\label{dfn:mucsm}
 In the conditions and notation of Proposition \ref{prop exotic cross
   sectional Hilbert module}, we say that $C^*_\mu(\Xcal)$ is the
 $\mu\dsh$cross-sectional Hilbert bimodule of $\Xcal$; the map
 $q^\Xcal_\mu\colon C^*(\Xcal)\to C^*_\mu(\Xcal)$ is the restriction
 of $q^{\Lb(\Xcal)}_\mu$ to $C^*(\Xcal)$, and the ideal $I^\Xcal_\mu$
 (of $C^*(\Xcal)$ regarded as a $C^*$-tring) is defined to be
 $\kernel(q^\Xcal_\mu).$   
\end{definition}

\begin{proposition}\label{proposition induction of ideals of crossed
    products} 
 Let $\mu$ be a pseudo crossed product functor with the hereditary subbundle
 property, and suppose $\Xcal$ is an $\Acal\dsh
 \Bcal\dsh$equivalence bundle. 
 Then $C^*(\Xcal)$ induces $I^\Bcal_\mu$ to $I^\Acal_\mu.$
 Moreover, the submodule of $C^*(\Xcal)$ corresponding to these ideals \cite[Theorem~3.22]{Raeburn1998morita} is $I^\Xcal_\mu.$
\end{proposition}
\begin{proof}
 It suffices to show that $I^\Xcal_\mu = I^\Acal_\mu
 C^*(\Xcal)=C^*(\Xcal)I^\Bcal_\mu.$ 
 We prove $I^\Xcal_\mu = C^*(\Xcal)I^\Bcal_\mu$ and leave the
 remaining identity to the reader. 
 \par We regard $C^*(\Bcal)$ as a $C^*$-subalgebra of
 $C^*(\Lb(\Xcal))$ (as we are allowed by Theorem~\ref{theorem
   hereditary sub bundle}).  
 Since $I^\Bcal_\mu = C^*(\Bcal)\cap I^{\Lb(\Xcal)}_\mu$, we have 
 $C^*(\Xcal)I^\Bcal_\mu \subset I^\Xcal_\mu$ and $\la
 I^\Xcal_\mu,I^\Xcal_\mu\ra_{C^*(\Bcal)}$ is contained in
 $I^\Bcal_\mu$.  Since $I^\Xcal_\mu$ is closed we have  
$$I^\Xcal_\mu = I^\Xcal_\mu \spncl \la
I^\Xcal_\mu,I^\Xcal_\mu\ra_{C^*_\mu(\Bcal)}\subset
C^*(\Xcal)I^\Bcal_\mu\subset I^\Xcal_\mu.$$ 
\end{proof}

One can guess that the $\mu-$cross-sectional $C^*$-algebra of a
Fell bundle, considered as a bimodule over itself, is the same as its
$\mu-$cross-sectional equivalence bundle. This is precisely the case:    

\begin{corollary}
  Let $\Bcal$ be a Fell bundle, and denote it $\Xcal$
  when regarded as a $\Bcal\dsh\Bcal\dsh$equivalence bundle.  
  If we identify $C^*(\Xcal)$ with $C^*(\Bcal)$ as in Corollary
  \ref{corollary cross sectional equiv bundle and cross-sectional cast
    algebra of a Fell bundle}, then $I^\Bcal_\mu = I^\Xcal_\mu,$ for
  every pseudo crossed product functor $\mu$ with the hereditary subbundle
  property.   
  In particular, $C^*_\mu(\Xcal)$ is $C^*_\mu(\Bcal)$ regarded as an
  equivalence bimodule. 
\end{corollary}
\begin{proof}
  By Proposition \ref{proposition induction of ideals of crossed
    products}, $C^*(\Xcal)I^\Bcal_\mu = I^\Xcal_\mu.$ 
  Then, considering $C^*(\Xcal)=C^*(\Bcal)$, we get $I^\Xcal_\mu =
  C^*(\Bcal)I^\Bcal_\mu = I^\Bcal_\mu.$ 
\end{proof}

\begin{comment}
With a suitable notion of hereditary (equivalence) bundle, we will be
able to extend the notion of csa-functor replacing the categories of
Fell bundles and $C^*$-algebras respectively by the categories of
\todo{What????} 
equivalence bundles and Hilbert modules. 
\end{comment}

\begin{definition}\label{dfn:cpfext}
  Let $\Xcal$ and $\Ycal$ be an $\Acal\dsh \Bcal$ and a $\Ccal\dsh
  \Dcal\dsh$equivalence bundles respectively. 
  We say that $\Xcal$ is an equivalence subbundle of $\Ycal$ if: 
  \begin{itemize}
   \item $\Xcal$ is a Banach subbundle of $\Ycal,$ $\Acal$ is a Fell
     subbundle of $\Ccal$ and $\Bcal$ a Fell subbundle of $\Dcal.$ 
   \item The equivalence bundle structure of $\Xcal$ agrees with that
     inherited from the equivalence bundle structure of $\Ycal.$  
  \end{itemize}
  Besides, we say $\Xcal$ is hereditary in $\Ycal$ if $\Xcal\la
  \Ycal,\Xcal\ra_{\Dcal}\subset \Xcal$ (note this is equivalent to the
  condition ${}_\Ccal\langle\Xcal,\Ycal\rangle\Xcal\subseteq\Xcal$).    
\end{definition}

\begin{proposition}\label{prop:cpfext}
  Let $\Xcal$ and $\Ycal$ be an $\Acal\dsh \Bcal$ and a $\Ccal\dsh
  \Dcal\dsh$equivalence bundle, respectively, such that $\Xcal$ is an
  equivalence subbundle of $\Ycal$. 
  Then the following are equivalent:
  \begin{enumerate}[(1)]
   \item $\Xcal$ is hereditary in $\Ycal.$
   \item $\Lb(\Xcal)$ is an hereditary Fell subbundle of $\Lb(\Ycal).$ 
  \end{enumerate}
  
  Besides, if $\mu$ is a pseudo crossed product functor with the hereditary subbundle 
  property and the conditions above are satisfied, then $C^*(\Xcal)$
  is 
  (isomorphic to) the closure of $C_c(\Xcal)$ in $C^*(\Ycal)$ and $
  I^\Xcal_\mu = C^*(\Xcal)\cap I^\Ycal_\mu.$ 
  In particular, $C^*_\mu(\Xcal)$ is isomorphic (as a $C^*$-tring) to
  $q^\Ycal_\mu(C^*(\Xcal)).$ 
\end{proposition}
\begin{proof}
  Assume $\Xcal$ is hereditary in $\Ycal.$
  We can regard $\Lb(\Xcal)$ as a Banach subbundle of $\Lb(\Ycal)$
  because every continuous section of $\Xcal$ ($\Acal,\Bcal$) is a
  continuous section of $\Ycal$ ($\Ccal,\Dcal,$ respectively). 
  Besides, the product and involution of $\Lb(\Xcal)$ are the ones
  inherited from $\Lb(\Ycal)$ because they are defined in terms of the
  equivalence bundle structure of $\Xcal,$ which is inherited from
  $\Ycal.$   
  Then $\Lb(\Xcal)$ is a Fell subbundle of $\Lb(\Ycal).$
  
  Now fix $x\in \Lb(\Xcal)\Lb(\Ycal)\Lb(\Xcal).$
  Then  
  \begin{align*}
   x_{1,1}\in \Acal \Ccal \Acal + {}_\Ccal\la  \Xcal,\Ycal\ra \Acal +
   \Acal{}_\Ccal\la \Ycal,\Xcal\ra + {}_\Ccal\la \Xcal
   \Dcal,\Xcal\ra. 
%    \Acal\Ccal\Xcal + \laa \Xcal,\Ycal\raa\Xcal + \Acal\Ycal \Bcal + \Xcal\Dcal\Bcal\\
%    \widetilde{\Acal\Ccal\Xcal} + \widetilde{\Acal\Ycal\Bcal} + \widetilde{\Xcal\la \Ycal,\Xcal\ra_\Dcal} + \widetilde{\Xcal\la \Ycal,\Xcal\ra_\Dcal}
  \end{align*}
  
  Firstly, we claim that $\Acal\Ccal\Acal\subset \Acal.$
  By Lemma \ref{lemma first tools}, it suffices to prove that
  ${}_\Ccal\la \Xcal,\Xcal\ra \Ccal {}_\Ccal\la \Xcal,\Xcal\ra \subset
  \Acal$, which is true because 
  \begin{align*}
   {}_\Ccal\la \Xcal,\Xcal\ra \Ccal {}_\Ccal\la \Xcal,\Xcal\ra
      & = {}_\Ccal\la \Xcal,\Ccal \Xcal\ra  {}_\Ccal\la \Xcal,\Xcal\ra
        \subset   {}_\Ccal\la {}_\Ccal\la \Xcal,\Ycal\ra  \Xcal,\Xcal\ra\\
      & = {}_\Ccal\la \Xcal\la \Ycal,  \Xcal\ra_\Dcal ,\Xcal\ra
          \subset {}_\Ccal\la \Xcal,\Xcal\ra
          \subset \Acal.
  \end{align*}
  Secondly, since the computations above also imply that ${}_\Ccal\la
  \Xcal,\Ycal\ra {}_\Ccal\la \Xcal,\Xcal\ra\subset \Acal$,   
  another invocation of Lemma~\ref{lemma first tools} shows that
  ${}_\Ccal\la  \Xcal,\Ycal\ra \Acal\subset \Acal$ and that
  $\Acal{}_\Ccal\la \Ycal,\Xcal\ra \subset \Acal.$ 
  Finally, we can use again Lemma \ref{lemma first tools}, which
  together with the identity $\Xcal=\Xcal\Acal$ allows us to deduce
  that ${}_\Ccal \Xcal\la \Xcal\Dcal,\Xcal\ra\subset \Acal$, because
  $\Xcal\Dcal\la \Xcal,\Xcal\ra_\Dcal\subset \Xcal.$ 
   \par Putting all together we conclude that $x_{1,1}\in \Acal.$
  Using similar arguments it can be shown that $x_{1,2}\in \Xcal,$
  $x_{2,1}\in \widetilde{\Xcal}$ and $x_{2,2}\in \Bcal.$ 
  Thus $\Lb(\Xcal)$ is hereditary in $\Lb(\Ycal).$
  \par Conversely, every continuous section of $\Xcal$ ($\Acal,\Bcal$)
  is a continuous section of $\Lb(\Xcal)$ and so one of $\Lb(\Ycal).$ 
  This implies that $\Xcal$ ($\Acal,\Bcal$) is a Banach subbundle of
  $\Ycal$ ($\Ccal,\Dcal,$ respectively). 
  Also, observe that the equivalence bundle structure of $\Xcal$ is
  the one inherited from $\Ycal$ because it is defined in terms of the
  product and involution of $\Lb(\Xcal)$, whose operations are
  inherited  from  $\Lb(\Ycal).$ 
  Furthermore, $\Xcal$ is hereditary in $\Ycal$ because $\Xcal\la
  \Ycal,\Xcal\ra_\Dcal\subset [\Lb(\Xcal)\Lb(\Ycal)\Lb(\Xcal)]\cap
  \Ycal\subset \Xcal.$

  If $\Lb(\Xcal)$ is an hereditary Fell subbundle of $\Lb(\Ycal),$
  then we can regard $C^*(\Lb(\Xcal))$ as a $C^*$-subalgebra of
  $C^*(\Lb(\Ycal)).$ 
  Considering the canonical representation of $C^*(\Xcal)$ and
  $C^*(\Ycal)$ in $C^*(\Lb(\Xcal))$ and $C^*(\Lb(\Ycal)),$
  respectively, we conclude that $C^*(\Xcal)$ is a $C^*$-subtring of 
  $C^*(\Ycal).$  
  Thus $q^\Ycal_\mu(C^*(\Xcal))$ is a $C^*$-subtring of
  $C^*_\mu(\Ycal).$ 
  Besides, 
  $$ I^\Xcal_\mu = C^*(\Xcal)\cap I^{\Lb(\Xcal)}_\mu = C^*(\Xcal)\cap
  C^*(\Lb(\Xcal))\cap I^{\Lb(\Ycal)}_\mu = C^*(\Xcal)\cap
  I^{\Lb(\Ycal)}_\mu = C^*(\Xcal)\cap I^\Ycal_\mu. $$ 
  Hence $\kernel(q^\Ycal_\mu|_{C^*(\Xcal)})=I^\Xcal_\mu$ and there
  exists a unique isomorphism of $C^*$-trings $C^*_\mu(\Xcal)\to
  q^\Ycal_\mu(C^*(\Xcal)),\ x+I^\Xcal_\mu \mapsto q^\Ycal_\mu(x).$ 
\end{proof}

\begin{definition}
  Given a pseudo crossed product functor $\mu$, with the hereditary subbundle 
  property, and an equivalence bundle $\Xcal$, we say
  that $\Xcal$ is $\mu$-amenable if $I^\Xcal_\mu=\{0\}$. Similarly, a
  Fell bundle $\Bcal$ will be called $\mu$-amenable if
  $I^\Bcal_\mu=\{0\}$.    
\end{definition}

\begin{corollary}
  Let $\Xcal$ be an $\Acal\dsh \Bcal\dsh$equivalence bundle and $\mu$
  a pseudo crossed product functor with the hereditary property. 
  Then the following are equivalent:
  \begin{enumerate}
   \item $\Xcal$ is $\mu-$amenable.
   \item $\Acal$ is $\mu-$amenable.
   \item $\Bcal$ is $\mu-$amenable.
   \item $\Lb(\Xcal)$ is $\mu-$amenable.
   \item $\Kb(\Xcal)$ is $\mu-$amenable.
  \end{enumerate}
  In particular, $\mu-$amenability is preserved by equivalence of Fell
  bundles. 
\end{corollary}
\begin{proof}
  By Proposition \ref{proposition induction of ideals of crossed products} and the correspondence of ideals via equivalence bimodules, $I^\Acal_\mu=\{0\}$ $\Leftrightarrow$ $I^\Bcal_\mu=\{0\}$ $\Leftrightarrow$ $I^\Xcal_\mu=\{0\}.$
  Then (1), (2) and (3) are equivalent.
  Besides, that equivalence together with Remark \ref{remark view
    direct sums inside the linking for hilber modules} and Corollary
  \ref{corollary right \FH bundles are equivalence bundles} implies
  the last four claims are equivalent to each other. 
\end{proof}

\subsection{Extension of pseudo crossed product functors.} 

\par Every Fell bundle $\Bcal$ can be considered as a
$\Bcal\dsh\Bcal$-equivalence bundle, and every $C^*$-algebra can be
considered as a $C^*$-tring. Then the following
result says that the
crossed product functor %with the hereditary subbundle property
$C^*:\mathscr{F}\to\mathscr{C}$ can be extended to a 
functor $C^*:\mathscr{E}\to\mathscr{T}$ from the category
$\mathscr{E}$ of equivalence bundles to the category $\mathscr{T}$ of
$C^*$-trings.  

\begin{theorem}\label{thm:universalext}
 Assume $\Xcal$ and $\Ycal$ are $\Acal\dsh \Bcal$ and $\Ccal\dsh
 \Dcal\dsh$equivalence bundles over $G,$ respectively. 
 Then for every morphism of equivalence bundles, $\rho\colon \Xcal\to
 \Ycal,$ there exists a unique homomorphism of $C^*$-trings,
 $\rho^*\colon C^*(\Xcal)\to C^*(\Ycal),$ such that 
 $\rho^*(f)=\rho\circ f,$ for all $f\in C_c(\Xcal).$ 
 This homomorphism satisfies
 \begin{enumerate}[(a)]
  \item\label{item Cast and right operations on a morphism commute}
    $(\rho^*)^r=(\rho^r)^*$ and $(\rho^*)^l=(\rho^l)^*$. 
  \item $\Lb(\rho)^*=\Lb(\rho^*),$ under the isomorphism provided by Corollary \ref{corollary linking algebra of cross-sectional equivalence module}.
%  \item If $\rho$ is an isomorphism then $C^*(\rho)$ is an isomorphism
%    and $C^*(\rho)^{-1}=C^*(\rho^{-1}).$
  \item If $\mu$ is a pseudo crossed product functor then $\rho^*(I^\Xcal_\mu)\subset
    I^\Ycal_\mu.$  
 \end{enumerate}

\end{theorem}
\begin{proof}
  Uniqueness is clear because $C_c(\Xcal)$ is dense in $C^*(\Xcal)$
  and every homomorphism of $C^*$-trings is contractive. 
  To prove existence let $\Lb(\rho)\colon \Lb(\Xcal)\to \Lb(\Ycal)$ be
  the morphism given by Theorem \ref{theorem the linking of a
    morphism}. 
  If we think of $C^*(\Xcal)$ as a $C^*$-subtring of
  $C^*(\Lb(\Xcal)),$ then it follows that $\rho\circ f =
  \Lb(\rho)\circ f,$ for all $f\in C_c(\Xcal).$ 
  This implies
  $\Lb(\rho)^*(C^*(\Xcal))=\overline{\Lb(\rho)^*(C_c(\Xcal))}\subset
  \overline{C_c(\Ycal)}=C^*(\Ycal)$, so it is enough to define 
  $\rho^*:=\Lb(\rho)^*|_{C^*(\Xcal)}.$  
   \par Note \eqref{item Cast and right operations on a morphism
     commute} follows from the fact that, for all $f,g\in C_c(\Xcal)$:  
  $$ (\rho^*)^r(f^**g) = (\rho\circ f)^**(\rho\circ g) =
  \Lb(\rho)^*(f^**g) = (\rho^r)^*(f^**g).$$ 
  Similarly $(\rho^*)^l(f*g^*)=(\rho^l)^*(f*g^*).$
  \par To prove the second statement, note first that if 
  $\begin{pmatrix}
    \xi & f\\
    \widetilde{g} & \eta
   \end{pmatrix}\in C_c(\Lb(\Xcal))$, then 
   \[\Lb(\rho)^*\begin{pmatrix}
    \xi & f\\
    \widetilde{g} & \eta
   \end{pmatrix}
   =\Lb(\rho)\circ 
    \begin{pmatrix}
    \xi & f\\
    \widetilde{g} & \eta
   \end{pmatrix}
    =\begin{pmatrix}
    \rho^l\circ \xi & \rho\circ f\\
    \widetilde{\rho\circ g} & \rho^r\circ \eta
     \end{pmatrix}
   =\begin{pmatrix} 
    (\rho^l)^*(\xi) & \rho^*(f)\\
    \widetilde{\rho^*(g)} & (\rho^r)^*(\eta)
    \end{pmatrix}
   ,\] and therefore by (a) we have:   
   \begin{gather*}
   \Lb(\rho)^*\begin{pmatrix}
    \xi & f\\
    \widetilde{g} & \eta
   \end{pmatrix}
   =\begin{pmatrix}
    (\rho^*)^l(\xi) & \rho^*(f)\\
    \widetilde{\rho^*(g)} & (\rho^*)^r(\eta)
   \end{pmatrix}
   =\Lb(\rho^*) \begin{pmatrix}
    \xi & f\\
    \widetilde{g} & \eta
   \end{pmatrix}
\end{gather*}
   Then $\Lb(\rho)^*=\Lb(\rho^*)$ on a dense subset, so they agree.    
   \par As
  for the last statement, we have:  
  $$ \rho^*(I^\Xcal_\mu) = \Lb(\rho)^*( C^*(\Xcal)\cap
  I^{\Lb(\Xcal)}_\mu ) \subset C^*(\Ycal)\cap I^{\Lb(\Ycal)}_\mu =
  I^\Ycal_\mu.$$ 
\end{proof}

\begin{corollary}\label{cor:extmu}
Let $\mu:\mathscr{F}\to\mathscr{C}$ be a pseudo crossed product functor with
the hereditary subbundle property. Then $\mu$
can be extended to a functor $\mu:\mathscr{E}\to\mathscr{T}$, from the category
of equivalence bundles to the category of $C^*$-trings. 
\end{corollary}
\begin{proof}
Given an $\Acal-\Bcal\dsh$equivalence bundle $\Xcal$, we have
$C^*_\mu(\Xcal)=C^*(\Xcal)/I_\mu^\Xcal$, where $I^\Xcal_\mu = I^\Acal_\mu 
 C^*(\Xcal)=C^*(\Xcal)I^\Bcal_\mu$ (recall \ref{dfn:mucsm} and the
 proof of \ref{proposition induction of ideals of crossed products}). 
\par Now if $\rho:\Xcal\to\Ycal$ is a morphism of equivalence bundles, 
we have $\rho^*(I^\Xcal_\mu)\subseteq I^\Ycal_\mu$ by the third
statement of  \ref{thm:universalext}, so $\rho^*$ induces a 
unique homomorphism of $C^*$-trings $\rho_\mu:C^*_\mu(\Xcal)\to
C^*_\mu(\Ycal)$. It is easy to check that
$\Big(\Xcal\stackrel{\rho}{\to}\Ycal\Big)\longmapsto
\Big(C^*_\mu(\Xcal)\stackrel{\rho_\mu}{\to}C^*_\mu(\Ycal)\Big)$ is a
functor that extends $\mu$. 
\end{proof}
\begin{remark}  Note that the extended functor also has
the hereditary subbundle property, in the sense that if $\Xcal$ is
hereditary in $\Ycal$, then $C^*_\mu(\Xcal)$ is hereditary in
$C^*_\mu(\Ycal)$ (recall Definition~\ref{dfn:cpfext}, and also note
that any (positive) $C^*$-tring can be thought of as an equivalence bundle over
the trivial group).     
\end{remark}

\begin{remark}
 After Corollary \ref{cor:extmu} a natural question arises: given the
 extension of a crossed product functor with the hereditary subbundle
 property, $\mu,$ do we obtain commutative diagrams like
 \eqref{eq:crossed product functor diagram} (see diagram
 \eqref{eq:commutative diagram equivalence bundles} below) if we
 consider morphisms of equivalence bundles instead of morphism of Fell
 bundles? 
 To answer this question affirmatively we need to choose the
 map $p^\Xcal_\mu \colon C^*_\mu(\Xcal)\to C^*_r(\Xcal),$ for every
 equivalence bundle $\Xcal.$ 
 
 Given a morphism of equivalence bundles, $\phi\colon \Xcal\to \Ycal, $ the commutative diagram associated to $\Lb(\phi)\colon \Lb(\Xcal)\to  \Lb(\Ycal)$ is
\begin{equation}\label{eq:commutative diagram for linking bundles}
\xymatrix{
C^*(\Lb(\Xcal))
  \ar@/^{1.8pc}/[urrd]^{q^{\Lb(\Xcal)}_r=\Lambda^{\Lb(\Xcal)}}
  \ar[r]^{q^{\Lb(\Xcal)}_\mu}\ar[d]_{\Lb(\phi)^*} &
C^*_\mu(\Lb(\Xcal))
  \ar[r]^{p^{\Lb(\Xcal)}_\mu}
  \ar[d]_{\Lb(\phi)_\mu} &
C^*_r(\Lb(\Xcal))
  \ar[d]^{\Lb(\phi)_r}\\
C^*(\Lb(\Ycal))
  \ar@/_{1.8pc}/[urrd]_{q^{\Lb(\Ycal)}_r = \Lambda^{\Lb(\Ycal)}}
  \ar[r]_{q^{\Lb(\Ycal)}_\mu} &
C^*_\mu(\Lb(\Ycal))
  \ar[r]_{p^{\Lb(\Ycal)}_\mu} &
C^*_r(\Lb(\Ycal))
}
\end{equation}

Since $C^*_\bullet(\Xcal)=q^\Xcal_\bullet (C^*(\Xcal)),$ for $\bullet = r,\mu,$ we have
$$p^{\Lb(\Xcal)}_\mu(C^*_\mu(\Xcal)) = p^{\Lb(\Xcal)}_\mu\circ q^{\Lb(\Xcal)}_\mu(C^*(\Xcal))=  q^{\Lb(\Xcal)}_r(C^*(\Xcal))=C^*_r(\Xcal).$$
Then it is natural to define $p^\Xcal_\mu\colon C^*_\mu(\Xcal)\to
C^*_r(\Xcal)$ to be the restriction of $p^{\Lb(\Xcal)}_r$ to
$C^*_\mu(\Xcal)$. 
Adding to the commutative diagram \eqref{eq:commutative diagram for
  linking bundles} the fact that $\phi_\bullet \colon
C^*_\bullet(\Xcal)\to C^*_\bullet(\Ycal)$ is the corresponding
restriction of $\Lb(\phi)_\bullet$ (for $\bullet = r,\mu$) we obtain
the commutative diagram  
\begin{equation}\label{eq:commutative diagram equivalence bundles}
\xymatrix{
C^*(\Xcal)\ar@/^{1.8pc}/[urrd]^{q^{\Xcal}_r}
\ar[r]^{q^{\Xcal}_\mu}\ar[d]_{\phi^*}&C^*_\mu(\Xcal)\ar[r]^{p^{\Xcal}_\mu}\ar[d]_{\phi_\mu}&C^*_r(\Xcal)\ar[d]^{\phi_r}\\
C^*(\Ycal)\ar@/_{1.8pc}/[urrd]_{q^{\Ycal}_r}
\ar[r]_{q^{\Ycal}_\mu}&C^*_\mu(\Ycal)\ar[r]_{p^{\Ycal}_\mu}&C^*_r(\Ycal)
},
\end{equation}
which is the diagram mentioned in our question.
\end{remark}

\section{Internal tensor products and transitivity}
\par This last section is devoted to proving that equivalence of Fell
bundles is an equivalence relation. To this end we will define internal tensor products of Fell bundles. 

\par Suppose we are given an $\Acal\dsh \Bcal\dsh$equivalence bundle, $\Xcal$, and a 
$\Bcal\dsh \Ccal\dsh$equivalence bundle, $\Ycal.$  
Then we can form the internal tensor product
$C^*(\Xcal)\otimes_{C^*(\Bcal)}C^*(\Ycal),$ which establishes a
Morita-Rieffel equivalence between $C^*(\Acal)$ and $C^*(\Ccal).$ 
One could expect this equivalence to come from an $\Acal\dsh \Ccal$-equivalence 
bundle $\Zcal$ in such a way that $C^*(\Zcal)$ is isomorphic %unitarly
                                %equivalent 
to $C^*(\Xcal)\otimes_{C^*(\Bcal)}C^*(\Ycal).$ 
This bundle $\Zcal$ should then be denoted $\Xcal\otimes_\Bcal\Ycal,$
for then we would have 
$$ C^*(\Xcal\otimes_\Bcal\Ycal) = C^*(\Xcal)\otimes_{C^*(\Bcal)}C^*(\Ycal).$$
\par In what follows we will construct such a bundle $\Zcal$. 
The construction is a bit complicated, and will be done along several
steps. The final part of the process 
will consist in obtaining an equivalence bundle from a kind of a pre
equivalence bundle. The following two results will serve to this
purpose.   

\begin{proposition}\label{theorem construction of right bundle}
  Let $\Bcal$ be a Fell bundle over $G.$
  Assume there is a bundle of normed vector spaces
  $\Xcal:=\{X_t\}_{t\in G},$ sets of sections $\Gamma_\Xcal$ and
  $\Gamma_\Bcal$ (of $\Xcal$ and $\Bcal$ respectively) and maps 
  \begin{equation}\label{equ B operations}
    \Xcal\times \Bcal\to \Xcal: \ (x,b)\mapsto xb,\mbox{ and }
    \Xcal\times \Xcal\to \Bcal: \ (x,y)\mapsto \lab x,y\rab, 
  \end{equation}
  with the following properties:
  \begin{enumerate}
   \item Conditions (1R)-(7R) from Definition \ref{definition right
       \FH} hold. 
   \item $\Gamma_\Xcal$ is a complex vector space with pointwise
     operations and $\Gamma_\Bcal\subset C_c(\Bcal)$. 
    \item For all $t\in G$ and $\Fcal\in \{\Bcal,\Xcal\},$
      $\{u(t)\colon u\in \Gamma_{\Fcal}\}$ is dense in $F_t$, where
      the norm considered on $X_t$ is $\|x\|=\|\lab x,x\rab\|^{1/2}.$ 
   \item\label{item continuity of norms of operations} For all $u,v\in
     \Gamma_\Xcal$ and $f,g\in \Gamma_\Bcal$ the maps 
   $$G\times G\to \R,\  (r,s)\mapsto \| \lab u(r),v(s)\rab - f(\rmu s)
   \|$$ 
   $$G\times G\to B_e,\  (r,s)\mapsto \lab u(r)f(s) - g(rs),u(r)f(s) -
   g(r s)\rab,$$ 
   and $G\to B_e, \ r\mapsto \lab u(r),u(r)\rab ,$ are continuous.
  \end{enumerate}

  Then there exists a unique right \FH $\Bcal\dsh$bundle
  $\overline{\Xcal}=\{\overline{X}_t\}_{t\in G}$ such that 
  \begin{itemize}
   \item For all $t\in G,$ $\overline{X}_t$ is the completion of $X_t.$
   \item $\Gamma_\Xcal$ is a set of continuous sections of $\overline{\Xcal}.$
   \item The inner product and action of $\overline{\Xcal}$ extend those of $\Xcal.$ 
  \end{itemize}
\end{proposition}
\begin{proof}
  First note that each fiber $X_t$ is a right pre Hilbert 
  $B_e\dsh$module with positive definite inner product $(x,y)\mapsto
  \lab x,y\rab.$ 
  %Then the unique norm to consider is $\|x\|:=\|\lab x,x\rab\|^{1/2}.$
  \par Given $t\in G,$ let $\overline{X}_t$ be the completion of
  $X_t$, and consider $\overline{\Xcal}=\{\overline{X}_t\}_{t\in
    G}$ as an untopologized bundle over $G.$ 
  It follows from \cite[II 13.18]{FlDr88} and condition \eqref{item continuity of norms of operations} that there exists a unique Banach bundle structure on $\overline{\Xcal}$ such that $\Gamma_\Xcal$ is a set of continuous $\overline{\Xcal}$ sections. 
\begin{comment}
   \par The action of $\Bcal$ and the inner product will be uniquely
  extended by continuity because, given $x,y\in \Xcal$ and $b\in
  \Bcal,$ we have $ 
  \|xb\|\leq \|b\|\|x\|$ and $ \|\lab x,y\rab \|\leq \|x\|\|y\|$ (to
  prove this just reproduce the arguments exposed in the proof of
  Lemma \ref{lemma first tools}). 
\end{comment}
  Using linearity and continuity arguments we can easily prove that
  the action of $\Bcal$ on $\Xcal$, as well as the inner product, can be
  extended in a unique way to an action and an inner product on
  $\overline{\Xcal}.$ 
  The same sort of arguments can be used to prove these new operations satisfy conditions (1R)-(7R) from Definition~\ref{definition right \FH}.
  Finally, by Proposition \ref{prop continuity} and condition \eqref{item continuity of norms of operations}, the inner product and the action are continuous.
\end{proof}

\par We also have a bilateral version of the previous result: we just
need to show that the norms coming from the left and right structures
agree: 

\begin{proposition}\label{theorem construction of equivalence bundle}
  Let $\Acal$ and $\Bcal$ be a Fell bundles over $G.$
  Assume there is a bundle of complex vector spaces $\Xcal:=\{X_t\}_{t\in G},$ sets of sections $\Gamma_\Fcal$ for $\Fcal\in \{\Acal,\Bcal,\Xcal\}$ and maps
  \begin{equation}
    \Acal\times \Xcal\to \Xcal, \ (a,x)\mapsto ax,\ \Xcal\times \Xcal\to \Acal, \ (x,y)\mapsto \laa x,y\raa
  \end{equation}
  \begin{equation}
    \Xcal\times \Bcal\to \Xcal, \ (x,b)\mapsto xb,\ \Xcal\times \Xcal\to \Bcal, \ (x,y)\mapsto \lab x,y\rab
  \end{equation}
  with the following properties:
  \begin{enumerate}
   \item Conditions (1R)-(5R), (1L)-(5L), (7R) and (7L) of
     Definitions~\ref{definition right \FH} hold and, for all
     $x,y,z\in \Xcal,$  
   $\laa x,y\raa z=x\lab y,z\rab.$
   \item $\Gamma_\Xcal$ is a complex vector space with pointwise
     operations and $\Gamma_\Fcal\subset C_c(\Fcal),$ for $\Fcal\in
     \{\Acal,\Bcal\}.$ 
   \item For all $t\in G$ and $\Fcal \in \{\Acal,\Bcal,\Xcal\},$
     $\{u(t)\colon u\in \Gamma_{\Fcal}\}$ is dense in $F_t$, where the
     norm considered on $X_t$ is $\|x\|_\Bcal=\|\lab x,x\rab\|^{1/2}$.  
   \item Condition (4) of Proposition~\ref{theorem construction of
       right bundle} holds and, analogously, for all $u,v\in
     \Gamma_\Xcal$ and $f,g\in \Gamma_\Acal$ the maps 
   $$G\times G\to \R,\ (r,s)\mapsto \| \laa u(r),v(s)\raa - f(\rmu s) \|$$
   $$G\times G\to A_e,\  (r,s)\mapsto \laa u(r)f(s) - g(\rmu
   s),u(r)f(s) - g(\rmu s)\raa,$$ 
   and $G\to A_e, \ r\mapsto \laa u(r),u(r)\raa,$ are continuous.
  \end{enumerate}
  \par Then there exists a unique $\Acal\dsh \Bcal\dsh$equivalence bundle
  $\overline{\Xcal}=\{\overline{X}_t\}_{t\in G}$ such that 
  \begin{itemize}
   \item For all $t\in G,$ $\overline{X}_t$ is the completion of $X_t.$
   \item $\Gamma_\Xcal$ is a set of continuous sections of $\overline{\Xcal}.$
   \item The inner products and actions of $\overline{\Xcal}$ extend those of $\Xcal.$ 
  \end{itemize}
\end{proposition}
\begin{proof}
  Forgetting the left structure, and defining
  $\|x\|_\Bcal=\|\lab x,x\rab\|^{1/2}$, $\forall x\in\Xcal$,  we are
  in the hypotheses of 
  Proposition~\ref{theorem construction of right bundle}. 
  Let $\overline{\Xcal}$ be the right \FH $\Bcal\dsh$bundle given by
  Proposition~\ref{theorem construction of right bundle}. We will show
  that this bundle can be made into an 
  $\Acal\dsh \Bcal\dsh$equivalence bundle preserving the right
  structure. 
 \par Given $t\in G$ define $I_t:=\spn \laa X_t,X_t\raa\subseteq A_e.$ 
  Note that conditions (1L)-(5L) imply $I_t$ is an algebraic $*$-ideal 
  of $A_e.$ The compatibility of the left and right structure on
  $\Xcal$ ensures that given $a\in I_t$ there exists a unique operator
  $\rho_t(a)\in \Kb(\overline{X_t})$ such that $\rho_t(a)x=ax,$ for
  all $x\in X_t.$ 
  Then there exists a unique *-homomorphism $\phi_t\colon I_t\to \Kb(\overline{X_t}),\ a\mapsto \rho_t(a).$
  According to \cite[VI 19.11]{FlDr88}  the homomorphism $\phi_t$ is norm continuous, so it has a unique extension to a *-representation of the $C^*$-ideal $\overline{I_t}.$ 
  As this last representation is contractive, for all $x\in X_t$ we have
  $$ \| \lab x,x\rab \| = \| {}_{\Kb(\overline{X_t})}\la x,x\ra \| = \|\rho_t( \laa x,x\raa )\|\leq \|\laa x,x\raa\|. $$
  
  Now reverse the arguments: take the bundle of complex conjugate normed spaces $\widetilde{\Xcal}=\{\widetilde{X_\tmu}\}_{t\in G}$ with the natural action of $\Acal$ on the right and the $\Acal\dsh$valued inner product, letting $\Bcal$ act on the left.
  Then we conclude that for all $x\in \Xcal:$
  $$ \|\laa x,x\raa \|=\| \la \widetilde{x},\widetilde{x}\ra_{\Acal} \|\leq \|{}_\Bcal\la \widetilde{x},\widetilde{x}\ra\|=\|\lab x,x\rab\|. $$
  
  Moreover the adjoint bundle of
  $\overline{\widetilde{\Xcal}}$ is equal to $\overline{\Xcal}$ (as a Banach bundle).   
  Then $\overline{\Xcal}$ is, at the same time, a left \FH
  $\Acal\dsh$bundle and a right \FH $\Bcal\dsh$bundle. 
  We also know that $\laa x,y\raa z=x\lab y,z\rab$ holds for all
  $x,y,z\in \Xcal$, and a simple continuity argument implies the same 
  identity also holds for all $x,y,z\in \overline{\Xcal}.$   
\end{proof}
 
\subsection{A tensor product of equivalence bundles}\label{subsection a tensor product of bundles}
Fix, for the rest of this section, three Fell bundles $\Acal,\Bcal$
and $\Ccal$ over $G$, an $\Acal\dsh \Bcal\dsh$equivalence bundle,
$\Xcal,$ and a $\Bcal\dsh \Ccal\dsh$equi\-va\-lence bundle, $\Ycal.$ 
From these data we want to construct an $\Acal\dsh 
\Ccal\dsh$equivalence bundle.  
\par There are three bundles we will use in our construction of the $\Acal\dsh 
\Ccal\dsh$equivalence bundle. 
First we define a bundle $\Zcal$ over $G\times G$, whose fibers are
tensor products of the form $X_r\otimes_{B_e}Y_s.$ 
Then we construct a bundle $\Ucal$ over $G$ by defining $U_t$ as the
set of continuous sections of compact support of the reduction of
$\Zcal$ to $ \{(r,s)\in G\times G\colon rs=t\}.$ 
Finally, the fibers of the bundle $[\Ucal]$ will be quotients of the
fibers of $\Ucal.$ 
The desired $\Acal\dsh\Ccal\dsh$equivalence bundle will then be obtained from
$[\Ucal]$ using Proposition~\ref{theorem construction of equivalence
  bundle}. 
\bigskip
\paragraph{{\bf The Banach bundle $\mathbf{\Zcal}$.}}{\mbox{\quad}} 
\par Given $(r,s)\in G\times G$ let $Z_{(r,s)}$ be the
Hilbert $C_e\dsh$module $X_r\otimes_{B_e}Y_s$ and consider the bundle
$\Zcal:=\{Z_w\}_{w\in G\times G}.$ 
We want to endow $\Zcal $ with a Banach bundle structure such that every function of the form
$f\boxtimes g\colon G\times G\to \Zcal,$ $f\boxtimes
g(r,s)=f(r)\otimes g(s)$  ($f\in C_c(\Xcal)$, $g\in
C_c(\Ycal)$), is a continuous section of $\Zcal$. To this end 
we use \cite[II 13.18]{FlDr88}.
Let $\Gamma_\Zcal:=\spn \{ f\boxtimes g\colon f\in C_c(\Xcal),\ g\in
C_c(\Ycal) \}.$ 
It is clear that for each $(r,s)\in G\times G$ the set $\{ \xi(s,t)\colon
\xi\in \Gamma_\Zcal \}$ is dense in $X_r\otimes Y_s$.  
For $\xi = \sum_{j=1}^{n} f_j\boxtimes g_j\in \Gamma_\Zcal,$ the
map $(r,s)\mapsto \|\xi(r,s)\|$ is continuous because so is the
map  
$$ G\times G \to C_e,\ (r,s)\mapsto \la \xi(r,s),\xi(r,s)\ra_{C_e} =
\sum_{j,k=1}^n \lac g_j(s),\lab f_j(r),f_k(s)\rab g_k(s)\rac.$$ 
Then, by \cite[II 13.18]{FlDr88}, there exists a unique
topology on $\mathcal{Z}$ making it a Banach bundle such that every
$f\boxtimes g$ is a continuous section of $\Zcal$.   
\bigskip
\paragraph{{\bf The bundle $\mathbf{\Ucal}$.}}{\mbox{\quad}}  
\par To construct $\Ucal$, for each $t\in G$ let $\Zcal^t$ be the
reduction \cite[II 13.3]{FlDr88} of $\Zcal$ to $H_t:=\{ (r,s)\in
G\times G\colon rs=t \}$ and set $U_t:=C_c(\Zcal^t)$. 
Let $\Ucal$ be the (untopologized) bundle $\{U_t\}_{t\in G}.$
Every section $\xi\in C_c(\Zcal)$ defines a section of $\Ucal,$
$\xi|\colon G\to \Ucal$, given by  $t\mapsto \xi|_t,$ where $\xi|_t$
is the restriction of $\xi$ to $H_t.$   

\begin{lemma}\label{lemma density in Ut}
  For each $t\in G$ and $u\in U_t,$ there exists a compact set
  $K\subset G$ such that: for all $\varepsilon >0$ there exists
  $\xi\in \Gamma_\Zcal$ with $\|u-\xi|_t\|_\infty<\varepsilon$ and
  $\supp(\xi)\subset K\times K.$ 
  In particular, $\{\xi|_t\colon \xi\in \Gamma_\Zcal\}$ is dense in
  $U_t$ in the inductive limit topology of $C_c(\Zcal^t).$ 
\end{lemma}
\begin{proof}
  From Tietze's Extension Theorem for Banach bundles \cite[II
  14.8]{FlDr88} we know that there exists $\eta\in C_c(\Zcal)$ such that
  $\eta|_t = u.$ 
  Since $C_c(G)\otimes C_c(G)\Gamma_\Zcal\subset \Gamma_\Zcal,$ we can
  use \cite[Lemma 5.1]{Ab03} to deduce that $\Gamma_\Zcal$ is dense in
  $C_c(\Zcal)$ in the inductive limit topology. 
  Thus there exists a net $\{ \xi_j \}_{j\in J}\subset \Gamma_\Zcal$
  converging to $\eta$ in the inductive limit topology, and so
  uniformly on compact sets. 
  Take a compact set $K\subset G$ such that $K\times K$ contains the
  support of $\eta$ in its interior, and take $\phi\in C_c(G)$ with:
  $0\leq \phi \leq 1,$ $\phi\otimes\phi|_{\supp \eta}\equiv 1$ and
  $\phi|_{G\setminus K}\equiv 0.$ 
  Then $\{\phi\otimes \phi \xi_j\}_{j\in J}\subset \Gamma_\Zcal$, and
  $\{(\phi\otimes \phi \xi_j)|_t\}_{j\in J}$ converges uniformly to
  $\eta|_t=u.$ 
  Then there exists $j_0\in J$ such that $\| (\phi\otimes \phi
  \xi_{j_0})|_t -u\|<\varepsilon.$ 
  Finally, note that $\supp (\phi\otimes \phi \xi_{j_0})\subset K\times K.$
\end{proof}

The $\Ccal\dsh$valued inner product of $\Ucal,$ and so the seminorm of
$\Ucal,$ will be described as the integral of a kind of inner product
defined on $\Zcal,$ which we now construct. Recall that a map between
vector bundles is said to be quasi-linear if it is linear when
restricted to each fiber \cite[page 790]{FlDr88}. Quasi-bilinear
maps are defined analogously. 

\begin{lemma}
  There exist unique continuous maps  
  $$ \Zcal\times \Zcal \to \Ccal:\ (u,v)\mapsto u\triangleright v
  \mbox{ and }\ \Zcal \times \Zcal \to \Acal:\ (u,v)\mapsto
  u\triangleleft v,\ $$ 
  such that:
  \begin{enumerate}
   \item $\triangleright$ ($\triangleleft$) is quasi-linear in the second
     (first) variable and conjugate quasi-linear in the first (second)
     variable. 
   \item $Z_{(r,s)}\triangleright Z_{(p,q)}\subset C_{(rs)^{-1}pq}$ and $Z_{(r,s)}\triangleleft Z_{(p,q)}\subset A_{rs(pq)^{-1}},$ for all $r,s,p,q\in G.$
   \item $\|u\triangleright v\|\leq \|u\|\|v\|$ and $\|u\triangleleft v\|=\|u\|\|v\|,$ for all $u,v\in \Zcal.$
   \item $(x\otimes y)\triangleright (z\otimes w) = \la y,\lab x,z\rab
     w\ra_{\Ccal}$ and $(x\otimes y)\triangleleft (z\otimes w) = \laa
     x{}_\Bcal\la y,w\ra ,z\raa,$ for all $x\otimes y,z\otimes w\in \Zcal.$
  \end{enumerate}
\end{lemma}
\begin{proof}
  Take $u=\sum_{i=1}^n x_i\otimes y_i\in Z_{(r,s)}$ and
  $v=\sum_{j=1}^n z_j\otimes w_j\in Z_{(p,q)}$. To satisfy (1) and
  (4),  $u\triangleright v$ must be given by: $u\triangleright v:= 
  \sum_{j,k=1}^{n}  \lac y_j,\lab x_j,z_k\rab w_k\rac$. To see that
  this is really a definition, it suffices to show that
  $\norm{\sum_{j,k=1}^{n}  \lac y_j,\lab x_j,z_k\rab
    w_k\rac}\leq\norm{u}\,\norm{v}$.   
  Now if $\{ a_\lambda \}_{\lambda \in \lambda}$ is an approximate
  unit of $A_e$ as the one given by Lemma \ref{lemma approximate
    units}, then 
  \begin{equation}\label{equation norm of inner product}
    \|\sum_{j,k=1}^{n}  \lac y_j,\lab x_j,z_k\rab w_k\rac\|
       = \lim_\lambda \|\sum_{j,k=1}^{n}  \lac y_j,\lab a_\lambda x_j,z_k\rab w_k\rac\|
  \end{equation}
  where $a_\lambda = \sum_{l=1}^{n_\lambda}  \laa \xi^\lambda_l,\xi^\lambda_l\raa,$ for some $\xi^\lambda_l\in X_{t^\lambda_l}$ ($l=1,\ldots,n_\lambda$).
  
  From Lemma \ref{lemma first tools} it follows that
  \begin{equation}\label{equation rewrite inner products}
    \begin{split}
    \lac y_j,\lab a_\lambda x_j,z_k\rab w_k\rac 
%      & = \lac y_j,\lab \laa \xi^\lambda_i,\xi^\lambda_i\raa x_j,z_k\rab w_k\rac
%         = \lac y_j,\lab \xi^\lambda_i\lab \xi^\lambda_i, x_j\rab,z_k\rab w_k\rac\\
     & = \sum_{l=1}^{n_\lambda}\lac \lab \xi^\lambda_l, x_j\rab y_j,\lab \xi^\lambda_l,z_k\rab w_k\rac.
  \end{split}
  \end{equation}

  Given $l=1,\ldots,n_\lambda,$ for all $j=1,\ldots,n$ we have $\lab
  \xi^\lambda_l, x_j\rab y_j\in Y_{(t^\lambda_l)^{-1}rs}$ and $\lab
  \xi^\lambda_l, z_j\rab w_j\in Y_{(t^\lambda_l)^{-1}pq}.$ 
  Define
  $$\eta^\lambda_l:= \sum_{j=1}^n \lab \xi^\lambda_l,x_j\rab y_j\qquad
  \zeta^\lambda_l:=\sum_{j=1}^n \lab \xi^\lambda_l,z_j\rab w_j.$$ 
  From \eqref{equation rewrite inner products} and \eqref{equation norm of inner product} we obtain
  \begin{equation}\label{norm of triangleleft as limit}
   \norm{\sum_{j,k=1}^{n}  \lac y_j,\lab x_j,z_k\rab
    w_k\rac}=\lim_\lambda \|\sum_{k=1}^n \lac \eta^\lambda_k,\zeta^\lambda_k\rac\|. 
  \end{equation}
  
  When $v=u$ we get $\|u\|^2= \norm{\sum_{j,k=1}^{n}  \lac
    y_j,\lab x_j,x_k\rab y_k\rac} =\lim_\lambda \|\sum_{k=1}^n \lac
  \eta^\lambda_k,\eta^\lambda_k\rac\|$ and, analogously, we obtain $\|
  v\|^2=\lim_\lambda \|\sum_{k=1}^n \lac
  \zeta^\lambda_k,\zeta^\lambda_k\rac\|.$ 
  Therefore the inequality $\norm{\sum_{j,k=1}^{n}  \lac y_j,\lab x_j,z_k\rab
    w_k\rac}\leq \|u\|\|v\|$ follows from the inequality
  \begin{equation*}
     \|\sum_{k=1}^{n_\lambda} \lac
     \eta^\lambda_k,\zeta^\lambda_k\rac\|\leq \|\sum_{k=1}^{n_\lambda}
     \lac \eta^\lambda_k,\eta^\lambda_k\rac\|\|\sum_{k=1}^{n_\lambda}
     \lac \zeta^\lambda_k,\zeta^\lambda_k\rac\|, 
  \end{equation*}
  which holds (for all $\lambda$) by Lemma \ref{lemma approximate units}.
  
  To show $\triangleright$ is continuous we use Proposition \ref{prop
    continuity}. 
  Take $\xi,\eta\in \Gamma_\Zcal$ and $f\in C_c(\Ccal).$
  It suffices to show that $(r,s,p,q)\mapsto \| \xi(r,s)\triangleright
  \eta(p,q) - f( (rs)^{-1}pq ) \|$ is continuous. 
  But, since $(r,s,p,q))\mapsto f( (rs)^{-1}pq )$ is continuous, it
  suffices to show that $(r,s,p,q)\mapsto \xi(r,s)\triangleright
  \eta(p,q)$ is continuous. 
  It is enough to consider $\xi=f\boxtimes g$
  and $\eta=h\boxtimes k.$ 
  In this case we have $ \xi(r,s)\triangleright \eta(p,q) = \lac
  g(s),\lab f(r), h(p)\rab k(q)\rac, $ which is clearly a continuous
  function of $(r,s,p,q).$ 
  
  The existence of the operator $\triangleleft$ can be inferred from
  the previous arguments applied to the adjoint bundles of $\Xcal$ and
  $\Ycal.$ 
\end{proof}

Now we define two maps we will use to construct the actions of
$\Acal$ and $\Ccal$ on the (still not precisely defined) bundle
$[\Ucal].$  

\begin{lemma}
  There are unique quasi-bilinear and continuous maps
  $$ \Acal\times \Zcal\to \Zcal,\ (a,z)\mapsto az,\quad
  \mbox{and}\quad \Zcal\times \Ccal\to \Zcal,\ (z,c)\to zc, $$ 
  such that
  \begin{enumerate}
   \item $A_rZ_{(s,t)}\subset Z_{(rs,t)}$ and $Z_{(s,t)}C_r\subset
     Z_{(s,tr)},$ for all $r,s,t\in G.$ 
   \item $\|az\|\leq \|a\|\|z\|$ and $\|zc\|\leq \|z\|\|c\|,$ for all
     $a\in \Acal,$ $z\in \Zcal$ and $c\in \Ccal.$ 
   \item $a(x\otimes y)=(ax)\otimes y$ and $(x\otimes y)c=x\otimes
     (yc),$ for all $a\in \Acal,$ $x\in \Xcal,$ $y\in \Ycal$ and $c\in
     \Ccal.$ 
  \end{enumerate}
\end{lemma}
\begin{proof}
  Take $u=\sum_{i=1}^n x_i\otimes y_i\in Z_{(r,s)}$ and $a\in \Acal.$
  Recall that there exists a natural representation of $A_e,$
  $\psi\colon A_e\to \Bb(Z_{(r,s)}),$ such that $\psi(b)(x\otimes y) =
  (bx)\otimes y.$ 
  Now observe that
  \begin{align*}
   \| \sum_{i=1}^n (ax_i)\otimes y_i \|^2 
      & = \| \sum_{i,j=1}^n \lac y_i, \la ax_i,ax_j\ra_\Bcal y_j\rac \|
        = \| \sum_{i,j=1}^n \lac y_i, \la x_i,a^*ax_j\ra_\Bcal y_j\rac \|\\
      &= \|\la \psi(a^*a)u,u\ra_{C_e}\|\leq \|a\|^2\|u\|^2.
  \end{align*}
  On the other hand, for every $c\in \Ccal$ we have 
  \begin{align*}
   \| \sum_{i=1}^n x_i\otimes (y_ic) \|^2 
      & = \| \sum_{i,j=1}^n \lac y_ic, \la x_i,x_j\ra_\Bcal y_jc\rac \|
        = \|c^* \sum_{i,j=1}^n \lac y_i, \la x_i,x_j\ra_\Bcal y_j\rac c \|\\
      &\leq \|c\|^2\|\la u,u\ra_{C_e}\|=\|c\|^2\|u\|^2.
  \end{align*}
  
  With these inequalities we can define the left action of $\Acal$ and
  the right action of $\Ccal$ on $\Zcal$ on each product of fibers
  ($A_r\times Z_{(s,t)}$ and $Z_{(s,t)}\times C_r$). 
  To prove that the resulting map is continuous it suffices to use
  Proposition \ref{prop continuity} (adapting the arguments we gave
  during the construction of $\triangleright $ and $\triangleleft$). 
\end{proof}

The identities we prove in the following Lemma will be used to show
the compatibility of the left and right \FH bundle structures of our
$\Acal \dsh \Ccal\dsh$equivalence bundle. 

\begin{lemma}\label{lemma some identities}
  For all $z_1,z_2,z_3,z_4\in \Zcal,$ $a\in \Acal$ and $c\in \Ccal$ we
  have 
    \begin{enumerate}
     \item $ a(z_1\triangleleft z_2) = (az_1)\triangleleft z_2 $ and $
       (z_1\triangleright z_2)c = z_1\triangleright (z_2c). $ 
     \item $(z_1\triangleleft z_2)^* = z_2\triangleleft z_1$ and
       $(z_1\triangleright z_2)^* = z_2\triangleright z_1.$ 
     \item $ ((z_1\triangleleft z_2)z_3)\triangleright z_4 = (z_1(z_2
       \triangleright z_3))\triangleright z_4.$ 
    \end{enumerate}
\end{lemma}
\begin{proof}
  By linearity it suffices to consider elementary tensors $z_i$. 
  Assume $z_i =x_i\otimes y_i,$ for $i=1,2,3,4.$
  Then
  $$ a(z_1\triangleleft z_2)
    = a\laa x_1 {}_\Bcal\la y_1,y_2\ra,x_2\raa
    = \laa ax_1 {}_\Bcal\la y_1,y_2\ra,x_2\raa
    = (az_1)\triangleleft z_2.$$
%   $$ (z_1\triangleleft z_2)^*
%     = \laa x_1 {}_\Bcal\la y_1,y_2\ra,x_2\raa^*
%     = \laa x_2{}_\Bcal\la y_2,y_1\ra,x_1  \raa
%     = z_2\triangleleft z_1.$$
  The second identity in (1) and the two of (2) are left to the reader.
  Besides, (3) follows from
  \begin{align*}
    ((z_1\triangleleft z_2)z_3)\triangleright z_4
      & = (\laa x_1 {}_\Bcal\la y_1,y_2\ra,x_2\raa x_3\otimes y_3)\triangleright z_4\\
      & = \la y_3 ,  \lab \laa x_1 {}_\Bcal\la y_1,y_2\ra,x_2\raa x_3  ,x_4\rab y_4\rac \\
      & = \la y_3 ,  \lab x_1 {}_\Bcal\la y_1,y_2\ra \lab x_2,x_3\rab,x_4 \rab y_4\rac \\
      & = \la {}_\Bcal\la y_1,y_2\ra \lab x_2,x_3\rab y_3 ,  \lab x_1 ,x_4 \rab y_4\rac \\
      & = \la y_1\la y_2, \lab x_2,x_3\rab y_3\rac ,  \lab x_1 ,x_4 \rab y_4\rac \\
      & = (z_1 \la y_2, \lab x_2,x_3\rab y_3\rac) \triangleright z_4 \\
      & = (z_1 (z_2\triangleright z_3)) \triangleright z_4 \\
  \end{align*}
\end{proof}

To define the pre-inner products and actions on $\Ucal$ take $u\in
U_r,$  $v\in U_s,$ $a\in A_t$ and $c\in C_t.$ 
Note that $G\times G\to C_{\rmu s}:\ (p,q)\mapsto u(p,\pmu
r)\triangleleft v(q,\qmu s),$ and $G\times G\to A_{r\smu}:\
(p,q)\mapsto u(p,\pmu r)\triangleright v(q,\qmu s)$ are continuous
maps. 
Then we can define
\begin{align}
\plaua u,v \praua & :=\iint_{G\times G} u(p,\pmu r)\triangleleft
v(q,\qmu s)\, dpdq; \\ 
\plauc u,v \prauc &:=\iint_{G\times G} u(p,\pmu r)\triangleright
v(q,\qmu s)\, dpdq; \\ 
au\in U_{tr} &\mbox{ by the formula } (au)(p,\pmu tr):= au(\tmu p,\pmu
tr)\mbox{ and }\\ 
uc\in U_{rt} & \mbox{ by the formula } (uc)(p,\pmu r t):= u(p,\pmu
r)c
\end{align}

\begin{remark}\label{remark good operations}
  Some straightforward arguments together with Lemma \ref{lemma some
    identities} imply $\plaua \ ,\ \praua$ behaves like a left
  pre-inner product, that is: it is quasi-linear in the first
  variable, 
  $\plaua a u ,v \praua = a\plaua  u ,v \praua$ and $\plaua  u ,v
  \praua^* = \plaua v ,u \praua.$  
  Also $\plauc \ ,\ \prauc$ behaves like a right pre-inner product
  with respect to $\Ccal.$  
\end{remark}

\begin{remark}\label{remark ilt continuity of pre inner product}
  For every compact set $K\subset G$ and $u,v\in \Ucal$ supported in 
  $K\times K,$ we have $\| \plauc u,v\prauc \|\leq M^2 \|
  u \|_\infty \|v\|_\infty,$ where $M$ is the measure of $K.$  
%   In particular this implies the map $\plauc\ ,\ \prauc  \colon
%   U_r\times U_s\to C_{\rmu s}$ is continuous in the inductive limit
%   topology. 
\end{remark}

\begin{lemma}\label{lemma the key}
 For all $u,v,w,x\in \Ucal$ and $a\in \Acal$ we have
 \begin{align}
  \plauc \plaua u,v \praua w ,x \prauc &= \plauc u\plauc v,w \prauc ,x \prauc \\ \label{equ quaternary identity}
  0 & \leq \plauc u,u\prauc\\
  \plauc au,au\prauc & \leq \|a\|^2\plauc u,u\prauc \label{equ left action bounded}
 \end{align}
\end{lemma}
\begin{proof}
  To prove the first identity assume $u\in U_r,\ v\in U_s,\ w\in U_t$ and $x\in U_q.$
  Then, by Lemma \ref{lemma some identities},
\begin{align*}
   \plauc \plaua u, v\praua w,x\prauc
     & = \int_{G^2} \left[ \plaua u, v\praua w \right](p_1,{p_1}^{-1}r\smu t)\triangleright x(p_2,{p_2}^{-1}q)\, d(p_1,p_2)\\
     & = \int_{G^2} \left[ \plaua u, v\praua w(s\rmu p_1,{p_1}^{-1}r\smu t) \right]\triangleright x(p_2,{p_2}^{-1}q)\, d(p_1,p_2)\\
     & = \int_{G^4} \left[  (u(p_3,{p_3}^{-1}r)\triangleleft v(p_4,{p_4}^{-1}s))w(s\rmu p_1,{p_1}^{-1}r\smu t) \right]\triangleright \\
     &         \qquad \qquad \qquad \qquad \qquad \qquad\qquad \qquad  x(p_2,{p_2}^{-1}q)\, d(p_1,p_2,p_3,p_4)\\
     & = \int_{G^4} \left[  u(p_3,{p_3}^{-1}r) (v(p_4,{p_4}^{-1}s)\triangleright w(s\rmu p_1,{p_1}^{-1}r\smu t)) \right]\triangleright \\
     &         \qquad \qquad \qquad \qquad \qquad \qquad\qquad \qquad  x(p_2,{p_2}^{-1}q)\, d(p_1,p_2,p_3,p_4)\\
     & = \int_{G^4} \left[  u(p_3,{p_3}^{-1}r) (v(p_4,{p_4}^{-1}s)\triangleright w(p_1,{p_1}^{-1}t)) \right]\triangleright \\
     &         \qquad \qquad \qquad \qquad \qquad \qquad\qquad \qquad  x(p_2,{p_2}^{-1}q)\, d(p_1,p_2,p_3,p_4)\\
     & = \int_{G^2} \left[  u(p_3,{p_3}^{-1}r) \la v,w\ra^\Ucal_\Ccal
     \right]\triangleright  x(p_2,{p_2}^{-1}q)\, d(p_2,p_3) 
         \\
     & = \int_{G^2} \left[  u\la v,w\ra^\Ucal_\Ccal
       \right](p_3,{p_3}^{-1}r\smu t)\triangleright  x(p_2,{p_2}^{-1}q)\,
       d(p_2,p_3)   
         \\
     & = \plauc u\plauc v,w \prauc ,x \prauc  
  \end{align*}
  
  From Remarks \ref{remark ilt continuity of pre inner product} and \ref{lemma density in Ut} we conclude that it suffices to show \eqref{equ quaternary identity} and \eqref{equ left action bounded} hold for  $u=\xi|_r,$ with $\xi\in \Gamma_\Zcal.$
  Assume $\xi = \sum_{i=1}^n f_i\boxtimes g_i.$
  Then $\plauc u,u\prauc$ and $\plauc au,au\prauc$ are respectively
  the integrals in $G\times G$ of the functions $\eta,\theta\colon
  G\times G\to C_e$ given by 
  $$ \eta(p,q) = \sum_{i,j=1}^n \la g_i (\pmu r),\lab f_i(p),f_j(q)\rab
  g_j (\qmu r)\rac $$ 
  $$ \theta(p,q) = \sum_{i,j=1}^n \la g_i (\pmu r),\lab
  af_i(\tmu p),af_j(\tmu q)\rab g_j(\qmu r)\rac .$$
  
  Take a compact set $K\subset G$ such that
  $\supp(\eta)\cup\supp(\theta)\subset K\times K.$ 
  Given a compact neighborhood $W$ of $e\in G,$ take
  $p^W_1,\ldots,p^W_{n_W}\in G$ such that $K$ is contained in the interior
  of $p^W_1 W\cup\cdots\cup p^W_{n_W} W$. 
  Now let $\{\phi^W_1,\ldots,\phi_{n_W}^W\}\subset C_c(G)^+$ be a
  partition of the unit of $K$ subordinated to the covering $\{p^W_1
  W,\ldots,p^W_{n_W} W\}$.  
  Define
  \begin{align*}
   \eta_W(p,q)&:=\sum_{j,k=1}^{n_W}\eta(p^W_j,p^W_k)\phi_j^W(p)\phi_k^W(q),\\
  \theta_W(p,q)&:=\sum_{j,k=1}^{n_W}\theta(p^W_j,p^W_k)\phi_j^W(p)\phi_k^W(q).
  \end{align*}
  We order the family $\Ncal$ of compact neighborhoods of $e$ by
  decreasing inclusion.  
  Then we have nets $\{\eta_W\}_{W\in \Ncal}$ and $\{\theta_W\}_{W\in
    \Ncal}$ that can be shown to converge to $\eta$ and $\theta$, 
  respectively, in the inductive limit topology.
  The inequalities $0\leq \plauc u,u\prauc$ and $\plauc au,au\prauc \leq \|a\|^2\plauc u,u\prauc$ follow by taking limit in $W$ after we show that
  \begin{equation}\label{equation THE INEQUALITY I}
       0\leq \int_G\int_G \eta_W(p,q)\, dpdq.
  \end{equation}  
  \begin{equation}\label{equation THE INEQUALITY II}
       \int_G\int_G \theta_W(p,q)\, dpdq\leq \|a\|^2\int_G\int_G \eta_W(p,q)\, dpdq.
  \end{equation}

  Consider $W$ fixed and put $m:=n_\lambda,$ $\lambda_j:=\int_G \phi^W_j(p)\, dp$ and $p_k:=p^W_k.$
  Then
  \begin{align*}
     \int_G\int_G \theta_W&(p,q)\, dpdq
       = \sum_{k,l=1}^m \theta(p_k,p_l)\lambda_k\lambda_l\\
      & = \sum_{k,l=1}^m \sum_{i,j=1}^n  \la \lambda_k g_i({p_k}^{-1} r),\lab af_i(\tmu p_k),af_j(\tmu p_l)\rab \lambda_l g_j({p_l}^{-1} r)\rac
  \end{align*}
  The key is to interpret the latter sum as an inner product, what we
  do next.  
  
  Let $\Mb_{\pb}(\Bcal)$ be the $C^*$-algebra  provided by 
  Lemma~\ref{lemma approximate units} for $\pb=(p_1,\ldots,p_m)\in
  G^m$, and let 
  $\Xb_{\pb}'=X_{p_1}\oplus\cdots \oplus X_{p_m}$, where the direct
  sum is as left Hilbert $A_e\dsh$modules. The left
  Hilbert $A_e$-module $\Xb_{\pb}'$ can be 
  given an $A_e\dsh \Mb_{\pb}(\Bcal)$ Hilbert bimodule structure in
  the following way.  Write the elements of $\Xb_\pb'$ as row matrices
  and define the right action by matrix multiplication; the right
  inner product is defined to be   
  $$\la (x_1,\ldots,x_m),(y_1,\ldots,y_m)\ra_{\Mb_{\pb}(\Bcal)}=(\lab
  x_i,y_j\rab)_{i,j=1}^m .$$ 
  (the positivity of this inner product is shown in the same way as
  done in the proof of Lemma~\ref{lemma approximate units} for the
  inner product ${}_{\Mb_{\rb^{-1}}(\Acal)}\la \,,\,\ra$). 
  
  Now let $\Yb_{\pb^{-1}t}$ be the direct sum $Y_{{p_1}^{-1}t}\oplus
  \cdots\oplus Y_{{p_m}^{-1}t},$ considered as a right Hilbert
  $C_e\dsh$module. 
  Writing the elements of $\Yb_{\pb^{-1}t}$ as column matrices,
  the matrix multiplication by elements of $\Mb_{\pb}(\Bcal)$ defines
  a *-homomorphism of $\Mb_{\pb}(\Bcal)$ into
  $\Bb(\Yb_{\pb^{-1}t})$.

  Now, if we define
  $$\fb_i :=  (f_i(\tmu p_1),\ldots,f_i(\tmu p_m))  \in \Xb_{\tmu \pb}' ,$$
  $$\gb_j:= (\lambda_1g_j({p_1}^{-1}r),\ldots,\lambda_mg_j({p_m}^{-1}r))  \in 
  \Yb_{\pb^{-1}r},$$
  then we have
  $$ \int_G\int_G \theta_W(p,q)\, dpdq = \sum_{i,j=1}^n \la \gb_i,\la a\fb_i,a\fb_j\ra_{\Mb_{\pb}(\Bcal)}\gb_j\ra_{C_e}. $$
  
  We shall interpret the latter double sum as an inner product.
  Consider $\Xb_{\tmu\pb}^n$ as a $\Mb_n(A_e)\dsh
  \Mb_n(\Mb_\pb(\Bcal))$ Hilbert bimodule in the usual way.
  Considering $\Yb_{\pb^{-1}r}^n$ as a Hilbert $C_e\dsh$module, and 
  thinking of its elements as column matrices, matrix multiplication
  provides us with a representation $\Mb_n(\Mb_\pb(\Bcal))\to
  \Bb(\Yb_{\pb^{-1}r}^n).$  

  If $\fb=(\fb_1,\ldots,\fb_n),$ $\gb=(\gb_1,\ldots,\gb_n)$ and $D_{a^*a}\in \Mb_n(A_e)$ is the diagonal matrix with value $a^*a$ in the diagonal, then
  $$\int_G\int_G \theta_W(p,q)\, dpdq = \la \gb,\la D_{a^*a} \fb, \fb\ra_{\Mb_n(\Mb_\pb(\Bcal))}\gb\ra_{C_e} \leq \|a\|^2 \la \gb,\la \fb,\fb\ra_{\Mb_n(\Mb_\pb(\Bcal))}\gb\ra_{C_e}.$$
  
  Using the interpretation of the double integral of $\eta_W$ as an inner product we conclude that
  $$0\leq  \la \gb,\la \fb,\fb\ra_{\Mb_n(\Mb_\pb(\Bcal))}\gb\ra_{C_e} = \int_G\int_G \eta_W(p,q)\, dpdq. $$
  Putting the last two inequalities together we get \eqref{equation THE
    INEQUALITY I} and \eqref{equation THE INEQUALITY II}. 
\end{proof}

\begin{lemma}\label{lemma continuity of operations tensor product bundle}
  For every $\xi,\eta \in \Gamma_\Zcal$ the following maps are continuous:
  \begin{enumerate}
   \item $ G\times G\to \Ccal, \ (p,q)\mapsto \plauc \xi|_p,\eta|_q\prauc. $
   \item $G\times G\to C_e,\ (p,q)\mapsto  \plauc g(p)\xi|_q - \eta|_{pq},g(p)\xi|_q - \eta|_{pq}\prauc  .$
%    \item $ G\times G\to A_e, \ (p,q)\mapsto \laa \plaua \xi|_p,\eta|_q\praua - f(p\qmu)\praua^2 .$
%    \item $G\times G\to A_e,\ (p,q)\mapsto \plaua \xi|_qf(q)-\eta(pq)   \praua^2.$
  \end{enumerate}
\end{lemma}
\begin{proof}
 Let $\theta$ be the map in (1).
 It suffices to consider the case $\xi = f\boxtimes g$ and $\eta = h\boxtimes k.$
 Then 
 $$ \theta(p,q) = \iint_{G\times G} \lac g(\rmu p),\lab f(r),h(s)\rab k(\smu q)\rac \, drds. $$
 Fix $(p_0,q_0)\in G\times G$ and take a compact neighborhood $V$ of $e\in G$.  
 We show $\theta$ is continuous in $W:=p_0V\times q_0 V.$
 
 Let $\Vcal$ be the retraction of $\Ccal$ by $W\to G, \ (p,q)\mapsto
 \pmu q,$ and define $\eta\colon G\times G\to C(\Vcal)$ as
 $\eta(r,s)(p,q)=\lac g(\rmu p),\lab f(r),h(s)\rab k(\smu q)\rac.$ 
 Note that $\eta$ has compact support and $C(\Vcal)$ is a Banach space
 with the supremum norm because $W$ is compact. 
 Moreover, if $\{(r_i,s_i,p_i,q_i)\}_{i\in I}\subset G\times G\times
 W$ is a net converging to $(r,s,p,q),$ then
 $\eta(r_i,s_i)(p_i,q_i)\to \eta(r,s)(p,q).$ 
 This implies $\eta$ is continuous.
 Then $\theta$ is continuous because $\theta = \iint_{G\times G}
 \eta(r,s)\, drds.$ 
 
 We just give an indication of how to prove the map defined in (2) is
 continuous. 
 The trick here is to think of that map as the integral of a
 continuous map from $G\times G$ to the space of continuous sections
 of the trivial Banach bundle over $W$ with constant fiber $C_e.$  
\end{proof}

\begin{remark}\label{remark mirror image}
  The last two Lemmas and their proofs can be carried out with $\plauc \ ,\ \prauc$ replaced by $\plaua\ ,\ \praua$ and the actions of $\Acal$ and $\Ccal$ on $\Ucal$ interchanged.
\end{remark}
%\medskip
\paragraph{{\bf The bundle $\mathbf{[\Ucal]}$}}\mbox{{\quad}}
\par
Now we enter to the final phase of our construction of a tensor
product bundle. 
Each fiber $U_t$ is a pre Hilbert $C_e\dsh$module with the inner product $\plauc\ ,\ \prauc,$ so $\|u\|_\Ccal:= \|\plauc u,u\prauc\|^{1/2}$ defines a seminorm on $U_t.$
We also know that ${}_\Acal\|u\|:=\|\plaua u,u\praua\|^{1/2}$ is a seminorm on $U_t.$
If in \eqref{equ left action bounded} we put $a=\plaua u,u\praua,$
then use Remark \ref{remark good operations}, the relations \eqref{equ quaternary identity}-\eqref{equ left action bounded} and, finally, take
norms in $\Ccal$, we get 
$$  {\|u\|_\Ccal}^6=\|\plauc u,u\prauc^3 \|\leq {}_\Acal\|u\|^4\|\plauc u,u\prauc\| = {}_\Acal\|u\|^4{\|u\|_\Ccal}^2. $$
Then $\|u\|_\Ccal\leq {}_\Acal\|u\|$ and, by symmetry, it must be $\|u\|_\Ccal={}_\Acal\|u\|.$

Let $U^0_t:=\{u\in U_t\colon \|u\|_\Ccal=0\}$ and $[U]_t:=U_t/U^0_t.$
We denote $u\mapsto [u]$ the quotient map of all fibers.
Then form a bundle $[\Ucal]:=\{[U]_t\}_{t\in G}$ and consider the set of sections
\begin{equation}\label{equ definition of sections of tensor product}
  \Gamma_{[\Ucal]}:=\{ [\xi]\colon \xi\in \Gamma_\Zcal \}, \mbox{ where } [\xi](t):=[\xi|_t],\ \forall \ \xi\in  \Gamma_\Zcal,\ t\in G.
\end{equation}

The action of $\Ccal$ on the left is
$$ [\Ucal]\times \Ccal\to [\Ucal],\ ([u],c)\mapsto [uc],$$
which is defined because for all $u\in  \Ucal$ and $c\in C:$
$$ {\|uc\|_\Ccal}^2 = \| c^*\plaua u,u\praua c\|\leq \|c\|^2{\|u\|_\Ccal}^2.$$

The $\Ccal\dsh$valued inner product is 
$$[\Ucal]\times [\Ucal]\to \Ccal,\ ([u],[v])\mapsto \plauc u,v\prauc.$$
To show this operation is defined note that for $u\in U_r$ and $v\in U_s$ we have $v,u \plauc u,v\prauc\in U_s.$
So it follows that
$$ \|\plauc u,v\prauc\|^2 = \| \plauc u\plauc u,v\prauc,v \prauc \|\leq \| u \|_\Ccal \|\plauc u,v\prauc\| \|v\|_\Ccal, $$
which in turn implies $\|\plauc u,v\prauc\|\leq \|u\|_\Ccal\|v\|_\Ccal.$

The left hand side operations are
$$ \Acal\times [\Ucal] \to [\Ucal],\ (a,[u])\mapsto [au]\quad \mbox{ and }\quad   [\Ucal]\times [\Ucal]\to \Acal,\ ([u],[v])\mapsto \plaua u,v\praua .$$

Now we use Proposition \ref{theorem construction of equivalence bundle} to construct an $\Acal\dsh \Ccal\dsh$equivalence bundle from $[\Ucal].$
We already have the operations and inner products. Take $\Gamma_\Acal=C_c(\Acal),$ $\Gamma_\Ccal=C_c(\Ccal)$ as sets of sections, and $\Gamma_{[\Ucal]}$ as we have defined in \eqref{equ definition of sections of tensor product} above.

Conditions (1R-5R) and (1L-5L) follow by construction, Remark \ref{remark good operations} and symmetry.
Now we show (7R), and by symmetry we will have (7L).
Take $x_1,x_2\in X_r,$ $y_1,y_2\in Y_s$ and $\varepsilon >0.$
It suffices to find $u,v\in \Ucal$ such that 
$$\| \lac y_1,\lab x_1,x_2\rab y_2\rac - \plauc u,v\prauc  \|<\varepsilon.$$

Choose $f_i\in C_c(\Xcal)$ and $g_i\in C_c(\Ycal)$ such that $f_i(r)=x_1$ and $g_i(s)=y_i$ ($i=1,2$).
Now choose a compact neighborhood of $e\in G,$ $W,$ and $\phi_W\in C_c(G)^+$ with $\int_G \phi_W(t)\phi_W(\smu \tmu s)\, dt=1.$
Let $\xi^W_i\in \Gamma_\Zcal$ be defined as 
$$\xi^W_i(p,q):=  (\phi_W(\rmu p )f_i(p) )\boxtimes (\phi_W(\smu q)g_i(q)).$$
Then
\begin{align*}
  \plauc \xi^W_1|_{rs}, \xi^W_2|_{rs}\prauc
  & = \iint_{G\times G}  \phi_W(\rmu p)\phi_W(\smu \pmu rs)\phi_W(\rmu q)\phi_W(\smu \qmu rs)\\
  & \qquad \qquad \qquad \qquad  \lac g_1(\pmu rs),\lab f_1(p),f_2(p)\rab g_2(\qmu rs)\rac \, dpdq.
\end{align*}

The function inside the integral is zero outside $rW\times sW.$
With $W$ small enough we can arrange the expression in the bottom of the equation (without $dpdq$) to be at most at distance $\varepsilon /2$ from $c:=\lac g_1(s),\lab f_1(r),f_2(r)\rab g_2(s)\rac=\lac y_1,\lab x_1,x_2\rab y_2\rac,$ for all $(p,q)\in W.$
Using the identity
$$  \iint_{G\times G}  \phi_W(\rmu p)\phi_W(\smu \pmu rs)\phi_W(\rmu q)\phi_W(\smu \qmu rs)c\, dpdq  =c ,$$
it follows that $ \|\plauc \xi^W_1|_{rs}, \xi^W_2|_{rs}\prauc -c\|<\varepsilon. $
\par Once we have verified (1R-5R), (7R), (1L-5L), and (7L), we deal
with the compatibility of the left and right operations. We have  
$ \plaua [u],[u]\praua [w]=[u]\plauc [v],[w]\prauc $
because, if $x:= \plaua [u],[u]\praua [w]-[u]\plauc [v],[w]\prauc,$
 then Lemma~\ref{lemma the key} implies
$$ \plauc x,x\prauc =  \plauc \plaua [u],[u]\praua [w],x\prauc - \plauc [u]\plauc [v],[w]\prauc,x\prauc = 0.$$
Thus $x=0.$
\par Note that hypothesis (2) of Proposition~\ref{theorem construction
  of equivalence bundle} is immediate in the present situation. 
Besides, hypothesis (3) follows immediately from Remark~\ref{remark
  ilt continuity of pre inner product} and Lemma~\ref{lemma density in
  Ut}. 
Finally (4) follows from Lemma \ref{lemma continuity of operations
  tensor product bundle} and symmetry. 

\begin{definition}\label{definition tensor product bundle}
  The internal tensor product of the equivalence bundles $\Xcal$ and
  $\Ycal$ is the equivalence bundle given by
  Proposition~\ref{theorem construction of equivalence bundle} for
  $[\Ucal].$ 
  This tensor product bundle is denoted $\Xcal\otimes_\Bcal\Ycal$.
\end{definition}

The existence of $\Xcal\otimes_\Bcal\Ycal$ proves the transitivity of
the relation of equivalence of Fell bundles. So we get:

\begin{theorem}\label{theorem Morita equiv is an equiv relation}
  Equivalence of Fell bundles is an equivalence relation.
\end{theorem}
\begin{proof} From Example \ref{example reflexive} and Remark
  \ref{remark symmetric} 
we know that equivalence of Fell bundles is reflexive and symmetric. 
It is also transitive because of the above construction: if $\Xcal$ is
an $\Acal \dsh \Bcal \dsh$equivalence bundle and $\Ycal$ is a $\Bcal \dsh \Ccal$
equivalence bundle, then $\Xcal\otimes_\Bcal\Ycal$ is an $\Acal \dsh \Ccal$
equivalence bundle. 
\end{proof}

\begin{corollary}
 If $\Xcal$ is an $\Acal\dsh \Bcal\dsh$equivalence bundle, then $\Acal,$ $\Bcal,$ $\Kb(\Xcal)$ and $\Lb(\Xcal)$ are equivalent Fell bundles.
\end{corollary}
\begin{proof}
  Since equivalence of Fell bundles is an equivalence relation, it
  suffices to note that $\Kb(\Xcal)$ is isomorphic to $\Acal,$
  $\Acal\oplus \Xcal$ is an $\Acal\dsh \Lb(\Xcal)\dsh$equivalence
  bundle and that $\Xcal\oplus \Bcal$ is a $\Lb(\Xcal)\dsh
  \Bcal\dsh$equivalence bundle. 
\end{proof}
\begin{corollary}\label{cor:partial-saturated}
Every Fell bundle associated to a partial action is equivalent to a
Fell bundle associated to a global action, thus to a
saturated Fell bundle. 
\end{corollary}
\begin{proof}
By \cite[Theorem~6.1]{Ab03}, every partial action $\alpha$ has a
Morita enveloping action $\beta$, that is, the partial action $\alpha$
is equivalent to a partial action $\alpha'$ that has an enveloping
action $\beta$.
The equivalence between $\alpha$ and $\alpha'$
provides a $\Bcal_\alpha \dsh \Bcal_{\alpha'}\dsh$equivalence bundle $\Xcal$
(see examples \ref{subsection Morita equivalence of partial actions} and 
\ref{example morita equivalence of pa 2}), and
the enveloping action $\beta$ of $\alpha'$ provides a
$\Bcal_{\alpha'}\dsh \Bcal_{\beta}\dsh$equivalence bundle $\Ycal$  
(see examples \ref{subsection enveloping actions} and  
\ref{subsection enveloping actions 2}). Therefore
$\Xcal\otimes_{\Bcal_{\alpha'}}\Ycal$ is a 
$\Bcal_{\alpha}\dsh \Bcal_{\beta}\dsh$equivalence bundle. 
\end{proof}
\subsection{Tensor products and cross-sectional Hilbert bimodules}

\begin{theorem}\label{theorem tensor products of bundles and modules}
  Assume $\Xcal$ is an $\Acal\dsh \Bcal\dsh$equivalence bundle and
  $\Ycal$ a $\Bcal\dsh \Ccal\dsh$equivalence bundle. 
  Let $\Xcal\otimes_\Bcal\Ycal$ be the equivalence bundle of
  Definition \ref{definition tensor product bundle} (see also the
  construction in Section \ref{subsection a tensor product of
    bundles}). 
  Then there exists a unique unitary 
  $$U\colon C^*(\Xcal)\otimes_{C^*(\Bcal)}C^*(\Ycal)\to
  C^*(\Xcal\otimes_\Bcal\Ycal)$$ 
  such that
  $U(f\otimes g) = [f\boxtimes g],$ for all $f\in C_c(\Xcal)$ and
  $g\in C_c(\Ycal).$ 
  \begin{comment}
  Moreover, the map $\Kb(C^*(\Xcal)\otimes_{C^*(\Bcal)}C^*(\Ycal))\to
  \Kb(C^*(\Xcal\otimes_\Bcal\Ycal)),\ T\mapsto U\circ T\circ U^*,$ is
  the identity as a map from $C^*(\Acal)$ to $C^*(\Acal).$  
  \end{comment}
\end{theorem}
\begin{proof}
  To prove the existence of the linear isometry $U$ it is enough to show
  that, for all $f,f'\in C_c(\Xcal)$ and $g,g'\in C_c(\Ycal)$, we have: 
  $$  \la f\otimes g, f'\otimes g'\ra_{C^*(\Ccal)} = \la
  [f\boxtimes g],[f'\boxtimes g']\ra_{C^*(\Ccal)}, $$ 
  where the inner product in the left member of the equality above
  corresponds to $C^*(\Xcal)\otimes_{C^*(\Bcal)}C^*(\Ycal)$, while the
  inner product in the right member is that of
  $C^*(\Xcal\otimes_\Bcal\Ycal)$. 
  On the one hand we have, for $r\in G\times G:$
  \begin{align*}
   \la f\otimes g, f'\otimes g'\ra_{C^*(\Ccal)}(r)
    & = \iiint_{G^3}  \lac g(s),\lab f(p),f'(psrt)\rab g'(\tmu)\rac\, dpdtds.  
  \end{align*}
  On the other hand
  \begin{align*}
   \la [f\boxtimes g],[f'\boxtimes g']\ra_{C^*(\Ccal)}(r)
    & = \iiint_{G^3}  \lac g(\pmu s),\lab f(p),f'(t)\rab g'(\tmu
    sr)\rac\, dpdtds.   
  \end{align*}
  These triple integrals agree because the second one is obtained form
  the first one with the following substitutions (consecutively):
  $s=\pmu s'$ and $t'=s'rt.$ 
  \par A procedure analogous to the preceding one allows us to see that $U$
  also preserves the left inner product.  
  \par Let us show that $U$ is surjective by proving that 
  $$S:=\spn \{[f\boxtimes g]\colon f\in C_c(\Xcal),\ g\in C_c(\Ycal)\}$$
  is dense in the inductive limit topology in
  $C_c(\Xcal\otimes_\Bcal\Ycal)$ (see Remark \ref{remark:topologies on C_c(X)}).
  Let $\overline{S}$ be the closure of $S$ in the inductive limit
  topology. 
  We already know that $\{u(t)\colon u\in \overline{S}\}$ is dense in
  $(\Xcal\otimes_\Bcal \Ycal)_t,$ for all $t\in G$ (Lemma \ref{lemma
    density in Ut}). 
  Then from \cite{FlDr88} we conclude it suffices to show that
  $C_c(G)\overline{S}\subset \overline{S}$ or, equivalently, that $\phi
  [f\boxtimes g]\in \overline{S}$ for all $\phi\in C_c(G),$ $f\in
  C_c(\Xcal)$ and $g\in C_c(\Ycal).$ 
  
  Choose compact sets $K_1,K_2\subset G$ such that $K_1$ is contained
  in the interior of $K_2$ and $K_1$ contains the supports of $f$ and
  $g$ in its interior. 
  Then take $\psi\in C_c(G)$ such that $\psi f=f,$ $\psi g=g$ and $\supp \psi \subset K_1.$
  The function $\Phi:G\times G\to \C,\ (p,q)\mapsto \psi
  (p)\psi(q)\phi(pq),$ is continuous, has compact support and vanishes
  outside $K_1\times K_1.$ 
  Then for every $\varepsilon >0$ there exist $\varphi^{h,\varepsilon}_j\in C_c(G)$ ($h=f,g$ and $j=1,\ldots,n_\varepsilon$) such that $\|\Phi - \sum_{j=1}^{n_\varepsilon} \varphi^{f,\varepsilon}_j\otimes \varphi^{g,\varepsilon}_j  \|_\infty <\varepsilon.$
  Moreover, we may assume $\supp (\varphi^{h,\varepsilon}_j)\subset K_2,$ for all $h,j,\varepsilon.$ 
  Now set $\xi^\varepsilon :=\sum_{j=1}^{n_\varphi} (\varphi^{f,\varepsilon}_jf)\boxtimes (\varphi^{g,\varepsilon}_jg).$
  
  Note that $[\xi^\varepsilon]\in S$ and  $\supp(\xi^\varepsilon)\subset K_2\times K_2,$ for all $\varepsilon>0.$
  Also $\supp (\phi[f\boxtimes g])\subset K_2\times K_2.$
  Besides, if $M_2$ is the measure of $K_2,$ then Remark \ref{remark ilt continuity of pre inner product} implies that for all $t\in G$ we have
  $$ \| \phi(t)[f\boxtimes g](t)  -[\xi^\varepsilon](t) \|_{\Xcal\otimes_\Bcal\Ycal} \leq M_2 \|\phi(t) f\boxtimes g|_{H_t} -   \xi^\varepsilon|_{H_t}\|_\infty. $$
  For all $r\in G$ we have
  $$\phi(t) f\boxtimes g|_{H_t}(r,\rmu t)= \phi(t) f(r)\boxtimes g(\rmu t) = \Phi(r,\rmu t)f(r)\otimes g(\rmu t).$$
  Then
  \begin{align*}
    \| \phi(t) f(r)\boxtimes g(\rmu t) -   \xi^\varepsilon(r,\rmu t)\|
      &\leq  \|\Phi - \sum_{j=1}^{n_\varepsilon} \varphi^{f,\varepsilon}_j\otimes \varphi^{g,\varepsilon}_j  \|_\infty \|f\|_\infty\|g\|_\infty\\
      &\leq \varepsilon \|f\|_\infty\|g\|_\infty.
  \end{align*}
  Putting all together we conclude that $\| \phi [f\boxtimes g] - [\xi^\varepsilon] \|_\infty\leq M_2\varepsilon \|f\|_\infty\|g\|_\infty.$
  Thus we have that $\phi [f\boxtimes g] \in \overline{S}$.
 
\begin{comment}  
  Finally, to prove the last claim concerning the conjugation by $U,$
  it suffices to prove that $U$ preserves the left
  $C^*(\Acal)\dsh$valued inner products, but this holds by symmetry. 
\end{comment}
\end{proof}

\begin{corollary}
 In the hypotheses of Theorem \ref{theorem tensor products of bundles
   and modules} and for every pseudo crossed product $\mu$ with the hereditary subbundle property, there exists a unique unitary 
 $$U_\mu\colon C^*_\mu(\Xcal)\otimes_{C^*_\mu(\Bcal)}C^*_\mu(\Ycal)\to
 C^*_\mu(\Xcal\otimes_\Bcal\Ycal)$$ 
  such that
  $U_\mu(q^\Xcal_\mu(f)\otimes q^\Ycal_\mu(g)) =
  q^{\Xcal\otimes_\Bcal\Ycal}_\mu[f\boxtimes g],$ for all $f\in
  C_c(\Xcal)$ and $g\in C_c(\Ycal).$ 
\end{corollary}
\begin{proof}
 Uniqueness is clear, we deal with existence.
 The submodule of $C^*(\Xcal)\otimes_{C^*(\Bcal)}C^*(\Ycal)$
 corresponding to $I^\Ccal_\mu$ is
 $I:=C^*(\Xcal)\otimes_{C^*(\Bcal)}C^*(\Ycal) I^\Ccal_\mu$ 
 and, by Proposition \ref{proposition induction of ideals of crossed
   products}, 
 $$ U(I) = C^*(\Xcal\otimes_\Bcal\Ycal)I^\Ccal_\mu =
 I^{\Xcal\otimes_\Bcal\Ycal}_\mu.$$ 
 Then there exists a unique unitary
 $$V\colon  C^*(\Xcal)\otimes_{C^*(\Bcal)}C^*(\Ycal)/I \to
 C^*_\mu(\Xcal\otimes_\Bcal\Ycal), \ \xi+I\mapsto
 q^{\Xcal\otimes_\Bcal\Ycal}_\mu\circ U(\xi),$$ 
 where we have identified $C^*_\mu(\Xcal\otimes_\Bcal\Ycal)$ with the
 quotient of $C^*(\Xcal\otimes_\Bcal\Ycal)$ by
 $I^{\Xcal\otimes_\Bcal\Ycal}_\mu.$ 
 
 To prove the existence of a unitary 
 $$W\colon  C^*(\Xcal)\otimes_{C^*(\Bcal)}C^*(\Ycal)/I  \to
 C^*_\mu(\Xcal)\otimes_{C^*_\mu(\Bcal)}C^*_\mu(\Ycal) $$ 
 such that $W(f\otimes g+I) = q^\Xcal_\mu(f)\otimes q^\Ycal_\mu(g),$
 it suffices to prove that 
 $$\|\sum_{i=1}^n q^\Xcal_\mu(f_i)\otimes q^\Ycal_\mu(g_i)\|=\| \sum_{i=1}^n f_i\otimes g_i + I\|$$
 for all $f_1,\ldots,f_n\in C_c(\Xcal),$ $g_1,\ldots,g_n\in C_c(\Ycal)$ and $n\in \N.$
 But, thinking of $C^*_\mu(\Zcal)$ as $C^*(\Zcal)/I^\Zcal_\mu$ for $\Zcal = \Xcal, \Ycal,\Ccal,$ we obtain
 \begin{align*}
  \|\sum_{i=1}^n q^\Xcal_\mu(f_i)\otimes q^\Ycal_\mu(g_i)\|^2
    & = \|\sum_{i,j=1}^n\la q^\Ycal_\mu(g_i),\la q^\Xcal_\mu(f_i),q^\Xcal_\mu(f_j)\ra_{C^*_\mu(\Bcal)}q^\Ycal_\mu(g_j)\ra_{C^*_\mu(\Ccal)}\|\\
    & = \|\sum_{i,j=1}^n\la q^\Ycal_\mu(g_i),(\la f_i,f_j\ra_{C^*(\Bcal)} + I^\Bcal_\mu)q^\Ycal_\mu(g_j)\ra_{C^*_\mu(\Ccal)}\|\\
    & = \|\sum_{i,j=1}^n\la q^\Ycal_\mu(g_i),q^\Ycal_\mu(\la f_i,f_j\ra_{C^*(\Bcal)} g_j)\ra_{C^*_\mu(\Ccal)}\|\\
    & = \|\sum_{i,j=1}^n\la g_i,\la f_i,f_j\ra_{C^*(\Bcal)} g_j\ra_{C^*(\Ccal)} + I^\Ccal_\mu \|\\
    & = \|\sum_{i=1}^n f_i\otimes g_i + I\|^2.
 \end{align*}

 Then the unitary $U_\mu$ we are looking for is $V\circ W^*$ because, for all $f\in C_c(\Xcal)$ and $g\in C_c(\Ycal),$ 
 $$ V\circ W^*( q^\Xcal_\mu(f)\otimes q^\Ycal_\mu(g) ) =V( f\otimes g
 + I ) = q^{\Xcal\otimes_\Bcal\Ycal}_\mu\circ U(f\otimes g) =
 q^{\Xcal\otimes_\Bcal\Ycal}_\mu [f\boxtimes g].$$ 
\end{proof}

\bibliographystyle{amsplain}
\bibliography{equivalence_Fell_bundles_bibliography}

\end{document}